\documentclass[11pt,reqno]{amsart}
\usepackage[left=0.9in,top=0.9in,right=0.9in,bottom=0.9in]{geometry}
\usepackage{amsaddr}

\usepackage{amsfonts,amsmath,amsthm,amssymb,latexsym,mathrsfs,stmaryrd}
\usepackage{hyperref}
\usepackage{mathtools}
\usepackage{graphicx}
\usepackage{subcaption}
\usepackage{cleveref}
\usepackage{soul}

\usepackage{tikz}\usetikzlibrary{cd,fit,matrix,arrows,decorations.pathmorphing, shapes.geometric, calc}
\tikzset{commutative diagrams/.cd}
\usepackage{quiver}

\tikzset{commutative diagrams/.cd}

\numberwithin{equation}{section}

\newtheorem{theorem}{Theorem}[section]
\newtheorem{corollary}[theorem]{Corollary}
\newtheorem{lemma}[theorem]{Lemma}
\newtheorem{proposition}[theorem]{Proposition}
\newtheorem{conjecture}[theorem]{Conjecture}

\theoremstyle{definition}
\newtheorem{definition}[theorem]{Definition}
\newtheorem{definition-theorem}[theorem]{Definition-Theorem}
\newtheorem{example}[theorem]{Example}

\newtheorem{notation}[theorem]{Notation}

\newtheorem{question}[theorem]{Question}

\newtheorem{problem}[theorem]{Problem}
\newtheorem{remark}[theorem]{Remark}

\theoremstyle{remark}
\newtheorem*{warning}{Warning}
\newtheorem*{remark*}{Remark}

\usepackage{enumitem}
\setenumerate[1]{label=\textup{(\alph*)}}
\setenumerate[2]{label=\textup{(\roman*)}}

\newcommand\Z{{\mathbb Z}}
\newcommand\N{{\mathbb N}}
\newcommand\Q{{\mathbb Q}}
\newcommand\C{{\mathbb C}}
\newcommand\A{{\mathbb A}}
\newcommand\R{{\mathbb R}}
\newcommand\F{{\mathbb F}}
\def\P{{\mathbb P}}
\newcommand\Fp{{\mathbb F}_p}
\newcommand\Fq{{\mathbb F}_q}

\newcommand\Zp{{{\mathbb Z}_p}}

\def\O{\mathcal O}

\newcommand{\Bot}{\operatorname{Bot}_{\alpha}}

\newcommand{\Spec}{\operatorname{Spec}}
\newcommand{\gr}{\operatorname{gr}}

\newcommand{\cH}{\mathscr H}

\newcommand{\tcP}{\widetilde{\mathcal P}}
\newcommand{\tr}{{t_{r}}}
\newcommand{\tc}{{t_{c}}}
\newcommand{\qtil}{\mathsf Q}

\DeclareMathOperator{\im}{im}

\DeclareMathOperator{\Hom}{Hom}
\DeclareMathOperator{\End}{End}
\DeclareMathOperator{\Aut}{Aut}

\DeclareMathOperator{\GL}{GL}
\DeclareMathOperator{\Gr}{Gr}
\DeclareMathOperator{\Ext}{Ext}

\DeclareMathOperator{\Mat}{Mat}
\newcommand\resp{{\textit{resp. }}}
\newcommand\subeq{\subseteq}
\newcommand\cotimes{\widehat{\otimes}}

\newcommand\supeq{\supseteq}

\newcommand{\diag}{\mathrm{diag}}

\newcommand\bbar{\overline} 
\newcommand\tl{\widetilde} 
\newcommand\hhat{\widehat} 
\newcommand\onto{\twoheadrightarrow}
\newcommand\incl{\hookrightarrow}
\newcommand{\map}[1][]{{\xrightarrow{#1}}} 

\DeclarePairedDelimiter{\abs}{\lvert}{\rvert}
\DeclarePairedDelimiter{\norm}{\lVert}{\rVert}
\DeclarePairedDelimiter{\set}{\{}{\}}
\DeclarePairedDelimiter{\pairing}{\langle}{\rangle}
\DeclarePairedDelimiter{\parens}{\lparen}{\rparen}

\DeclarePairedDelimiter\floor{\lfloor}{\rfloor}

\newcommand{\defn}{\textbf}

\newcommand{\alert}[1]{{{\color{red}{#1}}}} 

\def\funceq{Theorem~1.5}

\newcommand{\qbinom}[2]{{#1\brack #2}}

\newcommand{\cW}{\mathcal{W}}
\renewcommand{\k}{\Bbbk}
\renewcommand{\L}{\mathbb{L}}

\newcommand{\KVar}[1]{K_0(\mathrm{Var}_{#1})}

\newcommand{\Quot}{\mathrm{Quot}}

\newcommand{\Hilb}{\mathrm{Hilb}}

\newcommand{\init}{\mathit{in}}

\newcommand{\bmat}[1]{\begin{bmatrix}#1\end{bmatrix}}

\usepackage{accents}

\newcommand{\Fl}{\mathrm{Fl}}
\newcommand{\calZ}{\mathcal{Z}}

\newcommand{\calV}{\mathcal{V}}

\newcommand{\bG}{\mathbb{G}}
\newcommand{\Gm}{{\bG_m}}
\newcommand{\dinv}{\mathtt{dinv}}

\newcommand{\BB}{Bia{\l}ynicki-Birula~}

\DeclareMathOperator*{\Prob}{Prob}
\newcommand{\E}{\mathbb{E}}

\usepackage{leftindex}

\begin{document}
\title{Quot scheme of points on torus knot singularities}

\author{Yifeng Huang}
\address{Dept.\ of Mathematics, Emory University}
\email{yifeng.huang@emory.edu}

\author{Ruofan Jiang}
\address{Dept.\ of Mathematics, University of California, Berkeley}
\email{ruofanjiang@berkeley.edu}

\author{Alexei Oblomkov}
\address{Dept.\ of Mathematics and Statistics, University of Massachusetts, Amherst}
\email{oblomkov@umass.edu}

\date{\today}

\begin{abstract}
    For $\gcd(a,b)=1$, we show that the moduli space of $m$-codimensional $\k[\![T^a,T^b]\!]$-submodules of $\k[\![T]\!]^n$ is paved by affine cells, by proving that each Bia\l ynicki-Birula stratum of a closed related moduli space with respect to the natural $\bG_m$-action is an affine bundle over the fixed point locus and that the fixed point locus is an iterated Grassmannian bundle. As an application, we determine the motive of this moduli space in the Grothendieck ring of varieties in terms of an explicit two-variable series $N_{a,b;n}(q,t)$, and use it to explicit compute the groupoid volume of the category of finite modules over $\Fq[\![T^a,T^b]\!]$.

    The series $N_{a,b;n}$ carries the conjectures we then formulate. At $n=\infty$ we conjecture a bi-infinite family of Rogers--Ramanujan type identities by specializing the $t$-variable; we identify their product side with the normalized character of a module over the $\cW$-algebra minimal model $\cW_a(a,a+b)$, and observe a connetion to colored Jones tails. At $n<\infty$ we conjecture that $N_{a,b;n}$ is computed by the bottom $\alpha$-row of the trigraded $S^n$-colored HOMFLY homology of the torus knot $T(a,b)$, and that this same bottom row also computes the Quot schemes of finite codimensional $\k[\![T^a,T^b]\!]$-submoudles of $\k[\![T^a,T^b]\!]^n$ and the punctual Hilbert schemes of the non-reduced curve $(Y^a-X^b)^n=0$; the three quantities are special values at three points of the trigrading, and when $n=1$ they recover both the conjectures of Oblomkov--Rasmussen--Shende and of Kivinen--Trinh.
    Finally we conjecture that the one direction of the trigrading these three points do not see is a perverse filtration on the moduli spaces themselves, and we verify its prediction for a smooth germ at $n=2$ by computing the decomposition theorem for the $\GL_2$ spectral-curve family.
\end{abstract}

\maketitle
\setcounter{tocdepth}{1}
 \tableofcontents

\section{Introduction}
In this paper, for positive integers $\gcd(a,b)=1$ and an arbitrary field $\k$, we consider the germ of the $(a,b)$-torus knot singularity
\[ R=\frac{\k[\![X,Y]\!]}{(Y^a-X^b)}\simeq \k[\![T^a,T^b]\!]\]
and study the (reduced structure of the) \defn{punctual Quot scheme} of a finitely generated $R$-module $M$, defined as
\[ \Quot_m^R(M) = \Quot_m(M) = \{M'\subeq_R M: \dim_\k M/M'=m\}.\]
It is projective (though in general not smooth) and can be realized as a closed subset of a Grassmannian (see, e.g., \cite{huangjiang2023torsionfree}). We focus on the \emph{high-rank regime}, namely, where 
\[M=E^n \text{ with } E=\k[\![T]\!] \text{ or } R.\]
The Quot scheme $\Quot_m(R^n)$ is a literal high-rank generalization of the \defn{punctual Hilbert scheme} $\Hilb_m(R)=\Quot_m(R^1)$, while $\{\Quot_m^R(\k[\![T]\!]^n)\}_{m\geq 0}$ is in spirit a high-rank generalization of the compactified Jacobian. The rank 1 story is already a rich subject with many connections to fields including but not limited to knot theory, geometric representation theory, algebraic combinatorics, and algebraic statistics \cite{bushnellreiner1980zeta,gorskymazin2013compactified1,gorskymazin2014compactified2,maulik2016stable,maulikyun2013macdonald,oblomkovshende2012,OblomkovRasmussenShende12,solomon1977zeta,yun2013orbital}. The case of torus knot singularities is the simplest examples to illustrate the connection: the Poincar\'e polynomials of $\Hilb_m(R)$ and $\Quot_m^R(\k[\![T]\!])$ can be packaged by the $a\times b$ rational $q,t$-Catalan numbers \cite{gorskymazin2013compactified1,lusztigsmelt1989}, which is essentially the HOMFLY-PT polynomial of the $(a,b)$-torus knot. However, as the first and the second authors pointed out in \cite{huangjiang2023torsionfree}, explicit understanding of the Quot scheme faces extreme challenges in the high-rank case. Here are two obvious things to do and how they fail: 1. The torus action of $\bG_m^n$ on $\Quot_m^R(E^n)$ induced from its action on $E^n=E\otimes_\k \k^n$ is hard to extract invariants finer than Euler characteristic due to the singularity of the Quot scheme; 2. The method of $p$-adic integration, which is good at proving rationality and functional equation for the point-count aspect of $\Quot_m^R(E^n)$ over finite fields $\k=\Fq$, is not fruitful so far in explicit computation. It is made worse by the fact that the integration recipe takes as input the set of isomorphism classes of torsion-free $R$-modules, but this set is inaccessible because the module theory over $\k[\![T^a,T^b]\!]$ is \emph{wild} for most $(a,b)$. The state of the art in the high-rank computation is the case $a=2$, solved in \cite{huangjiang2023torsionfree}, but the method there essentially exploits the ``finite CM type'' of the $A_n$ singularities, which is impossible to be generalized to arbitrary $(a,b)$.

As our main structural result, we prove the following. In this paper, a \defn{cell decomposition} or an \defn{affine paving} refers to a decomposition of a $\k$-variety into finitely many locally closed subspaces, each of which is isomorphic to an affine space over $\k$; to avoid distraction, we do not attempt to prove whether the cell decomposition is filtered (i.e., the closure of each cell is a union of cells).

\begin{theorem}\label{thm:paving}
    Let $\gcd(a,b)=1$ and $R=\k[\![T^a,T^b]\!]$. Then for all $m,n\geq 0$, $\Quot_m^R(\k[\![T]\!]^n)$ admits a cell decomposition.
\end{theorem}
An explicit description of the cell decomposition will be given in \S\ref{sub:singaffinegrass}.

The motivic generating function is the natural framework to keep track of the enumerative content of a cell decomposition. Let $\KVar{\k}$ be the Grothendieck ring of $\k$-varieties, generated by classes $[X]$ for $\k$-varieties $X$ as an abelian group, with scissor relation $[X]=[Z]+[X\setminus Z]$ for closed subscheme $Z\subeq X$ and multiplication $[X]\cdot [Y]=[X\times_\k Y]$.
Define the Lefschetz motive $\L=[\A^1]$. As a result, if a variety $X$ has an affine paving with $a_i$ cells of dimension $i$ for each $i\geq 0$, then we have $[X]=\sum_{i\geq 0} a_i \L^i$. 

\begin{theorem}\label{thm:quot-zeta-tilde}
    Let $\gcd(a,b)=1$ and $R=\k[\![T^a,T^b]\!]$. For all $n\geq 0$, define the generating function
    \[ \calZ^R_{\k[\![T]\!]^n}(t)\coloneqq \sum_{m\geq 0} [\Quot_m^R(\k[\![T]\!]^n)]\, t^m \in \KVar{\k}[\![t]\!].\]
    Then
    \[ \calZ^R_{\k[\![T]\!]^n}(t) = \frac{1}{(1-t)(1-\L t)\cdots (1-\L^{n-1}t)} \L^{n^2\delta} \, t^{n\delta} N_{a,b;n}(\L^{-1},t^{-1}), \]
    where $\delta=(a-1)(b-1)/2=\dim_\k \tl R/R$ and $N_{a,b;n}(q,t)$ is an explicit polynomial defined in \eqref{eq:q,t-sum-finite}.
\end{theorem}

The case of $\Quot_m(R^n)$ appears to be considerably harder; as of now, we do not know whether $\Quot_m(R^n)$ has polynomial point counts over finite fields (except when $a=2$). However, a functional equation \cite[\funceq]{huangjiang2023torsionfree} allows us access to the $t=q^{-n}$ specialization of the point-count generating function. For each fixed prime power $q$, define
\begin{equation}\label{eq:quot-free-mod}
    Z_{a,b;n}(q,t)\coloneqq \sum_{m\geq 0} \#\Quot_m(\Fq[\![T^a,T^b]\!]^n)(\Fq)\, t^m\in \Z[\![t]\!];
\end{equation}
for ease of reference, the set in question has a direct definition below:
\[ \Quot_m(\Fq[\![T^a,T^b]\!]^n)(\Fq) = \{M'\subeq_{\Fq[\![T^a,T^b]\!]} \Fq[\![T^a,T^b]\!]^n: \dim_{\Fq} \Fq[\![T^a,T^b]\!]^n/M'=m\}. \]

\begin{theorem}\label{thm:quot-zeta-special}
    Let $\gcd(a,b)=1$ and $q$ be a prime power.
    Then for all $n\geq 0$,
    \[ Z_{a,b;n}(q,q^{-n}) = \frac{1}{(1-q^{-1})\cdots (1-q^{-n})} N_{a,b;n}(q^{-1},1). \]
\end{theorem}

The significance of the $t=q^{-n}$ specialization lies in that its $n\to \infty$ limit has additional interpretations. Consider the following two quantities in $[0,\infty]$ defined by countable sums of nonnegative real numbers:
\[ S_1 \coloneqq \sum_{M\in \mathbf{FinMod}_R/{\sim}} \frac{1}{\abs{\Aut_R(M)}},\]
the ``groupoid volume'' of the category of finite-cardinality modules over $R=\Fq[\![T^a,T^b]\!]$, and
\[ S_2 \coloneqq \sum_{n=0}^\infty \frac{\abs*{\set*{(A,B):A,B\in \mathrm{Nilp}_n(\Fq), AB=BA, A^a=B^b}}}{\abs{\mathrm{GL}_n(\Fq)}},\]
a weighted number of ``nilpotent matrix points'' of $X^a=Y^b$. Then we have the following formula.
\begin{theorem}\label{thm:groupoid-vol}
    Let $\gcd(a,b)=1$ and $q$ be a prime power. Then
    \[ S_1=S_2=\frac{1}{(1-q^{-1})(1-q^{-2})\cdots} N_{a,b;\infty}(q^{-1},1),\]
    where $N_{a,b;\infty}$ is an explicit sum defined in \eqref{eq:q,t-sum-infinite}.
\end{theorem}
\begin{remark}\label{rmk:pos-def}
    By the positive definiteness theorem \cite{huang2026dinv}, $N_{a,b;\infty}(q^{-1},1)<\infty$, so that $S_1=S_2<\infty$. 
\end{remark}
\begin{remark}\label{rmk:coh-zeta}
    In a broader context, $S_1=S_2$ is the $t=1$ special value of the \defn{Coh zeta function}
    \[ Z_{R;\infty}(t)\coloneqq \sum_M \frac{1}{\abs{\Aut_R M}}\, t^{\dim_{\Fq} M} \in \Q[\![t]\!],\]
    where $M$ ranges over isomorphism classes of finite-cardinality $R$-modules, for $R=\Fq[\![T^a,T^b]\!]$. Theorem~\ref{thm:groupoid-vol} and Remark~\ref{rmk:pos-def} imply that $Z_{R;\infty}(t)$ has radius of convergence at least $1$, and thus so does the normalized version
    \[ N_{R;\infty}(t) \coloneqq Z_{R;\infty}(t)/Z_{\tl R;\infty}(t)=(1-q^{-1} t)\cdots (1-q^{-n}t) Z_{R;\infty}(t).\]
    In \cite{huang2023mutually}, the first author asked if $N_{R;\infty}(t)$ always has infinite radius of convergence, at least if $R$ is the germ of a plane curve singularity over $\Fq$. The weaker problem of deciding if the radius is at least $1$ already appears deep even in the special case of torus knot singularities: though a cancellation-free summation formula of $Z_{R,\infty}(1)$ in terms of $N_{a,b;\infty}$ is available, the convergence at $t=1$ is still an analytical miracle (Remark~\ref{rmk:pos-def}) that requires a nontrivial positive definiteness theorem \cite{huang2026dinv}, not mentioning the combinatorial simplification (Lemma~\ref{lem:dinv-simplification}) that was crucially used to derive the current formula of $Z_{R,\infty}(1)$ in the first place. 
\end{remark}

A remarkable aspect of the series $N_{a,b;\infty}(q,t)$ is its observed factorizability in the $t=1$ specialization. For $a,b\in \Z_{\geq 1}$, not necessarily coprime, define the ``charge function'' $r_{a,b}:\Z\to \Z$ by
\[ r_{a,b}(i)\coloneqq \min\{a,b,\mathrm{dist}(ai,(a+b)\Z)\}+1_{(a+b)\Z}(i)-1,\]
where $\mathrm{dist}(ai,(a+b)\Z)$ is the distance from $ai$ to the closest multiple of $a+b$. See Remark~\ref{rmk:charge} for a list of elementary properties of $r_{a,b}$ and an alternative way to view it. Importantly, $r_{a,b}$ is $a+b$-periodic: $r_{a,b}(i)=r_{a,b}(i+a+b)$. 

Define
\[ P_{a,b}(q)\coloneqq \prod_{i=1}^\infty (1-q^n)^{-r_{a,b}(i)}.\]
Suppose that $a<b$, we will see that $P_{a,b}(q)=\overline{\chi}^{a,a,a+b}_{(j_0,j_1,...,j_{a-1})}(q)$, the right hand side being the normalized character of a rather special representation of a vertex operator algebra $\cW_a(a,a+b)$; cf. Section~\ref{subsec:product}. 


\begin{conjecture}\label{conj:rr-type}
For $\gcd(a,b)=1$, we have
\[ N_{a,b;\infty}(q,1)=P_{a,b}(q).\]
\end{conjecture}

This conjecture appears to be very deep and defies all proof attempts by connecting to known theories. The case $a=2,b=2m+1$ corresponds to the celebrated Andrews--Gordon identity \cite[Corollary 7.8]{andrewspartitions}
\[ \sum_{n_1\geq \dots\geq n_m\geq 0} \frac{q^{n_1^2+\dots+n_m^2}}{(q;q)_{n_1-n_2}\cdots (q;q)_{n_{m-1}-n_m}(q;q)_{n_m}} = \frac{(q^{m+1},q^{m+2},q^{2m+3};q^{2m+3})_\infty}{(q;q)_\infty},\] 
an $\mathrm{A}_1$-Rogers--Ramanujan identity. In recent series of works by the first author, Kenny Lau, Ken Ono, and Peter Paule \cite{hlop2026,lauono2026}, it is first discovered and then proven that the $a=3$ case can be reduced to certain $\mathrm{A}_2$-Rogers--Ramanujan identities due to Warnaar \cite{warnaar2021ag}. For $a<b$ in general, Conjecture~\ref{conj:rr-type} should give a new framework of $\mathrm{A}_{a-1}$-Rogers--Ramanujan identities; that it lives in arbitrary rank is partly what makes it exceptionally exciting. 

This is only part of the story. Further experiments suggest that ``linear-term'' modifications of $N_{a,b;\infty}(q,1)$ also yield infinite products of modulus $a+b$. More precisely, from \eqref{eq:q,t-sum-infinite} we have
    \[
    N_{a,b;\infty}(q,1)=\sum_{\mathbf{n}\in \Z^G} \qbinom{\infty}{\mathbf{n}}_{q;G} \, q^{\dinv(\mathbf{n})},\]
    where $\mathbf{n}$ is a vector indexed by the gap set $G=\N\setminus\langle a,b\rangle$ and $\dinv$ is a quadratic form. Enumerate the gaps as $g_1<\dots<g_\delta$, where $\delta=\abs{G}=(a-1)(b-1)/2$. For $0\leq r\leq \delta$, define
    \[ N_{a,b}^{(r)}(q)=\sum_{\mathbf{n}\in \Z^G} \qbinom{\infty}{\mathbf{n}}_{q;G} \, q^{\dinv(\mathbf{n})+n_{g_1}+\dots+n_{g_r}}.\]
    Note that the extreme cases are $N_{a,b}^{(0)}(q)=N_{a,b;\infty}(q,1)$, the subject of Conjecture~\ref{conj:rr-type}, and $N_{a,b}^{(\delta)}(q)=N_{a,b;\infty}(q,q)$, which appears in later discussions; see Remark~\ref{sub:Jonestail} and \eqref{eq:diagonal}. 
    The following conjecture generalizes Conjecture~\ref{conj:rr-type}.
    \begin{conjecture}\label{conj:RR_general}
For coprime $a<b$, and $0\leq r\leq (a-1)(b-1)/2$, we have          \[N_{a,b}^{(r)}(q)=P^{(r)}_{a,b}(q):=\prod_{i=1}^\infty (1-q^i)^{-p_{a,b}^{(r)}(i)}\]
    for some nonnegative $(a+b)$-periodic function $p_{a,b}^{(r)}$ which makes $P^{(r)}_{a,b}(q)$ a normalized $\cW_a(a,a+b)$-character\footnote{We have an algorithm to explictly compute $p_{a,b}^{(r)}(i)$ and the corresponding character, which is too complicated to be presented here.}. 
    \end{conjecture}
    
    For $a=2$, this covers the full Andrews--Gordon family
    \[ \sum_{n_1\geq \dots\geq n_m\geq 0} \frac{q^{n_1^2+\dots+n_m^2+n_{m-r+1}+\dots+n_m}}{(q;q)_{n_1-n_2}\cdots (q;q)_{n_{m-1}-n_m}(q;q)_{n_m}} = \frac{(q^{m+1-r},q^{m+2+r},q^{2m+3};q^{2m+3})_\infty}{(q;q)_\infty}.\]
    See Example~\ref{ex:twoclassicalRRs} for the case where $(a,b)=(2,3)$.
\begin{remark}[Connection to colored Jones polynomials]\label{sub:Jonestail}

It turns out that $P^{(\delta)}_{a,b}(q)$ equals the normalized 
character $\overline{\chi}^{a,a,a+b}_{(b+1,1,1,...,1)}(q)$ associated to the $\cW_a(a,a+b)$-representation with highest weight $b\Lambda_0$. It turns out that this character appears in the limit of colored Jones polynomials: from the computation in  \cite[Theorem 6.1, Remark 6.4(4)]{Kanade_old}, $\overline{\chi}^{a,a,a+b}_{(b+1,1,1,\dots,1)}(q)$ is essentially the colored Jones tail of the toric knot $T(a,a+b)$ colored by the symmetric product $\mathfrak{sl}_a$-representation with highest weight $j\Lambda_1$, for any $j\geq 1$. Therefore, $N_{a,b;\infty}(q,q)$, which arises from the algebraic geometry of $T(a,b)$, computes the colored Jones tail of $T(a,a+b)$. We expect that this is related to a more general mechanism linking 
$N_{a,b;n}$ to colored knot homology.
\end{remark}

\subsection{Speculations on a trivariate refinement}\label{sec:speculations}

Let $R$ be a planar curve germ with normalization $\tl R$. For $E=R$ or $\tl R$, recall the rank-$n$ Quot zeta function
\[
  \calZ^R_{E^n}(t)=\sum_{m\geq0}[\Quot^R_m(E^n)]t^m
\]
and its normalization
\[
  \mathcal N^R_{E^n}(t)=\frac{\calZ^R_{E^n}(t)}{\calZ^{\tl R}_{\tl R^n}(t)}.
\]
For $R=\k[\![T^a,T^b]\!]$ we compute $\mathcal N^R_{\tl R^n}(t)$ completely, but our results determine only one specialization of $\mathcal N^R_{R^n}(t)$. Computations for $(2,b)$ torus-link singularities in \cite{huang2025inert,chernhuang2025} indicate that neither normalized zeta function determines the other. They appear instead as two specializations of one trivariate polynomial. This polynomial also accounts for the functional equation and the rank-to-infinity stabilization that are visible on the $R^n$ side but not on the $\tl R^n$ side.

\begin{conjecture}[Master polynomial]\label{conj:three-var}
Let $R$ be a plane curve germ over $\k$. Suppose that, for every $n\geq1$ and $E=R,\tl R$, there is a polynomial $N_{E^n}(q,t)\in\Z[q,t]$ such that
\[
  \mathcal N^R_{E^n}(t)=N_{E^n}(\L,t).
\]
We expect this hypothesis to hold for every plane curve germ. Put $\delta=\dim_\k(\tl R/R)$. Then there is a polynomial $H_{R,n}(q,t,u)\in\Z[q,t,u]$ with the following properties.
\begin{enumerate}
  \item \emph{Functional equation:}
  \[
    H_{R,n}(q,t,u)=q^{\delta n^2}(tu)^{\delta n}H_{R,n}\bigl(q,q^{-n}u^{-1},q^{-n}t^{-1}\bigr).
  \]

  \item \emph{Interpolation:}
  \begin{align}
    N_{\tl R^n}(q,t)&=q^{\delta n^2}t^{\delta n}H_{R,n}(q^{-1},1,t^{-1})=H_{R,n}(q^{-1},q^nt,q^n),\\
    N_{R^n}(q,t)&=q^{\delta n^2}t^{2\delta n}H_{R,n}(q^{-1},t^{-1},t^{-1})=H_{R,n}(q^{-1},q^nt,q^nt).\label{eq:master-Rn}
  \end{align}

  \item \emph{Limit:} the coefficientwise limit
  \[
    H_{R,\infty}(q,t,u)=\lim_{n\to\infty}H_{R,n}(q,t,u)
  \]
  exists in $\Z[\![q,t,u]\!]$. Moreover, $H_{R,\infty}(q,t,u)$ converges for all $q,t,u\in \C$ with $\abs{q}<1$.

  \item \emph{Rank-one degeneracy:} $H_{R,1}(q,t,u)\in\Z[qt,u]$.

  \item \emph{Rogers--Ramanujan type identity:} $H_{R,\infty}(q,1,1)$ has an infinite-product expansion with periodic exponents.

  \item \emph{Cyclic sieving:} if $r\mid n$ and $\zeta_r$ is a primitive $r$-th root of unity, then
  \[
    H_{R,n}(\zeta_r,t,u)=H_{R,1}(1,t^r,u^r)^{n/r}.
  \]
  In particular,
  \[
    H_{R,n}(1,t,u)=H_{R,1}(1,t,u)^n.
  \]
\end{enumerate}
\end{conjecture}

The rescaled numerator
\[
  N_{R,n}(q,t)\coloneqq N_{R^n}(q,q^{-n}t)
\]
is recovered by
\[
  N_{R,n}(q,t)=H_{R,n}(q^{-1},t,t).
\]
Consequently,
\[
  N_{R,\infty}(q,t)=H_{R,\infty}(q^{-1},t,t),
\]
and the left-hand side is known to converge to the normalized Coh zeta function; see Remark~\ref{rmk:coh-zeta}. The functional equation above specializes to the functional equation of $N_{R^n}(q,t)$ from \cite{huangjiang2023torsionfree}; see also Theorem~\ref{thm:pointcount-func-eq}. In rank one, the degeneracy gives $N_R(q,t)=N_{\tl R}(qt,t)$, which is the Hilb-vs-Quot conjecture of Kivinen--Trinh \cite{kivinentrinh2023}.

Even without an intrinsic construction, the master polynomial may be of independent interest. It provides a framework for trivariate deformations of the finitized Rogers--Ramanujan type polynomials $H_{R,n}(q,1,1)$ that carry additional symmetry. For a torus-knot singularity it is also a candidate for a high-rank rational $q,t,u$-Catalan polynomial. In rank one, the third variable is degenerate, so it reduces to the usual rational $q,t$-Catalan polynomial. In higher rank, a direct combinatorial construction remains open.

We next propose topological and geometric constructions over $\C$. Write
\[
  R=\C[\![X,Y]\!]/(f),\qquad R_f^{(n)}=\C[\![X,Y]\!]/(f^n),
\]
and let $b(R)$ be the number of branches of $R$. Define the motivic Hilbert series and its normalized numerator by
\begin{align*}
  Z^{\Hilb}_{f;n}(t)&\coloneqq\sum_{m\geq0}[\Hilb_m(R_f^{(n)})]t^m,\\
  \cH_{f;n}(t)&\coloneqq(t;\L t)_n^{b(R)}Z^{\Hilb}_{f;n}(t).
\end{align*}
Whenever there is a polynomial $\cH_{f;n}(q,t)\in\Z[q,t]$ satisfying $\cH_{f;n}(t)=\cH_{f;n}(\L,t)$, we use this notation for a polynomial representative.
\begin{remark}
    In \cite{huangjiang2023torsionfree}, the first two authors prove that $\mathcal{N}^R_{\tl R^n}(t)$ and $\mathcal{N}^R_{R^n}(t)$ are polynomials in $\KVar{\k}[t]$. We conjecture that $\cH_{f;n}(t)\in \KVar{\k}[t]$ as well.
\end{remark}

Let
\[
  \mathcal N^{\mathrm p}_{\tl R^n}(q,t,u),\quad \cH^{\mathrm p}_{f;n}(q,t,u)
\]
be the normalized virtual perverse weight series defined in \eqref{eq:perv-normalized-quot}--\eqref{eq:perv-normalized}. Here $q$ records weight, $t$ records colength, and $u$ records perverse degree. Since the perverse filtration is exhaustive, setting $u=1$ gives the ordinary virtual weight realization, without any purity assumption. Thus, whenever the corresponding motivic numerators admit polynomial representatives,
\[
  N_{\tl R^n}(q,t)=\mathcal N^{\mathrm p}_{\tl R^n}(q,t,1),\qquad \cH_{f;n}(q,t)=\cH^{\mathrm p}_{f;n}(q,t,1).
\]
Purity and evenness would identify the underlying virtual weight series with ordinary Poincar\'e series. A compatible affine paving would further explain the predicted motivic polynomiality. Neither assertion by itself implies coefficientwise positivity due to normalization.

If $R$ is unibranched, let $K$ be its algebraic knot. Write $\tcP^K_{S^n}(\alpha,\qtil,\tr,\tc)$ for the Poincar\'e polynomial of the reduced quadruply graded HOMFLY homology colored by the one-row partition $(n)$, in the tilde normalization of Gorsky--Gukov--Sto\v si\'c \cite{GGS}, and put
\[
  \Bot\,\tcP^K_{S^n}(\qtil,\tr,\tc)\coloneqq[\alpha^{2n\delta}]\tcP^K_{S^n}(\alpha,\qtil,\tr,\tc).
\]

\begin{conjecture}[Three frameworks over $\C$]\label{conj:colorHOMFLYconj}
The following geometric expressions coincide:
\begin{align}
  H_{R,n}(q,t,u)&=q^{\delta n^2}u^{\delta n}\mathcal N^{\mathrm p}_{\tl R^n}\bigl(q^{-1},tu^{-1},q^nu\bigr),\label{eq:master-perv-quot}\\
  &=q^{\delta n^2}u^{\delta n}\cH^{\mathrm p}_{f;n}\bigl(q^{-1}t^{-1}u,tu^{-1},q^nu\bigr).\label{eq:master-perv-hilb}
\end{align}
If $R$ is unibranched, they also coincide with the knot-theoretic expression
\begin{equation}\label{eq:master-knot}
  H_{R,n}(q,t,u)=q^{\delta n^2}t^{\delta n}\Bot\,\tcP^K_{S^n}\bigl(u^{-1/2},u^{1/2}t^{-1/2},q^{-1/2}\bigr).
\end{equation}
The resulting polynomial satisfies Conjecture~\ref{conj:three-var}. Moreover, $\cH_{f;n}(t)$ admits a polynomial representative in $\Z[\L,t]$, and the cohomology entering the two perverse series is pure and concentrated in even degrees.
\end{conjecture}

The two geometric formulas immediately give the geometric predictions:
\begin{enumerate}
  \item[(a)] $\cH^{\mathrm p}_{f;n}(q,t,u)=\mathcal N^{\mathrm p}_{\tl R^n}(qt,t,u)$;
  \item[(b)] $\cH_{f;n}(q,t)=N_{\tl R^n}(qt,t)$;
  \item[(c)] $N_{R^n}(q,t)=t^{n\delta}\cH^{\mathrm p}_{f;n}(q,1,t)=t^{n\delta}\mathcal N^{\mathrm p}_{\tl R^n}(q,1,t)$.
\end{enumerate}
Here (a) is a perverse high-rank Hilb-vs-Quot identity, (b) is its unrefined motivic counterpart and is obtained at the virtual-weight level by setting $u=1$, while (c) follows from \eqref{eq:master-Rn} and would give a new geometric method to compute $N_{R^n}$ if true.

In the unibranched case, the knot-theoretic formula predicts
\begin{align}
  N_{\tl R^n}(q,t)&=t^{n\delta}\Bot\,\tcP^K_{S^n}\bigl(q^{-n/2},t^{-1/2},q^{1/2}\bigr),\label{eq:intro-spec-quot-tilde}\\
  N_{R^n}(q,t)&=t^{n\delta}\Bot\,\tcP^K_{S^n}\bigl(t^{1/2},1,q^{1/2}\bigr),\label{eq:intro-spec-quot-R}\\
  \cH_{f;n}(q,t)&=t^{n\delta}\Bot\,\tcP^K_{S^n}\bigl((qt)^{-n/2},t^{-1/2},(qt)^{1/2}\bigr).\label{eq:intro-spec-hilb}
\end{align}
At $n=1$, the geometric predictions recover the Hilb-vs-Quot conjecture. In the unibranched case, the knot-theoretic prediction recovers the bottom-row form of the ORS conjecture. Detailed consequences and evidence are collected in Section~\ref{sec:further}.

\begin{remark}[Arithmetic setup and the multibranched case]
Conjecture~\ref{conj:colorHOMFLYconj} supplies geometric and topological models over $\C$, but not over a finite field or for an arithmetic order. The computations of \cite{huang2025inert,chernhuang2025} suggest that the normalized arithmetic Quot zeta functions $\nu^R_{R^n}(s)$ and $\nu^R_{\tl R^n}(s)$ are likewise interpolated by a master function with one additional variable, and that this function interacts nontrivially with Galois twists. It would be natural to seek an arithmetic counterpart of Conjecture~\ref{conj:colorHOMFLYconj}.

The works \cite{huang2025inert,chernhuang2025} also suggest that arithmetic considerations may be relevant in the geometry of the multibranched case. Conjecture~\ref{conj:colorHOMFLYconj} does not provide a link-theoretic or combinatorial formula for $H_{R,n}(q,t,u)$ for multibranched singularities. For the $(2,2m)$ torus-link singularity
\[
  R_{2,2m}=\C[\![X,Y]\!]/(Y^2-X^{2m}),
\]
\cite{chernhuang2025} predicts a formula for $H_{R_{2,2m},n}(q,t,u)$:
\[
  H_{R_{2,2m},n}(q,t,u)=\sum_{k_1,\dots,k_m}q^{\sum_{i=1}^m k_i^2}(tu)^{\sum_{i=1}^m k_i}(u^{-1};q^{-1})_{k_1}\qbinom{n}{k_1}_q\qbinom{k_1}{k_2}_q\cdots \qbinom{k_{m-1}}{k_m}_q.
\]
This polynomial can have negative coefficients; already $H_{R_{2,2},1}(q,t,u)=1-qt+qtu$. Thus a na\"ive interpretation of its coefficients as dimensions of cohomology groups cannot hold in the multibranched case. On the other hand, for the quadratic twisted model
\[
  R'=R'_{2,2m}=\F_q[\![X,Y]\!]/\bigl((Y-\alpha X^m)(Y-\beta X^m)\bigr)\simeq\F_q[\![T]\!]+T^m\F_{q^2}[\![T]\!],
\]
where $\alpha\ne\beta\in\F_{q^2}$ are Frobenius conjugate over $\F_q$, \cite{chernhuang2025} predicts that a polynomial $H_{R',n}(q,t,u)$ interpolating $\nu^{R'}_{R'^n}(s)$ and $\nu^{R'}_{\tl R'^n}(s)$ is given by 
\[
  H_{R'_{2,2m},n}(q,t,u)=H_{R_{2,2m},n}(q,-t,-u),
\]
which has nonnegative coefficients. 
\end{remark}

\subsubsection{The knot-theoretic model and the knots--quivers formula}\label{subsec:intro-knot}

For $R=\C[\![T^a,T^b]\!]$ with $\gcd(a,b)=1$, the polynomial computed in this paper is the slice
\[
  N_{a,b;n}(q,t)=H_{R,n}(q,1,t).
\]
Consequently, \eqref{eq:master-knot}, together with GGS self-symmetry, identifies $N_{a,b;n}$ with a two-dimensional slice of the bottom row of the $S^n$-colored homology of $T(a,b)$. The specialization $q=1$ is compatible with refined exponential growth, while the diagonal $q=t$ recovers the bottom row of the symmetric-colored HOMFLY polynomial. These are consequences and consistency checks of Conjectures~\ref{conj:three-var} and \ref{conj:colorHOMFLYconj}, not additional organizing conjectures.

The knots--quivers correspondence of Kucharski--Reineke--Sto\v si\'c--Su\l kowski \cite{KRSS} gives a conjectural finite presentation of the same colored bottom row. We use it only as a computational model for the knot-theoretic term in Conjecture~\ref{conj:colorHOMFLYconj}. For torus knots, the rank-one case of Conjecture~\ref{conj:colorHOMFLYconj} identifies its quiver vertices with the order filters of the gap poset $G$, which also index the torus-fixed points of the poset flag varieties appearing in Theorem~\ref{thm:bb}. Rank one determines the vertex gradings and the diagonal of the quiver matrix, whereas its off-diagonal entries are genuinely higher-rank data, conjecturally encoded by the $R^n$ Quot slice and the perverse refinement. Section~\ref{subsec:Nhat} records the precise dictionary and the independent checks for $T(3,4)$, $T(3,5)$, and $T(3,7)$. For $T(3,4)$, the explicit master-polynomial formula in Remark~\ref{rmk:H-34-KRSS} goes further: its quadratic exponent determines all off-diagonal entries and recovers the KRSS matrix of \cite[Eq.~(5.41)]{KRSS}.

\subsubsection{Three geometric slices}\label{subsec:intro-three}

The rank-one identity $\Hilb_m(R)=\Quot^R_m(R)$ has two natural rank-$n$ extensions: one raises the module rank to $\Quot^R_m(R^n)$, and the other thickens the curve to $\Hilb_m(R_f^{(n)})$. Together with $\Quot^R_m(\tl R^n)$, the object studied in this paper, these are the three geometric specializations displayed after Conjecture~\ref{conj:colorHOMFLYconj}. Their relations are therefore consequences of the master framework. In particular, item~(b) is the higher-rank Hilb-vs-Quot prediction, and item~(c) proposes a new geometric route from either perverse model to the presently inaccessible numerator $N_{R^n}$.

On the knot side, the $R^n$ specialization lies on $\tr=1$, while the $\tl R^n$ and thickened-Hilbert specializations lie on $\qtil\tc^n=1$. The latter two are related by $q\mapsto tq$, or motivically by $\L\mapsto t\L$, including the branch-dependent change from $(t;\L)_n^{b(R)}$ to $(t;\L t)_n^{b(R)}$. The two subtori do not determine the full trigraded bottom row; the missing transverse coordinate is supplied by the perverse grading.

\subsubsection{The perverse grading}\label{subsec:intro-perv}
The proposed construction starts from a family of curves in a smooth surface whose central fiber is the germ in question and whose relevant relative Hilbert or Quot scheme is proper with smooth total space. The perverse Leray filtration refines the two series by a third variable. In the Hilbert-side coordinates, the grading dictionary is
\[
  t=\tr^{-2},\qquad q=(\tr\tc)^2,\qquad u=(\qtil\tc^n)^{-2}.
\]
Thus $u$ is exactly the coordinate transverse to $\qtil\tc^n=1$. Since this change of variables is an isomorphism of three-dimensional tori, either perverse formula in Conjecture~\ref{conj:colorHOMFLYconj} determines the full bottom row.

Section~\ref{sss:perv} defines the refinement using virtual perverse weight polynomials, so that $u=1$ recovers the ordinary virtual weight realization without a purity assumption. Conjecture~\ref{conj:colorHOMFLYconj} separately predicts motivic polynomiality and pure even cohomology; due to normalization, they do not predict coefficientwise positivity of $\mathcal{N}^{\mathrm{p}}_{\tl R^n}$ or $\cH_{f;n}^{\mathrm{p}}$. The grading shifts are constrained by the rank-one ORS description, the functional equation for $\Quot^R_m(R^n)$, and the computed torus-knot examples. Proposition~\ref{prop:ribbon-n2} supplies independent geometric evidence by verifying the predicted perverse concentration for the smooth germ at $n=2$ in the $\GL_2$ spectral-curve family.

\subsection{Geometry of singular affine Grassmannians}\label{sub:singaffinegrass}
As a concluding remark, we present a framework that effectively describe the cell decomposition on $\Quot_m^R(\k[\![T]\!]^n)$, as asserted in  Theorem~\ref{thm:paving}.
It uses the language of ``singular affine Grassmannian'' introduced in \cite[Appendix]{huangjiang2023torsionfree}. We then briefly point out how \Cref{thm:paving} follows from \Cref{thm:bb}. We refer the reader to Appendix~\ref{sec:append} for precise definitions and detailed explanations. 

Let $R=\k[\![T^a,T^b]\!]$ and let $\tl R=\k[\![T]\!]$ be its normalization. While the classical affine Grassmannian $\Gr_{\GL_n}$ classifies $\tl R$-lattices in $\k(\!(T)\!)^n$, the singular affine Grassmannian $\Gr_{\GL_n,R}$ classifies $R$-lattices in $\k(\!(T)\!)^n$. The geometry of cells on $\Quot_m^R(\k[\![T]\!]^n)$ can be elegantly translated to geometry of Schubert cells on  $\Gr_{\GL_n,R}$:

\begin{description}
    \item[(Extension fiber)]  As  pointed out in \cite[Appendix]{huangjiang2023torsionfree}, there is a constructible map $\underline{\vec{\pi}^*}:\Gr_{\GL_n,R}\rightarrow \Gr_{\GL_n}$
sending an $R$-lattice $L$ to $\tl R L$. Upon the decomposition $\Gr_{\GL_n}=\bigsqcup_{\omega\in X_*(\GL_n)} X_\omega^\circ$ into Iwahoric Schubert cells, $\underline{\vec{\pi}^*}$ becomes a trivial fibration with fibers isomorphic to the \textit{extension fiber} $E_R(\tl R^n)$.   
 \item[(Spherical Schubert cell)] In Appendix~\ref{sec:append}, we introduce a notion of (spherical) Schubert cells for $\Gr_{\GL_n,R}$, which naturally generalizes its 
classical counterpart for $\Gr_{\GL_n}$. These Schubert cells are labelled by dimension vectors $\mathbf{n}$ recording the initial term data of $R$-lattices parametrized by $\Gr_{\GL_n,R}$. We show that $E_R(\tl R^n)$ decomposes into a finite union of Schubert cells with dimension vectors $\mathbf{n}\in \mathbb{Z}^P$, where $P$ is a certain rectangular grid poset. 
\item[(Bia\l ynicki-Birula morphism)] Let $\Gr_{\GL_n,R,\mathbf{n}}$ be the Schubert cell with dimension vectors $\mathbf{n}\in \mathbb{Z}^P$. Then there is an affine fibration (Bia\l ynicki-Birula morphism) $\Gr_{\GL_n,R,\mathbf{n}}\rightarrow \Fl_P(\mathbf{
n};n)$, and the target is a poset flag variety, which is an iterated Grassmannian bundle over a point. This encodes Theorems~\ref{thm:bb} and \ref{thm:poset-flag}.
\end{description}
The following diagram summarizes the  structure of $\Gr_{\GL_n,R}$ (where each rectangle encloses a relation between two objects, which is either an isomorphism up to decomposition, or an actual isomorphism, or a  fibration, depending on the context): 

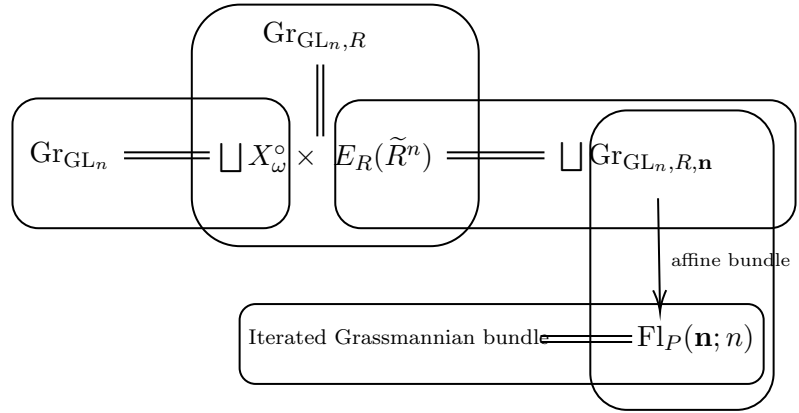
\begin{center}

\tikzset{every picture/.style={line width=0.75pt}} 

\begin{tikzpicture}[x=0.75pt,y=0.75pt,yscale=-1,xscale=1]

\draw    (483,175) -- (483.96,230) ;
\draw [shift={(484,232)}, rotate = 268.99] [color={rgb, 255:red, 0; green, 0; blue, 0 }  ][line width=0.75]    (10.93,-3.29) .. controls (6.95,-1.4) and (3.31,-0.3) .. (0,0) .. controls (3.31,0.3) and (6.95,1.4) .. (10.93,3.29)   ;
\draw    (470,244) -- (424,244) ;
\draw    (470,247) -- (424,247) ;
\draw    (315,108) -- (315,144) ;
\draw    (312,108) -- (312,144) ;
\draw    (258,152) -- (215,152) ;
\draw    (258,155) -- (215,155) ;
\draw    (426,152) -- (377,152) ;
\draw    (426,155) -- (377,155) ;
\draw   (159,138) .. controls (159,130.82) and (164.82,125) .. (172,125) -- (285,125) .. controls (292.18,125) and (298,130.82) .. (298,138) -- (298,177) .. controls (298,184.18) and (292.18,190) .. (285,190) -- (172,190) .. controls (164.82,190) and (159,184.18) .. (159,177) -- cycle ;
\draw   (249,102.2) .. controls (249,88.83) and (259.83,78) .. (273.2,78) -- (368.8,78) .. controls (382.17,78) and (393,88.83) .. (393,102.2) -- (393,174.8) .. controls (393,188.17) and (382.17,199) .. (368.8,199) -- (273.2,199) .. controls (259.83,199) and (249,188.17) .. (249,174.8) -- cycle ;
\draw   (321,138.72) .. controls (321,131.64) and (326.74,125.9) .. (333.82,125.9) -- (539.18,125.9) .. controls (546.26,125.9) and (552,131.64) .. (552,138.72) -- (552,177.18) .. controls (552,184.26) and (546.26,190) .. (539.18,190) -- (333.82,190) .. controls (326.74,190) and (321,184.26) .. (321,177.18) -- cycle ;
\draw   (449,149.4) .. controls (449,139.24) and (457.24,131) .. (467.4,131) -- (522.6,131) .. controls (532.76,131) and (541,139.24) .. (541,149.4) -- (541,262.6) .. controls (541,272.76) and (532.76,281) .. (522.6,281) -- (467.4,281) .. controls (457.24,281) and (449,272.76) .. (449,262.6) -- cycle ;
\draw   (273,236) .. controls (273,231.58) and (276.58,228) .. (281,228) -- (528,228) .. controls (532.42,228) and (536,231.58) .. (536,236) -- (536,260) .. controls (536,264.42) and (532.42,268) .. (528,268) -- (281,268) .. controls (276.58,268) and (273,264.42) .. (273,260) -- cycle ;

\draw (283,85.4) node [anchor=north west][inner sep=0.75pt]    {$\Gr_{\GL_{n} , R}$};
\draw (261,142.7) node [anchor=north west][inner sep=0.75pt]    {$\bigsqcup X_{\omega }^\circ \times \ E_{R}(\widetilde{R}^{n})$};
\draw (432,145.4) node [anchor=north west][inner sep=0.75pt]    {$\bigsqcup \Gr_{\GL_n , R, \mathbf{n}}  $};
\draw (488,201) node [anchor=north west][inner sep=0.75pt]  [font=\tiny] [align=left] {affine bundle};
\draw (471,236.4) node [anchor=north west][inner sep=0.75pt]    {$\Fl_{P}( \mathbf{n};n)$};
\draw (276,239) node [anchor=north west][inner sep=0.75pt]   [align=left] {{\scriptsize Iterated Grassmannian bundle}};
\draw (166,145.4) node [anchor=north west][inner sep=0.75pt]    {$\Gr_{\GL_{n}}$};

\end{tikzpicture}

\end{center}
One deduces the affine paving result in Theorem~\ref{thm:paving}: First, $\Quot^R(\tl R^n)=\bigsqcup_{m\geq 0}\Quot^R_m(\tl R^n)$, up to reduced induced structure, naturally embeds into $\Gr_{\GL_n,R}$. Similarly, $\Quot^{\tl{R}}(\tl R^n)_{\mathrm{red}}\hookrightarrow \Gr_{\GL_n}$. Second, $\Quot^{\tl{R}}(\tl R^n)_{\mathrm{red}}$ is a union of Iwahoric Schubert cells, i.e., $\Quot^{\tl{R}}(\tl R^n)_{\mathrm{red}}=\bigsqcup_{\omega\in V} X_\omega^\circ$ for a certain $V\subeq X_*(\GL_n)$. Third, the decomposition $\Gr_{\GL_n,R}=\bigsqcup_\omega X_{\omega }^\circ \times \ E_{R}(\widetilde{R}^{n})$ restricts to a decomposition $\Quot^R(\tl R^n)_{\mathrm{red}}=\bigsqcup_{\omega\in V} X_{\omega }^\circ \times \ E_{R}(\widetilde{R}^{n})$. An affine paving on $\Quot^R(\tl R^n)$ can then be obtained from one on $E_R(\tl R^n)$.

\subsection*{Acknowledgment} The authors thank fruitful discussions with Eugene Gorsky, Martin Olsson, Ken Ono, Vivek Shende, Sug Woo Shin, Peisheng Yu, Zhiwei Yun, and Peng Zhou.  AO research was partially supported by  NSF grant DMS-2200798. The authors acknowledge that the discovery of further connections and conjectures in Sections~\ref{sec:speculations} and \ref{sec:further} is partly assisted by generative AIs, such as ChatGPT-5.6 Sol.

\section{Definitions and notation}
In this paper, $\N=\{0,1,\dots\}$. We use the $q$-Pochhammer symbol for $n\in \N\cup \{\infty\}$:
\[ (a;q)_n=(1-a)(1-aq)\cdots (1-aq^{n-1})\]
and the $q$-binomial coefficient for $n,k\in \Z$:
\[ \qbinom{n}{k}=\begin{cases}
    \frac{(q;q)_n}{(q;q)_k(q;q)_{n-k}},& 0\leq k\leq n;\\
    0, & n<0,\; k<0, \text{ or }k>n,
\end{cases}\]
which is a polynomial in $\N[q]$. We have $[\Gr(r,n)]=\qbinom{n}{r}_{\L}$. 

For $\gcd(a,b)=1$, we let $G=\N\setminus \pairing{a,b}$ be the \defn{gap poset} of the numerical semigroup $\pairing{a,b}$, equipped with the ordering
\[ i\preceq j \text{ if and only if }j-i\in \pairing{a,b}.\]
The largest element of $G$ (in the usual order of $\N$), known as the \defn{Frobenius element}, is denoted by $f$ and is well-known to be given by $f=\max G=ab-a-b$. For an integer-valued vector $\mathbf{n}=(n_i)_{i\in G}\in \Z^G$ and $n\in \N$, we define the \defn{generalized $q$-multinomial coefficient}
\begin{equation}\label{eq:q-binom-gap}
    \qbinom{n}{\mathbf{n}}_{q;G}\coloneqq \frac{(q;q)_n}{(q;q)_{n-n_f}} \prod_{i\in G} \frac{(q;q)_{n_i-n_{i-a-b}}}{(q;q)_{n_i-n_{i-a}}(q;q)_{n_i-n_{i-b}}},
\end{equation}
where $n_j\coloneqq 0$ for $j<0$ by convention  (cf.~Lemma~\ref{lem:q-binom-gap}). It can be reformulated as a product of $q$-binomial coefficients in a non-unique way (see Lemma~\ref{lem:q-binom-tame}). Define also
\begin{equation*}
    \qbinom{\infty}{\mathbf{n}}_{q;G}\coloneqq  \prod_{i\in G} \frac{(q;q)_{n_i-n_{i-a-b}}}{(q;q)_{n_i-n_{i-a}}(q;q)_{n_i-n_{i-b}}}.
\end{equation*}

On $\Z^G$, define the quadratic form
\begin{equation}\label{eq:dinv}
    \dinv(\mathbf{n})=\sum_{i,j\in G}  K(j-i)\,n_{i}n_{j},
\end{equation}
where
\[K(d)=\mathbf{1}_{d\geq 0}-\mathbf{1}_{d\geq a}-\mathbf{1}_{d\geq b}+\mathbf{1}_{d\geq a+b}.
\]
and the cone
\[ C=\{\mathbf{n}\in \mathbb{Z}^G:n_i\geq 0, n_j\geq n_i\text{ if }j-i\in \langle a,b\rangle\}.\]
For $n\in \N\cup \set{\infty}$, the generalized $q$-multinomial coefficient $\qbinom{n}{\mathbf{n}}_{q;G}$ is nonzero in $\Z[q]$ if and only if
\[ \mathbf{n}\in C \text{ and } n_f \leq n,\] 
again see Section~\ref{subsec:multinomial}. For a vector $\mathbf{n}\in C$, define
\[ \abs{\mathbf{n}}\coloneqq \norm{\mathbf{n}}_1=\sum_{i\in G} n_i.\]

As an important property, which we initially conjectured and which is proved by the first author in \cite{huang2026dinv}, $\dinv$ is positive definite on $C$, namely, there is $\lambda>0$ such that 
\begin{equation}\label{eq:pos-def-eff}
\dinv(\mathbf{n})\geq \lambda\abs{\mathbf{n}}^2
\end{equation}
for all $\mathbf{n}\in C$. Moreover, when evaluated on $0,1$-vectors in $C$, $\dinv$ matches the classical rational $\dinv$ statistic \cite{gorskymazin2013compactified1,huang2026dinv}, justifying our choice of notation. 

We now define the key $q,t$-polynomials and series that appear in our results:
\begin{equation}
    \label{eq:q,t-sum-finite} 
    N_{a,b;n}(q,t)=\sum_{\mathbf{n}\in \Z^G} \qbinom{n}{\mathbf{n}}_{q;G} \, q^{\dinv(\mathbf{n})} \, t^{\abs{\mathbf{n}}},
\end{equation}
and
\begin{equation}
    \label{eq:q,t-sum-infinite}
    N_{a,b;\infty}(q,t)=\sum_{\mathbf{n}\in \Z^G} \qbinom{\infty}{\mathbf{n}}_{q;G} \, q^{\dinv(\mathbf{n})} \, t^{\abs{\mathbf{n}}}.
\end{equation}
For finite $n$, $N_{a,b;n}(q,t)$ is a polynomial because the sum is supported on the finite set $C\cap \{\mathbf{n}: \max_{i\in G} n_i\leq n\}$, and because $\dinv(\mathbf{n})\geq 0$ for $\mathbf{n}\in C$. The infinite sum defining $N_{a,b;\infty}(q,t)$ converges formally to a power series in $\Z[\![q,t]\!]$ because $\qbinom{\infty}{\mathbf{n}}_{q;G} = 1+O(q)$ and $\dinv$ is (strictly) positive definite on $C$; more precisely, the effective bound \eqref{eq:pos-def-eff} ensures that there are at most finitely many summands that contribute to a given degree. Moreover, for the same reason, the specialization $N_{a,b;\infty}(q,1)$ converges formally to a power series in $\Z[\![q]\!]$, and for complex numbers $\abs{q}<1, t\in \C$, the numerical series $N_{a,b;\infty}(q,t)$ converges. 

\begin{remark}\label{rmk:catalan}
    We shall treat $N_{a,b;n}(q,t)$ as a ``rank-$n$'' rational $q,t$-Catalan polynomial, but with a small twist. It connects to the classical polynomial \cite{gorskymazin2013compactified1} by
    \[N_{a,b;1}(q,t) = \sum_{D\in \mathrm{Dyck}_{a,b}} q^{\dinv(D)}t^{\frac{(a-1)(b-1)}{2}-\mathtt{area}(D)},\]
    while
    \[ C_{a,b}(q,t)=C_{a,b}(t,q)=\sum_{D\in \mathrm{Dyck}_{a,b}} q^{\dinv(D)}t^{\mathtt{area}(D)},\]
    so
    \[C_{a,b}(q,t)=t^{\frac{(a-1)(b-1)}{2}}\,N_{a,b;1}(q,t^{-1}).\]
    We do not observe any functional equation for $N_{a,b;n}(q,t)$ when $n>1$. This is expected: we expect a functional equation for $[\Quot_m(R^n)]$, not for $[\Quot_m^R(\k[\![T]\!]^n)]$ \cite[\funceq]{huangjiang2023torsionfree}. However, as is discussed in \S\ref{sec:speculations}, we expect a three-variable rank-$n$ Catalan polynomial that satisfies additional symmetry.
  \end{remark}

\section{Main ingredients}
The main geometric input of our results is in the study of another moduli space. Recall $\gcd(a,b)=1$ and $R=\k[\![T^a,T^b]\!]$. Denote $\tl R=\k[\![T]\!]$, the normalization of $R$. We consider an open subset of $\Quot_m^R(\tl R^n)$ defined by
\[ E_R(\tl R^n;m) = \{L\in \Quot_m^R(\tl R^n): \tl R L = \tl R^n\},\]
where $\tl R L$ is the smallest $\tl R$-submodule of $\tl R^n$ containing $L$. There is a natural $\k^\times$-action on $\tl R=\k[\![T]\!]$ given by $t\cdot f(T)\coloneqq f(tT)$. This action preserves $R$, and induces a $\bG_m$-action on $\Quot_m^R(\tl R^n)$ and $E_R(\tl R^n;m)$. 

\begin{theorem}[{$\subeq$ Theorem~\ref{thm:bb-detailed}}]
    The $\Gm$-action on $E_R(\tl R^n;m)$ has a smooth fixed point locus whose connected components $F$ can be explicitly described as iterated Grassmannian bundles. Moreover, each Bia\l ynicki-Birula stratum $X_F\coloneqq \{x:\lim_{t\to 0} t\cdot x\in F\}$ is smooth, and is an affine bundle\footnote{In this paper, an \defn{affine bundle} is just a Zariski-local fiber bundle whose fiber is an affine space.} over $F$.
    \label{thm:bb}
\end{theorem}

Explicit understanding of the fixed point components and the ranks of the affine bundles gives the following corollary.
\begin{corollary}[$=$ \eqref{eq:ext-fiber-motive}]
    Recall $N_{a,b;n}(q,t)$ from \eqref{eq:q,t-sum-finite}. We have
    \[ \sum_{m\geq 0} [E_R(\tl R^n;m)]\, t^m = \L^{n^2\delta} \, t^{n\delta} N_{a,b;n}(\L^{-1},t^{-1}),\]
    where $\delta=(a-1)(b-1)/2=\dim_\k \tl R/R$.
\end{corollary}

The proof of Theorem~\ref{thm:bb} depends on two general results about representations of the ``rectangular grid'' poset, which is the key innovation of our methods. Given an integer $n\in \Z$, a finite poset $(P,\preceq)$, and a vector $\mathbf{n}\in \Z^P$, define the \defn{$P$-flag variety}
\[ \Fl_P(\mathbf{n};n)=\{(V_i)_{i\in P}: V_i\subeq \k^n,\; \dim_\k V_i=n_i,\; V_i\subeq V_j \text{ if }i\preceq j\},\]
viewed as a closed subset of $\prod_{i\in P} \Gr(n_i,n)$. Of course, it is empty unless $0\leq n_i\leq n$ for all $i$ and $n_i\leq n_j$ for $i\preceq j$. For $r\in \N$, let $[r]\coloneqq \{1,\dots,r\}$ be equipped with the usual ordering; then an $[r]$-flag variety in $\k^n$ of dimension vector $(n_1\leq \dots\leq n_r\leq n)$ is simply the usual partial flag variety $\Fl(n_1,\dots,n_r;n)$. Our next result shows that the flag varieties for a more general family of posets are almost as nice, except that they are not homogeneous spaces. A \defn{rectangular grid poset} is the product poset $[w]\times [h]$ (where $w,h\geq 1$), equipped with the ordering $(x,y)\preceq (x',y')$ if and only if $x\le x'$ and $y\le y'$. 

\begin{theorem}
    [$\subeq$ \Cref{thm:poset-flag-tame} and Example~\ref{eg:grid-tame}]
    Let $P$ be a rectangular grid poset. For $n\in \Z$ and $\mathbf{n}\in \Z^P$, the poset flag variety $\Fl_P(\mathbf{n};n)$ is smooth projective, can be built as an iterated Grassmannian bundle over a point, and admits an affine paving. Moreover, the motive of $\Fl_P(\mathbf{n};n)$ is given by the generalized $q$-multinomial coefficient $\qbinom{n}{\mathbf{n}}_{\L;P}$ defined in \eqref{eq:q-binom-grid}.\label{thm:poset-flag}
\end{theorem}
Obviously, much more can be potentially said about these poset flag varieties, which would amount to a generalization of Schubert calculus in type A. We leave that to future exploration.

The next general result is a homological acyclicity theorem about representations of the rectangular grid poset. As usual, we view any poset $P$ as a category whose objects are elements of $P$ and such that there is a unique morphism $i\to j$ for each $i\preceq j$. For a field $\k$, a \defn{$P$-representation} is a functor $V$ from $P$ to the category of finite-dimensional $\k$-vector spaces, which can be concretely represented as a system of vector spaces $(V_i)_{i\in P}$ together with morphisms $V(i\to j):V_i\to V_j$ for $i\preceq j$ that for each chain $i\preceq j\preceq k$ make the obvious triangle commute. An \defn{injective/surjective $P$-flag} is a $P$-representation in which all arrows are injective/surjective. 

Now let $P=[w]\times [h]$ be a rectangular grid poset. Let $V$ be an injective $P$-flag, and $W$ be a surjective $P$-flag. Let $\Hom_P(V,W)$ be the space of morphisms from $V$ to $W$ in the category of $P$-representations. Concretely, it consists of maps $h=(h_i:V_i\to W_i)_{i\in P}$ such that the 3D diagram $V\overset{h}{\to} W$ commutes. 

Our next result states that $\Hom_P(V,W)$ can be resolved by a cochain complex, defined as follows. The space of $0$-cochains is simply
\[ C^0(V,W)=\bigoplus_{i\in P} \Hom_\k(V_i,W_i).\]
There is an inclusion map $\Hom_P(V,W)\to C^0(V,W)$.

To define the space of $1$-cochains, let
\[ Q_1(P)=\{(i,j)\in P^2: i\lessdot j\}\]
be the set of covering relations of $P$; here $\lessdot$ is the covering relation, see Definition~\ref{def:covering}. An element $(i,j)\in Q_1(P)$ is more suggestively denoted as an arrow $i\to j$. Concretely, $Q_1(P)$ consists of vertical arrows $(x,y)\to (x,y+1)$ and horizontal arrows $(x,y)\to (x+1,y)$. Then define
\[ C^1(V,W)=\bigoplus_{i\to j} \Hom_\k(V_i,W_j).\]
Define the differential map $\partial:C^0(V,W)\to C^1(V,W)$ by the rule
\[ (\partial h)_{i\to j} = h_j\circ V(i\to j) - W(i\to j)\circ h_i \in \Hom_\k(V_i,W_j).\]
In other words, $\partial h$ measures the failure of commutativity of the diagram $V\overset{h}{\to} W$. 

To define the space of $2$-cochains, let
\[ Q_2(P)=\{(i,j)\in P^2: j=i+(1,1)\},\]
so pairs in $Q_2(P)$ are those of the form $((x,y), (x+1,y+1))$. We suggestively denote a pair $(i,j)\in Q_2(P)$ by a diagonal arrow $i\nearrow j$. Then define
\[ C^2(V,W)=\bigoplus_{i\nearrow j} \Hom_\k(V_i,W_j).\]
Define the differential map $\partial:C^1(V,W)\to C^2(V,W)$ by
\[ (\partial \rho)_{(x,y)\nearrow (x+1,y+1)} = \rho_{(x,y)\to (x+1,y)}+\rho_{(x+1,y)\to (x+1,y+1)}-\rho_{(x,y+1)\to (x+1,y+1)}-\rho_{(x,y)\to (x,y+1)},\]
in which each term needs to be pre- or postcomposed with an appropriate morphism of the form $V(i\to j)$ or $W(i\to j)$ in order to land in $\Hom_\k(V_{(x,y)},W_{(x+1,y+1)})$.
Intuitively, if $\rho\in C^1(V,W)$ is a ``prescription'' of the failure of commutativity of an imaginary diagram $V\overset{h}{\to} W$, then $\partial \rho$ measures how it fails to be compatible along each counterclockwise loop.
\begin{theorem}
[$=$ Theorem~\ref{thm:grid-acyclicity}]
    Let $P$ be a rectangular grid poset, $V$ be an injective $P$-flag, and $W$ be a surjective $P$-flag. Then the following is an exact sequence:
    \begin{equation}\label{eq:cochain-complex}
        0\to \Hom_P(V,W)\to C^0(V,W)\overset\partial\to C^1(V,W)\overset\partial\to C^2(V,W)\to 0.
    \end{equation}
    In particular, $$\dim_\k \Hom_P(V,W)=\sum_{i=0}^2 (-1)^i \dim_\k C^i(V,W).$$\label{thm:acyclicity}
\end{theorem}
\begin{remark}~
    \begin{itemize}
        \item That \eqref{eq:cochain-complex} is a chain complex and that it is exact at $C^0(V,W)$ are not hard to see. The main content of Theorem~\ref{thm:acyclicity} is the exactness at $C^1$ and $C^2$. 
        \item While bearing resemblance, Theorem~\ref{thm:acyclicity} is not a mere consequence of the general theory of quiver representations, in that the \emph{relations} in the quiver play a role in the formation of $C^2(V,W)$; see Theorem~\ref{thm:ext}.
        \item Consider the constant representation $\underline{\k}$ such that $\underline{\k}_i=\k$ and $\underline{\k}(i\to j)$ is the identity map on $\k$. If both $V\simeq W\simeq \underline{\k}$, then the exactness of \eqref{eq:cochain-complex} is nothing but the fact that a solid rectangle, as a CW-complex formed by filling in the squares in a rectangular grid, is cohomologically contractible. It would be interesting to ask for analogues of Theorem~\ref{thm:acyclicity} for more general posets or quivers, such as higher-dimensional grids.
    \end{itemize}
\end{remark}

We now briefly describe our ideas to prove Theorem~\ref{thm:bb}, and how we invoke Theorems~\ref{thm:poset-flag} and \ref{thm:acyclicity}. Recall $R=\k[\![T^a,T^b]\!]$ and $\tl R=\k[\![T]\!]$. First, we observe that the $\Gm$-fixed points on $\Quot_\bullet^R(\tl R^n)$ are precisely the ``homogeneous'' submodules of $\k [\![T]\!]^n$, namely, $R$-submodules of the form
\[ V=\bigg(\bigoplus_{i=0}^{N-1} V_i T^i\bigg) \oplus T^N \k[\![T]\!]^n,\]
where $V_i$ is a subspace of $\k^n$. For simplicity, we denote this $V$ by
\[ V=\hhat \bigoplus_{i=0}^\infty V_i T^i,\]
where it is understood that $V_i=\k^n$ for $i\geq N$. (Conversely, whenever we write down $\hhat \bigoplus_{i=0}^\infty V_i T^i$ such that $V_i=\k^n$ for $i\gg 0$, we understand it this way.) The vector space $V$ is a submodule if and only if it is closed under multiplication by $T^a$ and $T^b$, namely, $V_i\subeq V_j$ whenever $j-i\in \pairing{a,b}$. As such, the fixed point locus on $\Quot_\bullet^R(\tl R^n)$ can be identified with the flag variety over the poset $\N_{a,b}$, defined by
\begin{equation}\label{eq:Nab-poset}
    \N_{a,b}=\N \text{ as a set, with } i\preceq j \text{ if and only if }j-i\in \pairing{a,b}.
\end{equation}

A typical fiber of the Bia\l ynicki-Birula map (denoted by $\pi$), namely, the attracting locus of a specific fixed point $V=\hhat \bigoplus_{i=0}^\infty V_i T^i$, can be described as a Gr\"obner stratum
\[ \pi^{-1}(V)=\set{L\subeq_R \tl R^n:\init(L)=V},\]
where $\init(L)$ is the ``initial term data'' of $L$. In general, a Gr\"obner stratum can be parametrized by an affine variety whose variables (``Gr\"obner coordinates'') correspond to free non-leading terms in reduced Gr\"obner bases, and whose equations correspond to the Buchberger criterion, which can generally be arbitrarily complicated. 

Traditionally, Gr\"obner method over subrings of $\k[\![T]\!]$ often involves minimal generators of ideals of the valuation semigroup (which provide the general form of reduced Gr\"obner basis and thus the Gr\"obner coordinate system), and syzygies of these generators (which provide the equations in the Buchberger criterion); see \cite{gorskymazin2013compactified1,OblomkovRasmussenShende12,gmo2026generic,hmz2025}. Our method differs from all of the above by using \emph{all} available variables and relations, not just the \emph{minimal} ones. Moreover, we do it in a manner free of non-canonical choice of basis and term order. The outcome is the following parametrization of $\pi^{-1}(V)$. Recall the poset $P=\N_{a,b}$ from \eqref{eq:Nab-poset}, and consider $P$-representations
\[ V=(V_i)_{i\geq 0}, W=\k^n/V=(\k^n/V_i)_{i\geq 0}\]
with usual inclusion maps on $V$ and quotient maps on $W$. Then $V$ is an injective $P$-flag and $W$ is a surjective $P$-flag. Consider the shift
\begin{equation}\label{eq:shifted-poset-flag}
    V[-d]=(V_{i-d})_{i\geq 0},
\end{equation}
where $V_i\coloneqq 0$ if $i<0$; note that with this extension by zero, $V[-d]$ is still an injective $P$-flag. Let $C^0(V[-d],W)\overset\partial\to C^1(V[-d],W)$ be defined as before; note that the construction applies verbatim to any finite poset, not just the rectangular grid. 

\begin{proposition}[$\subeq$ Proof of Theorem~\ref{thm:bb-fiber}]
    Notation as above. Then $\pi^{-1}(V)$ is isomorphic to the affine variety
    \[ \set*{({}_r \phi)_{r>0}: {}_r \phi\in C^0_{\N_{a,b}}(V[-r],W), \partial({}_r \phi)={}_r \rho},\]
    with explicit elements ${}_r \rho\in C^1_{\N_{a,b}}(V[-r],W)\otimes_\k \k[{}_1 \phi,\dots,{}_{r-1}\phi]$.
    \label{prop:bb-fiber}
\end{proposition}
Note that the ambient affine space 
\[ \A=\bigoplus_{r>0} C^0(V[-r],W)=\bigoplus_{r,d:d\geq r>0} \Hom_\k(V_{d-r},\k^n/V_d)\] is finite-dimensional because $\k^n/V_d=0$ for $d\gg 0$. The elements ${}_r\rho$ shall be viewed as elements of $C^1(V[-r],W)$ whose matrix entries depend polynomially on matrix entries of ${}_s \phi$, $s<r$. Conceptually, $\A$ parametrizes \emph{$\k$-vector subspaces} $L\subeq \tl R^n$ such that $\init(L)=V$, while $\partial({}_r \phi)={}_r \rho$ corresponds to the closed condition that $L$ is an $R$-module, namely, closed under multiplication by $T^a$ and $T^b$.

To solve the system of equations that cuts out $\pi^{-1}(V)$, we solve for the equation 
\begin{equation}
    \partial({}_r \phi)={}_r \rho \tag{$*_r$}
\end{equation}inductively, starting from $r=1$. At each step, the equation is a \emph{nonhomogeneous linear equation} in the new variables (since ${}_r \rho$ depends only on the old variables). In general, such a linearity observation alone is not enough to tame the solution set: the equation $(*_r)$ may or may not have a solution, and the locus of choices $({}_1 \phi,\dots,{}_{r-1}\phi)$ such that $(*_r)$ has a solution could be arbitrarily complicated.

However, we claim the following miraculous fact.
\begin{proposition}[$\subeq$ Proof of Theorem~\ref{thm:bb-fiber}]
    If $V$ satisfies $V_0=\k^n$, then whenever ${}_r \phi$ solves $(*_r)$ for all $r=1,\dots,d-1$, the equation $(*_d)$ has a solution.\label{prop:solution-exists}
\end{proposition}
We will see that $E_R(\tl R^n;m)$ is precisely the disjoint union of $\pi^{-1}(V)$ for $V\in \Quot_m^R(\tl R^n)^{\Gm}$ satisfying $V_0=\k^n$. Provided the existence of solutions, the solution space of ${}_r \phi$ is an affine space whose dimension does not depend on the choice of ${}_s \phi$, $s<r$. With some elementary argument (Lemma~\ref{lem:trivial-aff-bundle}), this implies that $\pi^{-1}(V)$ is an affine space.

To prove the existence of solutions, we do not approach it explicitly using the precise expression of ${}_r\rho$; its complexity blows up as $r$ increases. Instead, we invoke Theorem~\ref{thm:acyclicity}, thanks to the following observation:

\begin{center}
\emph{Though $\N_{a,b}$ cannot be embedded into a rectangular grid, the support of $W$ can.}
\end{center}

More precisely, since $V_0=\k^n$, $W_i=0$ for all $i\in \pairing{a,b}$. So, the support of $W$ is contained in the gap set $G=\N\setminus \pairing{a,b}$, which is well-known to take the shape of a Young diagram:
\begin{equation}\label{eq:gap-young}
    G=\{ab-ax-by>0:x,y\in \Z_{\ge 1}\}.
\end{equation}
As a result, $G$ is embeddable into a rectangular grid poset, say $Q$. The relevant embedding is explained in Example~\ref{eg:gap-young}.
After making suitable restrictions and extensions, there are injective and surjective $Q$-flags, still denoted by $V[-r]$ and $W$, such that the map $\partial:C^0_P(V[-r],W)\to C^1_P(V[-r],W)$ is identical to part of the exact sequence
\[0\to \Hom_Q(V[-r],W)\to C^0_Q(V[-r],W)\overset\partial\to C^1_Q(V[-r],W)\overset\partial\to C^2_Q(V[-r],W)\to 0,\]
an instance of \eqref{eq:cochain-complex}.

Now it turns out that it is not hard to prove that $\partial({}_r \rho)=0$ (assuming $(*_s)$ is solved for all $s<r$), which essentially results from the fact that the multiplication map by $T^a$ and the multiplication map by $T^b$ commute. By exactness, we conclude that $\partial({}_r \phi)={}_r \rho$ has a solution. This concludes the proof of the main engine of the paper.

  \section{An acyclicity theorem and proof of Theorem~\ref{thm:acyclicity}}\label{sec:acyclicity}

The goal of this section is to prove \Cref{thm:acyclicity}. We begin by setting up notation to put the theorem in a more general framework. Let $\k$ be any field, and $[n]=\set{1,\dots,n}$ for nonnegative integers $n$.

\subsection{Poset quivers}
\begin{definition}\label{def:covering}
    In a poset $(P,<)$, a \defn{covering relation}, denoted by $i\lessdot j$, is a strict comparison $i< j$ such that there does not exist $z\in P$ with $i<z<j$. A \defn{saturated ascending chain} from $i$ to $j$ is a chain of covering relations $i_0\lessdot i_1\lessdot\dots\lessdot i_n$ in $P$. 
\end{definition}
\begin{definition}
    Given a poset $P$, the \defn{poset quiver} associated with $P$ is a quiver with relations $(Q_P,I_P)$ whose vertices are elements of $P$, arrows are $i\to j$ where $i\lessdot j$ is a covering relation in $P$, and whose relation ideal $I_P$ is generated by
    \begin{equation}
        [i_0\to i_1\to\dots\to i_m] - [j_0\to j_1\to\dots \to j_n],
    \end{equation}
    where $i_0\lessdot \dots\lessdot i_m$ and $j_0\lessdot \dots\lessdot j_n$ are two saturated ascending chains from $i_0=j_0$ to $i_m=j_n$. The \defn{incidence algebra} of $P$ is $A_P:=\k Q_P/I_P$, where $\k Q_P$ is the path algebra of $Q_P$. A (finite dimensional) \defn{$P$-module} or \defn{$P$-representation} is a left $A_P$-module $V$ such that $\dim_\k V_x< \infty$ for all $x\in P$. A $P$-module is an \defn{injective $P$-flag} (\resp \defn{surjective $P$-flag}) if the linear map associated with each arrow is injective (\resp surjective). For two $P$-modules $V,W$, we denote $\Hom_{A_P}(V,W)$ simply by $\Hom_P(V,W)$, and similarly $\Ext^i_{P}(V,W):=\Ext^i_{A_P}(V,W)$. 
\end{definition}

\begin{warning}
    The notion of injective $P$-flags has nothing to do with the notion of injective modules in the category of $P$-representations (i.e., $A_P$-modules). 
\end{warning}

\subsection{The fundamental cochain complex}\label{subsec:cochain}

\begin{definition}\label{def:uniform}
    Given a quiver with relations $(Q,I)$, a nonzero relation $r\in I\subeq \k Q$ is \defn{uniform} if there exist vertices $s,t\in Q$ (necessarily unique) such that $r$ is a linear combination of paths starting at $s$ and ending at $t$. In this case, the source and target of $r$ are defined as $s(r):=s, t(r):=t$. A \defn{uniform generating set} of $I$ is a generating set $Q_2=\set{r_1,\dots,r_l}$ of $I$ consisting of uniform relations. 
\end{definition}

Given a quiver with relations $(Q,I)$, a uniform generating set $Q_2$ of relations, and two (left) $A:=\k Q/I$-modules $V,W$, we consider a cochain complex $C^\bullet(V,W)$, which will be the main player of this section. Let $Q_0, Q_1$ be the vertex set and the arrow set of $Q$. For an arrow $a\in Q_1$, let $s(a),t(a)$ denote the source and target of $a$. The cochain complex is defined as follows:
\begin{equation}\label{eq:fundamental-cochain-complex}
\begin{aligned}
    C^\bullet_P(V,W;Q_2)&= C^\bullet(V,W;Q_2)=\parens*{0\to C^0(V,W)\to C^1(V,W) \to C^2(V,W;Q_2)} \\   &:= \parens*{0\to \bigoplus_{x\in Q_0} \Hom_\k(V_x,W_x) \to \bigoplus_{a\in Q_1} \Hom_\k(V_{s(a)},W_{t(a)}) \to \bigoplus_{r\in Q_2} \Hom_\k(V_{s(r)},W_{t(r)})}.
\end{aligned}
\end{equation}
In the discussion below we often omit \(Q_2\) from the notations when it is clear from the context.

The differential maps are as follows. For $\phi=(\phi_x)_{x\in Q_0}$ and $a:x\to y$, let $\partial\phi_a:=\phi_y\circ a|_V-a|_W\circ \phi_x$; by abuse of notation, we also denote the defining formula for $\partial\phi_a$ by $\partial\phi_a:=\phi_y-\phi_x:V_x\to W_y$. This completes the definition of the zeroth differential: $\partial\phi = (\partial \phi_a)_{a\in Q_1}$. To define the first differential, for $\rho=(\rho_a)_{a\in Q_1}$ and any path $\gamma:x_0\map[a_1]  \dots \map[a_m] x_m$ of $Q$, define $\int_\gamma \rho:=\sum_{i=1}^m \int_\gamma^i\rho\in \Hom_\k(V_{x_0},W_{x_m})$, where
\begin{equation}
    \int_\gamma^i\rho := (V_{x_0}\map[a_{i-1}\cdots a_1|_V] V_{x_{i-1}}\map[\rho_{a_i}] W_{x_i}\map[a_m\cdots a_{i+1}|_W]W_{x_m}).
\end{equation}
By a similar abuse of notation, the definition can also be notated by $\int_\gamma \rho=\rho_{a_1}+\dots+\rho_{a_m}$. For $r\in Q_2$, the definition is extended to $\int_r \rho\in \Hom_\k(V_{s(r)},W_{t(r)})$ by $\k$-linearity. The first differential $\partial:C^1(V,W)\to C^2(V,W)$ is then defined as $\partial\rho = (\int_r \rho)_{r\in Q_2}$.

It is a classical fact that $(C^\bullet(V,W),\partial)$ is a cochain complex whose cohomology $H^\bullet(V,W)$ has natural interpretations. 
\begin{theorem}\label{thm:ext}
    Notation as above. Then $(C^\bullet(V,W),\partial)$ is a cochain complex, and its cohomology satisfies $H^i(V,W)\simeq_\k \Ext^i_A(V,W)$ for $i=0,1$.
\end{theorem}
\begin{proof}
    This is essentially \cite[Proposition 2.1]{bardzell97}; see also \cite[Eq.~(5.3)]{bergman08}. More precisely, for $x,y\in Q_0$, consider the $A,A$-bimodule $P_{xy}=Ae_x\otimes e_y A$, where $e_x$ is the empty path on the vertex $x$. By either of the citations, there is a resolution of $A$ as an $A,A$-bimodule:
    \begin{equation}
        \bigoplus_{r\in Q_2} P_{t(r)s(r)}\to \bigoplus_{a\in Q_1} P_{t(a)s(a)}\to \bigoplus_{x\in Q_0} P_{xx} \to A\to 0.
    \end{equation}
    Applying $(-)\otimes_A V$, the resulting sequence of left $A$-modules is still exact because $P_{xy}\otimes_A V=Ae_x\otimes_\k V_y$. This gives a resolution of $V$ as a left $A$-module. Applying $\Hom_A(-,W)$ to this resolution gives what we want.
\end{proof}
\begin{remark}
    The relations play a role in $\Ext^1$ though not $\Ext^0$: if we work over $\k Q$ instead, then $\Ext^i_{\k Q}(V,W)$ is the cohomology of the chain complex $0\to C^0(V,W)\to C^1(V,W)\to 0$.
\end{remark}

\subsection{The grid poset}

\begin{definition}
    Let $m,n\geq 1$. The \defn{standard $n\times m$ grid poset} $P_{n,m}$ is the poset $[n]\times [m]$ with the product order $(i,j)\leq (i',j')$ if and only if $i\leq i'$ and $j\leq j'$. The quiver, relation ideal, and the incidence algebra of $P_{n,m}$ are denoted by $Q_{n,m}, I_{n,m}$ and $A_{n,m}$, respectively. A poset $P$ is called a \defn{(finite) grid poset} if $P$ is isomorphic to $P_{n,m}$ for some $n,m\geq 1$. 
\end{definition}

We follow the Cartesian convention (not the matrix convention) to arrange $P_{n,m}$ on a plane: $(n,1)$ is at the southeast and $(n,m)$ is at the northeast.

\begin{remark}\label{rmk:grid-relation}
    Given the grid poset $P_{n,m}$, consider the set $R_{n,m}\subeq \k Q_{n,m}$ consisting of relations of the form
    \begin{equation}\label{eq:commuting-relation}
        r_{ij}:=[(i,j)\to (i+1,j)\to (i+1,j+1)]-[(i,j)\to (i,j+1)\to (i+1,j+1)], \quad (i,j)\in [n-1]\times [m-1].
    \end{equation}
    Then $R_{n,m}$ is a uniform generating set of $I_{n,m}$.
\end{remark}

Representations and the fundamental cochain complex of a grid poset $P$ have an intuitive interpretation. A representation of $P$ is a commutative diagram of $\k$-vector spaces with the shape of $P$. Given $P$-modules $V,W$, an element $\phi\in C^0(V,W)$ is an assignment of ``vertical arrows'' $\phi_x:V_x\to W_x$ for all $x\in P$. This element is in $\Hom_P(V,W)$ if and only if the resulting 3D diagram commutes. For a general $\phi\in C^0(V,W)$, the element $\partial\phi\in C^1(V,W)$ measures the failure of commutativity of each vertical unit square. A general element $\rho\in C^1(V,W)$ \textit{prescribes} a failure for each vertical square. If $\rho=\partial\phi$ for some $\phi$, then $\rho$ must be \textit{compatible} on each unit cube, i.e., $\partial \rho=0\in C^2(V,W)$. For a general prescription $\rho\in C^1(V,W)$, $\partial \rho$ measures the incompatibility of $\rho$. The following diagram illustrates the case $P=P_{2,2}$, where the elements of $P$ are labeled by $0< \set{1,2}< 3$.

\[\begin{tikzcd}[
	row sep={4em, between origins},
	column sep={4em, between origins}
]
	W_0 \ar[rr] \ar[dr] & & W_1 \ar[dr]
	\\
	& W_2 \ar[rr] & & W_3
	\\
	V_0 \ar[uu, "\phi_0"'] \ar[rr] \ar[dr]
		\ar[uurr, dashed, sloped, bend right=25, "\rho_{01}"', pos=0.55]
		\ar[ur, dashed, sloped, "\rho_{02}", pos=0.4]
	& & V_1 \ar[uu, "\phi_1"', pos=0.7] \ar[dr]
			\ar[ur, dashed, sloped, "\rho_{13}", pos=0.4]
	\\
	& V_2 \ar[uu, "\phi_2"', pos=0.3] \ar[rr]
		\ar[uurr, dashed, sloped, bend right=25, "\rho_{23}", pos=0.4]
	& & V_3 \ar[uu, "\phi_3" ']
\end{tikzcd}\]

The above discussion gives a direct explanation of the definition of $\partial$ and why $\partial^2=0$. To make the sign convention even more transparent, it is helpful to think of the following 2D picture
\[\begin{tikzcd}
	\phi_{12} \arrow[r, "\rho_{12,22}"] & \phi_{22} \\
	\phi_{11} \arrow[u, "\rho_{11,12}"] \arrow[r, "\rho_{11,21}"] \arrow[ur, phantom, "\circlearrowleft" description] & \phi_{21} \arrow[u, "\rho_{21,22}"']
\end{tikzcd}\]
where the standard label $(i,j)$ of $P$ is abbreviated by $ij$, and we view $\phi\in C^0(V,W)$ as an assignment to each node of $P$, while $\rho\in C^1(V,W)$ as an assignment to each arrow. Using the abuse of notation as in \Cref{subsec:cochain}, the differential rule is simply
\begin{equation}
    \partial\phi_{11,12}=\phi_{12}-\phi_{11}: V_{11}\to W_{12}, \; \text{etc.},
\end{equation}
and the ``counterclockwise rule''
\begin{equation}\label{eq:ccw-rule}
    \partial\rho_{r_{11}}=\rho_{11,21}+\rho_{21,22}-\rho_{12,22}-\rho_{11,12}:V_{11}\to W_{22}.
\end{equation}
If both $V,W$ are the constant representation, namely, putting $\k$ on each node and letting every arrow be the identity map, then each of $\phi_x, \rho_a$ can be viewed as a scalar in $\k$, so $C^\bullet(V,W)$ reduces to the CW cochain complex of the 2D CW complex associated with the grid (where each unit square corresponds to a two-cell). 

\subsection{The acyclicity theorem}
We are ready to state the main result of this section.
\begin{theorem}\label{thm:grid-acyclicity}
    Let $P$ be a grid poset, $V$ be an injective $P$-flag, and $W$ be a surjective $P$-flag. Then the cochain complex
    \begin{equation}\label{eq:grid-acyclicity}
        0\to \Hom_P(V,W)\to C^0(V,W)\to C^1(V,W)\to C^2(V,W)\to 0
    \end{equation}
    is exact. In particular, $\Ext^1_P(V,W)=0$ and $\dim_\k \Hom_P(V,W)=\sum_{i=0}^2 (-1)^i \dim_\k C^i(V,W)$.
\end{theorem}
In spirit, this theorem should be a reflection of the fact that the 2D CW complex associated with the grid is contractible as a topological space. For the rest of the section, we provide a direct proof. Note that by \Cref{thm:ext}, the exactness at $\Hom_P(V,W)$ and $C^0(V,W)$ is already known, so it remains to prove exactness at the other two places. 

\subsubsection{Proof of exactness at $C^2$} It suffices to make the following simple observation:
\begin{lemma}\label{lem:lift1}
    Let $P=P_{2,2}$ and assume the setting of \Cref{thm:grid-acyclicity}. Then for any $\delta\in \Hom_\k(V_{11},W_{22})$, $\rho_{11,21}\in \Hom_\k(V_{11},W_{21})$, $\rho_{11,12}\in \Hom_\k(V_{11},W_{12})$, and $\rho_{21,22}\in \Hom_\k(V_{21},W_{22})$, there exists $\rho_{12,22}\in \Hom_\k(V_{12},W_{22})$ such that $\partial \rho = \delta$.
    \[\begin{tikzcd}
	12 \arrow[r, dashed, "\exists \rho_{12,22}"] & 22 \\
	11 \arrow[u, "\rho_{11,12}"] \arrow[r, "\rho_{11,21}"] \arrow[ur, phantom, "\circlearrowleft \delta" description] & 21 \arrow[u, "\rho_{21,22}"']
\end{tikzcd}\]
\end{lemma}
\begin{proof}
    By \eqref{eq:ccw-rule} and partially undoing the abuse of notation, the equation for $\rho_{12,22}$ is
    \begin{equation}
        \rho_{12,22}|_{V_{11}}=\delta+\rho_{11,21}+\rho_{21,22}-\rho_{11,12}\in \Hom_\k(V_{11},W_{22}).
    \end{equation}
    If we denote the right-hand side by $f$, this is equivalent to the extension problem
    \[\begin{tikzcd}
	& {W_{22}} \\
	{V_{11}} & {V_{12}}
	\arrow["f", from=2-1, to=1-2]
	\arrow[hook, from=2-1, to=2-2]
	\arrow["{\rho_{12,22}}"', dashed, from=2-2, to=1-2]
    \end{tikzcd}\]
    Since the natural map $V_{11}\to V_{12}$ is injective, and any vector space is an injective module, a solution $\rho_{12,22}$ exists.
\end{proof}

\begin{proof}
    [Proof of exactness of \eqref{eq:grid-acyclicity} at $C^2$]
    Let $\delta\in C^2(V,W)$ be arbitrary. We construct $\rho\in C^1(V,W)$ such that $\partial \rho =\delta$ as follows. First, assign the values of $\rho$ at bottom horizontal arrows and bottom vertical arrows (namely, $\rho_{i1,(i+1)1}$ and $\rho_{i1,i2}$) arbitrarily, e.g., zero. Then fill in the values of $\rho$ at the second lowest horizontal arrows (namely, $\rho_{i2,(i+1)2}$) in an arbitrary way that respects $\delta$; this is always possible by \Cref{lem:lift1}. Now, assign $\rho$ at the second lowest vertical arrows (namely, $\rho_{i2,i3}$) arbitrarily, e.g., zero, and then repeat the process. 
\end{proof}

\subsubsection{Proof of exactness at $C^1$} Throughout this proof, let $\rho=(\rho_{x,y})_{x\lessdot y}\in C^1(V,W)$ be given such that $\partial\rho=0$. For any $x,y\in P$ with $x<y$, define $\rho_{x,y}=\sum_{i=1}^l \rho_{x_{i-1},x_i}\in \Hom_\k(V_x,W_y)$, where $x=x_0\lessdot \cdots \lessdot x_l=y$ is any saturated chain from $x$ to $y$. Because $\partial \rho=0$, $\rho_{x,y}$ is well-defined.

The goal of the proof is to construct $\phi\in C^0(V,W)$ that is ``compatible'' with $\rho$, in the sense that $\partial\phi=\rho$. The strategy is filling in entries of $\phi$ in a certain order, ensuring compatibility with $\rho$ along the way. We will need a few lemmas as the induction seeds.

\begin{lemma}\label{lem:lift2}
    Let $P=P_{2,1}=\set{1,2}$ and assume the setting of \Cref{thm:grid-acyclicity}. Then for any $\rho\in \Hom_\k(V_1,W_2)$ and $\phi_2\in \Hom_\k(V_2,W_2)$, there exists $\phi_1\in\Hom_\k(V_1,W_1)$ such that $\partial \phi=\rho$.
    \[\begin{tikzcd}
	|[circle, draw, dashed, inner sep=1pt]|{\exists\phi_{1}} &  {\phi_{2}} 
	\arrow["{\rho}", from=1-1, to=1-2]
\end{tikzcd}\]
    Dually, for any $\rho\in \Hom_\k(V_1,W_2)$ and $\phi_1\in \Hom_\k(V_1,W_1)$, there exists $\phi_2\in\Hom_\k(V_2,W_2)$ such that $\partial \phi=\rho$.
    \[\begin{tikzcd}
	{\phi_{1}}  & |[circle, draw, dashed, inner sep=1pt]|{\exists\phi_{2}} 
	\arrow["{\rho}", from=1-1, to=1-2]
\end{tikzcd}\]
\end{lemma}
\begin{proof}
    The equation for $\phi_1$ is $\pi_{1,2}\circ \phi_1=\phi_2|V_1-\rho:V_1\to W_2$, where $\pi_{1,2}:W_1\onto W_2$ is the natural map. A solution for $\phi_1$ exists because $\pi_{1,2}$ is onto. The proof for the dual version is analogous, where we instead use the injectivity of the map $V_1\to V_2$.
\end{proof}

\begin{lemma}\label{lem:fill-triangle}
    Consider the following diagram of finite-dimensional $\k$-vector spaces consisting of solid arrows, such that the square commutes. Then there exists $h:V_2\to W_1$ that makes both triangles commute.
    \[\begin{tikzcd}
	{W_1} & {W_2} \\
	{V_1} & {V_2}
	\arrow[two heads, from=1-1, to=1-2]
	\arrow["{\phi_1}", from=2-1, to=1-1]
	\arrow[hook, from=2-1, to=2-2]
	\arrow["\exists h"{description}, dashed, from=2-2, to=1-1]
	\arrow["{\phi_2}"', from=2-2, to=1-2]
    \end{tikzcd}\]
\end{lemma}
\begin{proof}
    Without loss of generality, let $V_1=\k^{n_1}\incl V_2=\k^{n_1+n_2}$ be the inclusion into the first $n_1$ coordinates, and $W_1=\k^{m_1+m_2}\onto W_2=\k^{m_2}$ be the projection onto the last $m_2$ coordinates. Let $\phi_1$ be represented by the block matrix $\bmat{A\\C}$, and $\phi_2$ by $\bmat{C'& D}$, where $A\in \Mat_{m_1\times n_1}(\k)$, $C,C'\in \Mat_{m_2\times n_1}(\k)$, and $D\in \Mat_{m_2\times n_2}(\k)$. The commutativity of the square means $C=C'$. Hence, for any choice of $B\in \Mat_{m_1\times n_2}(\k)$, the linear map
    \begin{equation}
        h=\bmat{A&B\\C&D}
    \end{equation}
    makes both triangles commute.
\end{proof}

\begin{lemma}\label{lem:lift3}
    Let $P=P_{3,1}=\set{1,2,3}$ and assume the setting of \Cref{thm:grid-acyclicity}. Then for any $\rho\in C^1(V,W)$, $\phi_1\in \Hom_\k(V_1,W_1)$, and $\phi_3\in \Hom_\k(V_3,W_3)$, such that $\rho_{1,3}=\phi_3-\phi_1\in \Hom_\k(V_1,W_3)$, there exists $\phi_2\in \Hom_\k(V_2,W_2)$ such that $\partial \phi=\rho$.
    \[\begin{tikzcd}
	{\phi_{1}} & |[circle, draw, dashed, inner sep=1pt]| {\exists\phi_{2}} & {\phi_{3}}
	\arrow["{\rho_{1,2}}"', from=1-1, to=1-2]
	\arrow["{\rho_{1,3}}", curve={height=-16pt}, from=1-1, to=1-3]
	\arrow["{\rho_{2,3}}"', from=1-2, to=1-3]
\end{tikzcd}\]
\end{lemma}
\begin{proof}
    Consider the following diagram:
    \[\begin{tikzcd}
	{W_1} & {W_2} & {W_3} \\
	{V_1} & {V_2} & {V_3}
	\arrow["{\pi_{1,2}}", two heads, from=1-1, to=1-2]
	\arrow["{\pi_{2,3}}", two heads, from=1-2, to=1-3]
	\arrow["{\phi_1}", from=2-1, to=1-1]
	\arrow["{\rho_{1,2}}"{description}, from=2-1, to=1-2]
	\arrow["{\iota_{1,2}}"', hook, from=2-1, to=2-2]
	\arrow[dashed, from=2-2, to=1-2]
	\arrow["{\rho_{2,3}}"{description}, from=2-2, to=1-3]
	\arrow["{\iota_{2,3}}"', hook, from=2-2, to=2-3]
	\arrow["{\phi_3}"', from=2-3, to=1-3]
    \end{tikzcd}\]
    Define $f_{1,2}=\pi_{1,2}\circ \phi_1 + \rho_{1,2}:V_1\to W_2$ and $f_{2,3}=\phi_3\circ \iota_{2,3}-\rho_{2,3}:V_2\to W_3$. Then $\partial\phi=\rho$ if and only if both triangles in the following diagram commute:
    \[\begin{tikzcd}
	& {W_2} & {W_3} \\
	{V_1} & {V_2}
	\arrow["{\pi_{2,3}}", two heads, from=1-2, to=1-3]
	\arrow["{f_{1,2}}", from=2-1, to=1-2]
	\arrow["{\iota_{1,2}}"', hook, from=2-1, to=2-2]
	\arrow["{\phi_2}"{description}, dashed, from=2-2, to=1-2]
	\arrow["{f_{2,3}}"', from=2-2, to=1-3]
    \end{tikzcd}\]
    We note that the outer arrows in this diagram commute, since
    \begin{align}
        \pi_{2,3}\circ f_{1,2}-f_{2,3}\circ \iota_{1,2}&=\pi_{1,3}\circ \phi_1+\pi_{2,3}\circ \rho_{1,2}-\phi_3\circ \iota_{1,3}+\rho_{2,3}\circ \iota_{1,2}\\
        &=(\phi_1-\phi_3)+\rho_{1,3}=0.
    \end{align}
    Thus, by \Cref{lem:fill-triangle}, there exists $\phi_2$ that fulfills the requirement.
\end{proof}

\begin{proof}
    [Proof of exactness of \eqref{eq:grid-acyclicity} at $C^1$]

    Let $\rho\in C^1(V,W)$ with $\partial\rho=0$ be given. Pick $\phi_{n1}$ arbitrarily, e.g., zero. Then we sequentially use \Cref{lem:lift2} to fill in $\phi_{(n-1)1}, \dots, \phi_{11}$. 

    Now to proceed to the second row, we first fill in $\phi_{n2}$ using \Cref{lem:lift2} along the arrow $n1\to n2$. Note that $\rho_{(n-1)1,n2}=\phi_{n2}-\phi_{(n-1)1}$ from the established path $(n-1)1\to n1\to n2$. Therefore, we may apply \Cref{lem:lift3} to the new path $(n-1)1\to (n-1)2\to n2$, and fill in $\phi_{(n-1)2}$. We have so far established the compatibility between $\phi$ and $\rho$ along the bottom row of horizontal arrows and the bottom right unit square. Continuing this process from right to left, we fill in the second row.

    Repeating this construction row by row, we complete the construction of $\phi$. The following diagram illustrates the order in which $\phi_x$ are filled in, for $P=P_{4,3}$.
    \[\begin{tikzcd}
	12 & 11 & 10 & 9 \\
	8 & 7 & 6 & 5 \\
	4 & 3 & 2 & 1
	\arrow[from=1-1, to=1-2]
	\arrow[from=1-2, to=1-3]
	\arrow[from=1-3, to=1-4]
	\arrow[from=2-1, to=1-1]
	\arrow[from=2-1, to=2-2]
	\arrow[from=2-2, to=1-2]
	\arrow[from=2-2, to=2-3]
	\arrow[from=2-3, to=1-3]
	\arrow[from=2-3, to=2-4]
	\arrow[from=2-4, to=1-4]
	\arrow[from=3-1, to=2-1]
	\arrow[from=3-1, to=3-2]
	\arrow[from=3-2, to=2-2]
	\arrow[from=3-2, to=3-3]
	\arrow[from=3-3, to=2-3]
	\arrow[from=3-3, to=3-4]
	\arrow[from=3-4, to=2-4]
    \end{tikzcd}\]
\end{proof}

\begin{proof}
    [Proof of \Cref{thm:grid-acyclicity} and \Cref{thm:acyclicity}]
    The preceding arguments conclude the proof of \Cref{thm:grid-acyclicity}. Unwrapping the statement of \Cref{thm:grid-acyclicity}, we get precisely \Cref{thm:acyclicity}.
\end{proof}

\begin{remark}
    If we track the proof of \Cref{thm:grid-acyclicity}, we note that the conclusion still holds as long as $V$ is injective and rows of $W$ (but not necessarily columns) are surjective. By symmetry, the same is true if $V$ is injective and columns of $W$ (but not necessarily rows) are surjective. This would imply, for example, the exactness still holds if $W$ is composed of a surjective flag on a subgrid and rows of zeros below; note that $W$ ceases to be a surjective flag. The same is true if $W$ is composed of a surjective flag on a subgrid and columns of zeros to its left. However, \Cref{thm:grid-acyclicity} stops being true if $W$ is composed of a surjective flag on a northeast block and zero elsewhere. For example, take $P=P_{2,2}=\set{11,21,12,22}$, $V_{11}=0$, $V_{21}=V_{12}=V_{22}=\k$, $W_{22}=\k$, $W_{11}=W_{21}=W_{12}=0$, and the nonzero maps in $V$ are identity maps. Take $\rho_{21,22}=a$ and $\rho_{12,22}=b$, where $a\neq b$ are scalars in $\k$. Note that $C^2(V,W)=0$, so $\partial\rho=0$. However, there does not exist $\phi\in C^0(V,W)$ such that $\partial\phi=\rho$, as this would force $\phi_{22}=a-0=b-0$.

    The above generalization will not be needed in the rest of the paper, but it would be necessary if one uses the framework of \Cref{sec:toric} to describe the punctual Quot scheme of a plane, $\Quot_m(\k[\![x,y]\!]^n)$, to be discussed in a forthcoming companion paper. 
\end{remark}

\section{Poset flag varieties and proof of Theorem~\ref{thm:poset-flag}}\label{sec:poset}

In this section, we study the poset flag varieties over a class of posets that contain the rectangular grids, and then prove a generalization of \Cref{thm:poset-flag}. Poset flag varieties are variants of quiver flag varieties studied in \cite{craw11quiver, gu24rim}; however, since poset quivers have relations, results in these works cannot be directly applied to our setting. 

\subsection{General setup}
Let $(P,\preceq)$ be a finite poset and $\k$ be a field. Recall the $P$-flag variety of dimension vector $\mathbf{n}\in \Z^P$ in $\k^n$ is
\[ \Fl_P(\mathbf{n};n)=\{V=(V_i)_{i\in P}: V_i\subeq \k^n,\; \dim_\k V_i=n_i,\; V_i\subeq V_j \text{ if }i\preceq j\}.\]
We view it as a reduced closed subscheme of $\prod_{i\in P} \Gr(n_i,n)$ cut out by closed conditions $V_i\subeq V_j$ for $i\preceq j$. 

Our results in this section will hold whether or not $\Fl_P(\mathbf{n};n)$ is empty, but for the sake of completeness, we verify the obvious numerical criterion for nonemptiness.
\begin{proposition}\label{prop:admissible}
    The variety $\Fl_P(\mathbf{n};n)$ is nonempty if and only if the pair $(\mathbf{n};n)$ is \defn{admissible}, namely, 
    \[ 0\leq n_i\leq n \text{ for all $i$ and } n_i\leq n_j \text{ for }i\preceq j.\]
\end{proposition} 
\begin{proof}
    The forward implication is evident. To prove the converse, let $(\mathbf{n};n)$ be admissible. Assign $V_i=\k^{n_i}\times 0\subeq \k^n$, the span of the first $n_i$ standard basis vectors. Then whenever $i\preceq j$, $n_i\leq n_j$ by admissibility, so $V_i\subeq V_j$.
\end{proof}

The poset flag variety is naturally equipped with a tautological bundle $\calV_\infty\coloneqq \k^n$, the trivial vector bundle of rank $n$ on $\Fl_P(\mathbf{n};n)$. We have tautological subbundles $\calV_i\subeq \calV_\infty$ of rank $n_i$, whose fiber at the point $V$ is $V_i$. For $i\preceq j$, the quotient $\calV_j/\calV_i$ is a vector bundle of rank $n_j-n_i$, whose fiber at the point $V$ is $V_j/V_i$. 

\begin{remark}~
    \begin{enumerate}
        \item The natural $\GL_n$-action on the poset flag variety may not be transitive. For example, let $P=[2]\times [2]=\{11,21,12,22\}$, $(\mathbf{n};n)=(0,1,1,2;2)$. Then the function $V\mapsto \dim V_{12}\cap V_{21}$ can take values $0$ and $1$, showing that $\Fl_P(\mathbf{n};n)$ cannot be a single $\GL_n$-orbit. In this example, $\calV_{12}\cap\calV_{21}$ is not a vector bundle.
        \item \label{rmk:poset-flag-butterfly} The poset flag variety may be singular. Consider the ``butterfly'' poset $P=\{1,2,3,4\}$ where $i\prec j$ if and only if $i\in \{1,2\}$ and $j\in \{3,4\}$. Consider $(\mathbf{n};n)=(1,1,2,2;3)$. Then an explicit computation by stratification according to $\dim V_1\cap V_2$ shows
        \begin{equation}
        [\Fl_P(\mathbf{n};n)]=1+4\L+6\L^2+5\L^3+2\L^4.
        \end{equation}
        Since $\Fl_P(\mathbf{n};n)$ is projective, the lack of Poincar\'e duality implies that $\Fl_P(\mathbf{n};n)$ is not smooth.
    \end{enumerate}
    \label{rmk:poset-flag}
\end{remark}

\subsection{Combinatorics of tame posets}
We define a property on finite posets that is satisfied by rectangular grids and is sufficient for the poset variety to be nice. 

\begin{definition}\label{def:tame}
  Let $P$ be a finite poset, and let $\hhat P=P\sqcup \set{\pm \infty}$. A \defn{traverse} of $P$ is a permutation $x_1,x_2,\dots,x_{\abs{P}}$ of $P$. Given a traverse of $P$ and $0\leq i\leq \abs{P}$, let $P_i:=\set{x_1,\dots,x_i}$. Furthermore, for $1\leq i\leq \abs{P}$, define subsets of $\hhat P$:
  \begin{equation}
      P_i^-:=(P_i\cap (-\infty, x_i)_{P})\sqcup \set{-\infty},\quad P_i^+:=(P_i\cap (x_i,\infty)_P)\sqcup \set{\infty}.
  \end{equation}
  The traverse is \defn{good} if for each $1\leq i\leq \abs{P}$,  
  \begin{enumerate}
      \item\label{it:prec} the set $P^-_i$ admits a unique maximal element in $\hhat{P}$, and
      \item\label{it:succ} the set $P^+_i$ admits a unique minimal element in $\hhat{P}$.
  \end{enumerate}
  In this case, we denote $x_i^-:=\max P_i^-$ and $x_i^+:=\min P_i^+$. We call $P$ \defn{tame}, if a good traverse of $P$ exists.   
\end{definition}

\begin{example}\label{eg:grid-tame}
    The grid poset $P=[w]\times [h]$ is tame; the ``reading order'' is a good traverse:
    \begin{equation}
        \begin{aligned}
            &(1,w),(2,w),\dots,(h,w),\\
            &(1,w-1),\dots,(h,w-1),\\
            &\cdots\\
            &(1,1),(2,1),\dots,(h,1).
        \end{aligned}
    \end{equation}
    Here, we orient $P$ such that the maximal element $(h,w)$ is located at the northeast.

    Moreover, for $x\in P$, it is not hard to verify that $x^+$ with respect to this traverse is the vertex directly above $x$ (or $\infty$ if outside the grid), and $x^-$ is the vertex to the left of $x$ (or $-\infty$ if outside the grid). Indeed, this follows from the observation that $P_i$ is a Young diagram in British convention, i.e., with gravity to the northwest. More generally, any standard Young tableau of the rectangular partition $(n^m)$ in British convention gives a good traverse with the same description of $x^+, x^-$.
\end{example}

\begin{example}
    Subposets of tame posets may not be tame. Indeed, consider $[3]\times [3]$, which is tame by above, and its subposet $P=\{12,21,23,13\}$. This subposet is isomorphic to the butterfly poset in Remark~\ref{rmk:poset-flag}\ref{rmk:poset-flag-butterfly}. One can directly verify that $P$ is not tame by checking all $4!=24$ permutations of $P$. Alternatively, by \Cref{thm:poset-flag-tame}, which we prove below, $P$ cannot be tame because the poset flag variety in Remark~\ref{rmk:poset-flag}\ref{rmk:poset-flag-butterfly} is not smooth.
\end{example}

However, \emph{convex} subposets of tame posets are tame, as the next lemma shows. A subposet $P'\subeq P$ is \defn{convex} if for any $x\prec y\prec z$, $x,z\in P'$ implies $y\in P'$. We adopt the usual interval notation $[x,y]_P, (x,y)_P, (x,\infty)_P$, etc. 

\begin{lemma}\label{lem:traverse-restrict}
    If $P$ is a tame poset and $P'\subeq P$ is a convex subposet, then $P'$ is also tame. Moreover, any good traverse on $P$ restricts to a good traverse on $P'$, and with respect to these traverses, for $x\in P'$,
    \begin{equation}
        x^+_{P'}=\begin{cases}
            x^+_P,& x^+_P\in P';\\
            \infty, & x^+_P\notin P',
        \end{cases}
        \quad
        x^-_{P'}=\begin{cases}
            x^-_P,& x^-_P\in P';\\
            -\infty, & x^-_P\notin P'.
        \end{cases}
    \end{equation}
\end{lemma}
\begin{proof}
    Let $x_1,\dots,x_{\abs{P}}$ be a good traverse of $P$. Let $\set{i_1< \dots< i_{\abs{P'}}}$ be the set of all indices $i$ such that $x_i\in P'$, and consider the traverse $\set{x_{i_j}}_j$ of $P'$. Then for $1\leq j\leq \abs{P'}$, 
    \begin{itemize}
        \item If $x^+_{i_j,P}\in P'$, then $x^+_{i_j,P}$ is least among $\set{x_1,\dots,x_{i_j-1}}\cap (x_{i_j},\infty)_{P}$, and thus least among $\set{x_{i_1},\dots,x_{i_{j-1}}}\cap (x_{i_j},\infty)_{P'}$, so that $x^+_{i_j,P'}$ exists and $x^+_{i_j,P'} = x^+_{i_j,P}$.
        \item If $x^+_{i_j,P}=\infty$, then clearly $x^+_{i_j,P'}=\infty$.
        \item If $x^+_{i_j,P}\in P\setminus P'$, then we claim that there does not exist $1\leq k\leq j-1$ such that $x_{i_k}\succ x_{i_j}$. If it existed, then by the definition of $x^+_{i_j,P}$, we have $x_{i_k}\succeq x^+_{i_j,P}\succ x_{i_j}$. By convexity of $P'$, $x^+_{i_j,P}\in P'$, a contradiction. This proves the claim, and it follows that $x^+_{i_j,P'}=\infty$.
    \end{itemize}
    We thus prove the formula for $x^+_{P'}$. The assertion for $x^-_{P'}$ is proved analogously. The existence of $x^+_{P'}$ and $x^-_{P'}$ implies that $\set{x_{i_j}}_j$ is a good traverse of $P'$.
\end{proof}

\begin{example}\label{eg:gap-young}
    If $\gcd(a,b)=1$, then the gap poset $G$ is tame because it can be realized as a convex subposet of a grid poset. More precisely, by \eqref{eq:gap-young}, there is an order preserving bijection between $G$ and
    \[ \{(x,y)\in \{-b,\dots,-1\}\times \{-a,\dots,-1\}: ab+ax+by>0\}, \]
    which is a convex subposet of $\{-b,\dots,-1\}\times \{-a,\dots,-1\}$ that can be visualized as a northeast-gravitating Young subdiagram.
\end{example}

Finally, we prove a technical lemma whose use will be obvious later. 
Recall that for any finite poset $P$ and $x\prec y\in P$, the \defn{M\"obius function} $\mu_P(x,y)$ is defined inductively by
\begin{equation}
    \mu_P(y,y)=1,\quad \sum_{z\in [x,y]_P}\mu_P(z,y)=0 \text{ for }x\prec y.
\end{equation}
The M\"obius function also satisfies the dual recursion
\[ \sum_{z\in [x,y]_P} \mu_P(x,z)=\delta_{xy}=\begin{cases}
1, & x=y;\\
0, & x\neq y.
\end{cases}\]
\begin{lemma}\label{lem:mobius}
    Suppose $P$ is a finite tame poset with a good traverse $x_1,\dots,x_{\abs{P}}$. Consider the free abelian group spanned by formal elements $e_{x,y}$, where $x,y\in \hhat P:=P\sqcup \set{\pm \infty}$ and $x\prec y$. Then we have the following identity
    \begin{equation}\label{eq:traverse-mobius}
        \sum_{x\in P} e_{x^-,x^+}-e_{x,x^+}-e_{x^-,x} = e_{-\infty,\infty}+\sum_{\substack{x,y\in \hhat P\\x\prec y}} \mu_{\hhat P}(x,y) e_{x,y}.
    \end{equation}
    In particular, the left-hand side is independent of the choice of the traverse.
\end{lemma}
\begin{proof}
    We use induction on $\abs{P}$. As the base case, $P=\varnothing$, so $\hhat P=\set{\pm \infty}$ and $\mu_{\hhat P}(-\infty,\infty)=-1$. The identity is verified by $0=e_{-\infty,\infty}-e_{-\infty,\infty}$.

    Now for $\abs{P}=i>0$, consider $z:=x_i$ and $P':=\set{x_1,\dots,x_{i-1}}$. Then by induction hypothesis, 
    \begin{equation}
        \sum_{x\in P'} e_{x^-,x^+}-e_{x,x^+}-e_{x^-,x} = e_{-\infty,\infty}+\sum_{\substack{x,y\in \hhat {P'}\\x\prec y}} \mu_{\hhat {P'}}(x,y) e_{x,y}.
    \end{equation}
    (Note that for $x\in P'$, the notion of $x^+$ and $x^-$ is unambiguous, i.e., $x^+_{P'}=x^+_P$, etc.) So the desired identity is equivalent to
    \begin{equation}\label{eq:mobius-induction}
        e_{z^-,z^+}-e_{z,z^+}-e_{z^-,z} = \sum_{\substack{x,y\in \hhat {P}\\x\prec y}} \parens*{\mu_{\hhat {P}}(x,y)-\mu_{\hhat {P'}}(x,y) }e_{x,y},
    \end{equation}
    where for convenience we adopt the convention that if $(x,y)\notin \hhat {P'}\times \hhat {P'}$, then $\mu_{\hhat{P'}}(x,y):=0$.

    It remains to verify \eqref{eq:mobius-induction}. While it is possible to check each $e_{x,y}$-coefficient via a simple though lengthy case-by-case analysis, we present a uniform proof below. Work with the free abelian group generated by $e_{x,y}$ for $x,y\in \hhat P$ and $x\preceq y$. Observe that the right-hand side of \eqref{eq:mobius-induction} is equal to
    \begin{equation}
        \mathrm{RHS}\eqref{eq:mobius-induction}=-e_{z,z}+\sum_{\substack{x,y\in \hhat {P}\\x\preceq y}} \parens*{\mu_{\hhat {P}}(x,y)-\mu_{\hhat {P'}}(x,y) }e_{x,y}.
    \end{equation}
    Define a linear operator $\Phi$ by
    \begin{equation}
        \Phi(e_{x,y})=\sum_{\substack{w\in \hhat P\\w\preceq x}}e_{w,y}.
    \end{equation}
    Since the matrix for $\Phi$ in a suitable basis ordering is triangular with diagonal entries one, $\Phi$ is invertible. It thus suffices to verify \eqref{eq:mobius-induction} after applying $\Phi$.

    To compute $\Phi(\mathrm{RHS}\eqref{eq:mobius-induction})$, we compute
    \begin{align}
        \Phi\parens*{\sum_{\substack{x,y\in \hhat P\\x\preceq y}}\mu_{\hhat P}(x,y)e_{x,y}}&=\sum_{\substack{w,x,y\in \hhat P\\w\preceq x\preceq y}}\mu_{\hhat P}(x,y)e_{w,y} =\sum_{\substack{w,y\in \hhat P\\w\preceq y}}e_{w,y} \sum_{x\in [w,y]_{\hhat P}}\mu_{\hhat P}(x,y) \\
        &= \sum_{\substack{w,y\in \hhat P\\w\preceq y}}e_{w,y} \delta_{wy} = \sum_{y\in \hhat P} e_{y,y},
    \end{align}
    and
    \begin{align}
        \Phi\parens*{\sum_{\substack{x,y\in \hhat {P}\\x\preceq y}}\mu_{\hhat {P'}}(x,y)e_{x,y}}&=\sum_{\substack{w,x,y\in \hhat P\\w\preceq x\preceq y}}\mu_{\hhat {P'}}(x,y)e_{w,y} =\sum_{\substack{w,y\in \hhat P\\w\preceq y}}e_{w,y} \sum_{x\in [w,y]_{\hhat P}}\mu_{\hhat {P'}}(x,y) \\
        &=\sum_{y\in \hhat P}e_{z,y}\sum_{x\in [z^+,y]_{P'}}\mu_{\hhat{P'}}(x,y)+\sum_{\substack{w,y\in \hhat{P'}\\w\preceq y}}e_{w,y} \sum_{x\in [w,y]_{\hhat {P'}}} \mu_{\hhat{P'}}(x,y) \\
        &=\sum_{y\in \hhat P}e_{z,y}\delta_{z^+y}+\sum_{\substack{w,y\in \hhat{P'}\\w\preceq y}}e_{w,y} \delta_{wy}=e_{z,z^+}+\sum_{y\in \hhat{P'}} e_{y,y},
    \end{align}
    where we use $\mu_{\hhat{P'}}(x,y)=0$ if $x=z$ or $y=z$, and the fact that $(z,\infty]_{\hhat P}=[z^+,\infty]_{\hhat P}$ from the definition of good traverses. It thus follows that
    \begin{equation}
        \Phi(\mathrm{RHS}\eqref{eq:mobius-induction})=e_{z,z}-e_{z,z^+}-\Phi(e_{z,z})=-e_{z,z^+}-\sum_{\substack{w\in \hhat P\\w\prec z}} e_{w,z}.
    \end{equation}
    On the other hand,
    \begin{align}
        \Phi(\mathrm{LHS}\eqref{eq:mobius-induction})&=\sum_{w\in \hhat{P}}((1_{w\preceq z^-} - 1_{w\preceq z})e_{w,z^+}-1_{w\preceq z^-}e_{w,z}) \\
        &= \sum_{w\in \hhat{P}}(-\delta_{wz}e_{w,z^+}-1_{w\prec  z}e_{w,z})=-e_{z,z^+}-\sum_{\substack{w\in \hhat P\\w\prec z}} e_{w,z},
    \end{align}
    where we use $[-\infty,z)_{\hhat P}=[-\infty,z^-]_{\hhat P}$ from the definition of good traverses. This verifies $\Phi(\mathrm{LHS}\eqref{eq:mobius-induction})=\Phi(\mathrm{RHS}\eqref{eq:mobius-induction})$ and completes the proof of the lemma.
\end{proof}

\subsection{Generalized $q$-multinomial coefficients on posets}\label{subsec:multinomial}
We now set up a uniform framework for the generalized $q$-multinomial coefficients that appear in Theorem~\ref{thm:poset-flag} and indirectly in Theorem~\ref{thm:quot-zeta-tilde}. For a nonempty finite poset $(P,<)$ and a vector $(\mathbf{n};n)\in \Z^P\times \Z_{\geq 0}$, we extend it to a vector $\mathbf{n}\in \Z^{\hhat P}$ (where $\hhat P=P\sqcup \{\pm \infty\}$) by setting $n_\infty=n, n_{-\infty}=0$. Define the \defn{generalized $q$-multinomial coefficient} on $P$ as
\begin{equation}\label{eq:q-binom-def-poset}
    \qbinom{n}{\mathbf{n}}_{q;P}\coloneqq (q;q)_n \prod_{\substack{x,y\in \hhat P\\ x\prec y}} (q;q)_{n_y-n_x}^{\mu_{\hhat P}(x,y)} \in \Q(q)
\end{equation}
if $(\mathbf{n};n)$ is admissible, and zero otherwise. Note that $(\mathbf{n};n)$ is admissible if and only if the extended $\mathbf{n}$ satisfies $n_y-n_x\geq 0$ for all $x,y\in \hhat P, x\prec y$. 

\begin{lemma}
    \label{lem:q-binom-tame}
    If $P$ is a nonempty tame poset, then $\qbinom{n}{\mathbf{n}}_{q;P}$ is a product of $q$-binomial coefficients, and thus is in $\N[q]$.
\end{lemma}
\begin{proof}
    Fix a good traverse of $P$, and define $x^-, x^+$ with respect to it. Then by Lemma~\ref{lem:mobius}, the following identity is true whenever $(\mathbf{n};n)$ is admissible:
    \begin{equation}
        \label{eq:q-binom-tame}
        \qbinom{n}{\mathbf{n}}_{q;P} = \prod_{x\in P} \qbinom{n_{x^+}-n_{x^-}}{n_{x^+}-n_x}_q,
    \end{equation}
    because both sides are the same Laurent monomial in $(q;q)_{n_y-n_x}$. If $(\mathbf{n};n)$ is not admissible, then the left-hand side of \eqref{eq:q-binom-tame} is zero by definition. We claim the right-hand side is zero. Since the extended vector $\mathbf{n}\in \Z^{\hhat P}$ is not monotone increasing, there exists a covering relation $x\lessdot y$ such that $n_x>n_y$. Because $P\neq \varnothing$, the $e_{x,y}$-coefficient of the right-hand side of \eqref{eq:traverse-mobius} is $\mu_{\hhat P}(x,y)=-1$. As a result, there exists $z\in P$ such that $e_{z^-,z^+}-e_{z,z^+}-e_{z^-,z}$ contributes to negative coefficients of $e_{x,y}$, namely, either $(x,y)=(z,z^+)$ or $(x,y)=(z^-,z)$. In either case, $\qbinom{n_{z^+}-n_{z^-}}{n_{z^+}-n_z}_q=\qbinom{n_{z^+}-n_{z^-}}{n_z-n_{z^-}}_q=0$ because $n_y-n_x<0$, proving the claim. 
\end{proof}

\begin{example}
    If $P=[r]$, then $\qbinom{n}{\mathbf{n}}_{q;P}$ is the usual $q$-multinomial coefficient:
    \[ \qbinom{n}{\mathbf{n}}_{q;P} = \frac{(q;q)_n}{(q;q)_{n_1}(q;q)_{n_2-n_1}\cdots (q;q)_{n_r-n_{r-1}}(q;q)_{n-n_r}} = \qbinom{n}{n_1, n_2-n_1,\dots,n_r-n_{r-1},n-n_r}_q.\]
\end{example}
\begin{example}
    If $P=[w]\times [h]$ is the grid poset, then the only nonzero M\"obius values in $\hhat P$ occur at: 
    \begin{itemize}
        \item covering relations, namely, from $-\infty$ to the southwest corner $(1,1)$, from the northeast corner $(w,h)$ to $\infty$, and along unit edges on the grid, where the M\"obius value is $-1$; 
        \item unit diagonals, namely, from $x$ to $y$ with $y-x=(1,1)$, where the M\"obius value is $1$.
    \end{itemize} 
    We get an explicit formula for the $q$-multinomial coefficients on grid posets:
    \begin{equation}
    \label{eq:q-binom-grid}
        \qbinom{n}{\mathbf{n}}_{q;P}=\frac{(q;q)_n}{(q;q)_{n-n_{(w,h)}}}(q;q)_{n_{(1,1)}} \prod_{y-x=(1,0)\text{ or }(0,1)} (q;q)_{n_y-n_x}^{-1} \prod_{y-x=(1,1)}(q;q)_{n_y-n_x}.
    \end{equation}
\end{example}

Let $\gcd(a,b)=1$. We verify that the $q$-multinomial coefficient on the gap poset $G$ is given by \eqref{eq:q-binom-gap}, as promised. We start with a convenient lemma.

\begin{lemma}\label{lem:q-binom-restrict}
    Let $P$ be a nonempty finite poset, and let $U
    \subeq P$ be an upper set of $P$ (so $x\preceq y, x\in U$ implies $y\in U$). Suppose $\mathbf{n}\in \Z^P$ is a vector supported in $U$ (i.e., $\mathbf{n}|_{P\setminus U}=0$). Then for any $n\geq 0$, $$\qbinom{n}{\mathbf{n}}_{q;P}=\qbinom{n}{\mathbf{n}|_U}_{q;U}.$$
\end{lemma}
\begin{proof}
    We claim $\mathbf{n}$ is admissible if and only if $\mathbf{n}|_U$ is. The forward implication is trivial. For the backward implication, the only way the extension-by-zero $\mathbf{n}$ fails to be admissible is when $x\preceq y, x\in U, y\notin U, n_x>n_y=0$. But this is impossible because $U$ is an upper set. 
    
    Now it suffices to consider the case $\mathbf{n}$ is admissible. Define the surjection of posets $\phi:\hhat P\to \hhat U$ by
    \[\phi(x)=\begin{cases}
        x, & x\in U\cup \{\infty\},\\
        -\infty, & \text{else}.
    \end{cases}\]
    Then by comparing the factors, it suffices to prove that
    \begin{equation}\label{eq:mobius-restrict}
        \sum_{\substack{x,y\in \hhat P\\x\prec y}} \mu_{\hhat P}(x,y) e_{\phi(x),\phi(y)} = \sum_{\substack{x,y\in \hhat U\\x\prec y}} \mu_{\hhat U}(x,y) e_{x,y}
    \end{equation}
    in the free abelian group generated by $e_{x,y}$, $x,y\in \hhat U, x\prec y$, where we adopt the convention $e_{x,x}=0$. 

    We analyze the coefficients. There are two cases to consider. Let $U'=U\cup \{\infty\}$.
    \begin{enumerate}
        \item $e_{x,y}$ with $x,y\in U'$: the coefficient on the left-hand side is $\mu_{\hhat P}(x,y)$, while the coefficient on the right-hand side is $\mu_{\hhat U}(x,y)$. However, these M\"obius values are equal because $[x,y]_{\hhat P}=[x,y]_{\hhat U}$ are the same interval, again because $U$ is an upper set.
        \item $e_{-\infty,y}$ with $y\in U'$: we need to prove
        \[ \sum_{\substack{x\in \hhat P \setminus U'\\x\prec y}} \mu_{\hhat P}(x,y)=\mu_{\hhat U}(-\infty,y).\]
        By the definition of M\"obius functions on $\hhat P$ and on $\hhat U$, we have
        \begin{equation*}
            \sum_{x\in [-\infty,y]_{\hhat P}} \mu_{\hhat P}(x,y)=0
        \end{equation*}
        \begin{equation*}
            \sum_{x\in [-\infty,y]_{\hhat U}} \mu_{\hhat U}(x,y)=0
        \end{equation*}
        Subtracting these two equations and cancelling the contributions from $x\in U'$ (recall $\mu_{\hhat P}(x,y)=\mu_{\hhat U}(x,y)$ for $x,y\in U'$), we get
        \begin{equation*}
            \sum_{\substack{x\in \hhat P \setminus U'\\x\prec y}} \mu_{\hhat P}(x,y) - \sum_{x=-\infty} \mu_{\hhat U}(x,y) = 0,
        \end{equation*}
        which is exactly what we want.
    \end{enumerate}
\end{proof}

\begin{lemma}[cf.~\eqref{eq:q-binom-gap}]
    Let $\gcd(a,b)=1$ and $G$ be the gap poset. Then 
    \[ \qbinom{n}{\mathbf{n}}_{q;G}= \frac{(q;q)_n}{(q;q)_{n-n_f}} \prod_{i\in G} \frac{(q;q)_{n_i-n_{i-a-b}}}{(q;q)_{n_i-n_{i-a}}(q;q)_{n_i-n_{i-b}}},\]
    with the convention that $n_i=0$ for $i<0$, and if any $(q;q)_{-m}, m>0$ appears in the denominator, then the product is zero.\label{lem:q-binom-gap}
\end{lemma}
\begin{proof}
    Using Example~\ref{eg:gap-young}, embed $G$ as a northeast-gravitating Young diagram in a sufficiently large rectangular grid poset $P$ such that the vertices on the west side and the south side of $P$ are disjoint from $G$. Note that $G$ is an upper set of $P$, and the Frobenius element $f$ is the northeast corner of $P$. By Lemma~\ref{lem:q-binom-restrict}, 
    \[ \qbinom{n}{\mathbf{n}}_{q;G}=\qbinom{n}{\mathbf{n}}_{q;P},\]
    where $\mathbf{n}$ is extended by zero to a vector in $\Z^P$. By \eqref{eq:q-binom-grid} and that $\mathbf{n}$ vanishes on the west and south sides of $P$, we have
    \[ \qbinom{n}{\mathbf{n}}_{q;P} = \frac{(q;q)_{n}}{(q;q)_{n-n_f}}\prod_{x\in P} \frac{(q;q)_{n_x-n_{x-(1,1)}}}{(q;q)_{n_x-n_{x-(1,0)}}(q;q)_{n_x-n_{x-(0,1)}}},\]
    as the only mismatched factors are $(q;q)_{0-0}=1$. Finally, since $G$ is an upper set, for $x\in P\setminus G$, the corresponding factor in the product is $1$. The desired formula follows.
\end{proof}

\subsection{Poset flag variety on tame posets}
We are ready to prove the main geometric theorem of the section.
\begin{theorem}\label{thm:poset-flag-tame}
    Let $P$ be a finite nonempty tame poset. For $n\in \N$ and $\mathbf{n}\in \Z^P$, the poset flag variety $\Fl_P(\mathbf{n};n)$ is connected, smooth, and projective. It can be built as an iterated Grassmannian bundle over a point, and admits an affine paving. Moreover, the motive of $\Fl_P(\mathbf{n};n)$ is given by the generalized $q$-multinomial coefficient $\qbinom{n}{\mathbf{n}}_{\L;P}$.
\end{theorem}
\begin{proof}
    Consider a good traverse $x_1,\dots,x_{\abs{P}}$ and subposets $P_i=\{x_1,\dots,x_i\}$ for $0\leq i\leq \abs{P}$. For each $1\leq i\leq \abs{P}$, consider the forgetful morphism $\pi_i:\Fl_{P_i}(\mathbf{n};n)\to \Fl_{P_{i-1}}(\mathbf{n};n)$ (here we denote the restriction of $\mathbf{n}$ to subposets still by $\mathbf{n}$). Since the traverse is good, choosing a point in the fiber of $(V_{x_j})_{j\in [i-1]}\in \Fl_{P_{i-1}}(\mathbf{n};n)$ amounts to choosing a subspace $V_{x_i}$ of dimension $n_{x_i}$ such that $V_{x_i^-}\subeq V_{x_i}\subeq V_{x_i^+}$. As such, $\pi_i$ is the relative Grassmannian of a tautological quotient bundle:
    \[ \Fl_{P_i}(\mathbf{n};n) \simeq_{\pi_i} \mathcal{G}r_{\Fl_{P_{i-1}}(\mathbf{n};n)}(n_{x_i}-n_{x_i^-}, \calV_{x_i^+}/\calV_{x_i^-}).\]
    This realizes $\Fl_P(\mathbf{n};n)$ as a (possibly empty) iterated Grassmannian bundle over a point $\Fl_{P_0}(\mathbf{n};n)$, and thus $\Fl_P(\mathbf{n};n)$ is smooth and projective. Because the motive in $\KVar{\k}$ is multiplicative on fiber bundles, we have
    $$[\Fl_P(\mathbf{n};n)]=\prod_{x\in P} \qbinom{n_{x^+}-n_{x^-}}{n_{x^+}-n_x}_\L,$$
    which by \eqref{eq:q-binom-tame} is equal to $\qbinom{n}{\mathbf{n}}_{\L;P}$.

    Finally, the affine paving $\Fl_P(\mathbf{n};n)$ can be immediately constructed using the \BB decomposition. Consider the standard action of the standard torus $\mathbb{T}_n$ on the vector space $\k^n$. This induces a $\mathbb{T}_n$-action on $\Fl_P(\mathbf{n};n)$, whose fixed points are flags $V$ in which every $V_x$ is the span of a subset of the standard basis for $\k^n$. Therefore, the torus fixed point locus is finite. Since $\Fl_P(\mathbf{n};n)$ is smooth and projective, the \BB theorem implies that $\Fl_P(\mathbf{n};n)$ has an affine paving.
\end{proof}
It remains an open problem whether there exists a non-tame poset $P$ such that $\Fl_P(\mathbf{n};n)$ is smooth for all $(\mathbf{n};n)$.

\section{The Bia\l ynicki-Birula decomposition of the extension fiber and proof of Theorem~\ref{thm:bb}}\label{sec:bb}
Let $\k$ be an arbitrary field. Methods in this section are explicit and not sensitive to the characteristic of $\k$ and whether $\k$ is finite or algebraically closed. The reader can treat $\k$ as finite and focus on the counting aspect, viewing every moduli space as the set of $\k$-points. For geometric assertions, we implicitly work over the algebraic closure and work with the set of $\bbar{\k}$-points (as only the induced reduced structure is relevant throughout), and then note that every construction descends to the ground field.

\subsection{Torus action and initial terms}\label{sub:Torus action and initial terms}
Throughout this section, we fix $\gcd(a,b)=1$ and $R=\k[\![T^a,T^b]\!]\subeq \tl R=\k[\![T]\!]$. Recall $\Quot^R(\tl R^n)=\coprod_{m\geq 0}\Quot^R_m(\tl R^n)$ parametrizes finite-codimensional $R$-submodules of $\tl R^n$. For the method of this section, it is important to embed it into a bigger space, $\Quot^\k(\tl R^n)=\coprod_{m\geq 0} \Quot^\k_m(\tl R^n)$,
where $\Quot^\k_m(\tl R^n)$ parametrizes $\k$-subspaces $L\subeq \tl R^n$ such that
\[ T^{M}\tl R^n\subeq L\subeq \tl R^n,\]
where $M\coloneqq m+(a-1)(b-1)$ is chosen to be big enough that $\Quot^\k_m(\tl R^n)$ contains $\Quot^R_m(\tl R^n)$, see Lemma~\ref{lem:bounded-truncation}; we choose to suppress it in our notation because all that matters is that $\Quot^\k_m(\tl R^n)$ is now a Grassmannian (of the finite-dimensional vector space $\tl R^n/T^{M}\tl R^n$), hence smooth and projective.

There is a natural $\Gm$-action on $\Quot^\k_m(\tl R^n)$ induced by scaling $T$. More concretely, for $L\in \Quot^\k_m(\tl R^n)$,
\[ t\cdot L\coloneqq \set*{\mathbf{f}(tT): \mathbf{f}(T)\in L}.\]
Since $R$ is stable under the $\Gm$-action, the $\Gm$-action on $\Quot^\k_m(\tl R^n)$ restricts to a $\Gm$-action on $\Quot^R_m(\tl R^n)$.

Since $\Quot^\k_m(\tl R^n)$ is projective, for $L\in \Quot^\k_m(\tl R^n)$, the limit $\lim_{t\to 0} t\cdot L$ exists uniquely, and is a torus fixed point. We denote it by \[\pi(L)=\lim_{t\to 0} t\cdot L\in \Quot^\k_m(\tl R^n)^\Gm.\] Moreover, since $\Quot^R_m(\tl R^n)$ is closed in $\Quot^\k_m(\tl R^n)$, whenever $L\in \Quot^R_m(\tl R^n)$, we have $\pi(L)\in \Quot^R_m(\tl R^n)^\Gm$. The map $\pi$ is known as the \defn{\BB retraction map}, and we say $L$ is \defn{attracted} to $\pi(L)$.

As usual, $\pi(L)$ has a description in terms of Gr\"obner theory. Note that every element in $\tl R^n$ is uniquely of the form
\[ v=\sum_{d\geq 0} v_d T^d, \quad v_d\in \k^n.\]
For $d\geq 0$ and $L\in \Quot^\k_m(\tl R^n)$, define
\[ \init_d(L)\coloneqq \{v_d\in \k^n: v_d T^d+\sum_{e>d} v_e T^e \in L \text{ for some } v_{d+1},\dots \in \k^n\}.\]
It is obvious that $\init_d(L)$ is a subspace of $\k^n$. For each $L$, we have $\init_d(L)=\k^n$ for $d\gg 0$ (in fact, for $d\geq m+(a-1)(b-1)$). The \defn{initial subspace} of $L$ is defined as
\[ \init(L)\coloneqq \hhat\bigoplus_{d=0}^\infty \init_d(L) T^d\coloneqq \parens*{\bigoplus_{d=0}^{N-1} \init_d(L)} \oplus T^N \tl R^n,\]
where $N\gg 0$ is such that $\init_d(L)=\k^n$ for all $d\geq N$. It is a standard fact that $\pi(L)=\init(L)$ (Lemma~\ref{lem:BB-is-init}).

The torus fixed locus $\Quot^\k_m(\tl R^n)^\Gm$ consists of homogeneous subspaces:
\[ \Quot^\k_m(\tl R^n)^\Gm=\set*{\hhat\bigoplus_{d=0}^\infty V_d T^d: V_d\subeq \k^n, \; \sum_{d\geq 0}(n-\dim V_d)=m}.\]
We denote a typical element of $\Quot^\k_m(\tl R^n)^\Gm$ simply by $V=(V_d)_{d\geq 0}$ such that $V_d=\k^n$ for $d\geq M$. Since $\Quot^R_m(\tl R^n)^\Gm$ simply picks out those that are $R$-submodules, we have
\begin{equation}\label{eq:torus-fixed-R}
    \Quot^R_m(\tl R^n)^\Gm=\set*{V\in \Quot^\k_m(\tl R^n)^\Gm: V_d\subeq V_{d+a}\cap V_{d+b} \text{ for all }d},
\end{equation}
in other words, an injective flag on the poset $\N_{a,b}$ defined in \eqref{eq:Nab-poset}.

We now describe the extension fiber $E_R(\tl R^n;m)\subeq \Quot^R_m(\tl R^n)$ consisting of $L$ such that $\tl R L=\tl R^n$. We will show that $L$ is in the extension fiber if and only if $\pi(L)=(V_d)_{d\geq 0}$ satisfies $V_0=\k^n$ (Lemma~\ref{lem:ext-fiber}).

For an $\N_{a,b}$-poset flag $V$ in $\k^n$ such that $V_0=\k^n$, we automatically have $V_d=\k^n$ for $d\in \pairing{a,b}$. Therefore, we may identify $V$ with a flag on the gap poset $G$.
Since $G$ is tame (Example~\ref{eg:gap-young}), any poset flag variety on $G$ is connected, smooth, and projective. We have just concluded that the fixed point locus $E_R(\tl R^n)^\Gm=\coprod_{m\geq 0} E_R(\tl R^n;m)^\Gm$ has the following decomposition into (possibly empty) connected components, each of which is smooth projective:
\begin{equation}
    \label{eq:fixed-locus}
    E_R(\tl R^n)^\Gm=\coprod_{\mathbf{n}\in \Z^G} \Fl_G(\mathbf{n};n),
\end{equation}
and $E_R(\tl R^n;m)^{\Gm}$ is the disjoint union of those with $\sum_{i\in G} (n-n_i)=m$. (To get a decomposition into nonempty components, just recall $\Fl_G(\mathbf{n};n)$ is nonempty if and only if $(\mathbf{n};n)$ is admissible, see Proposition~\ref{prop:admissible}.)

By the general \BB decomposition \cite[Theorem 7.8.14, (1) and (3c)]{alpermoduli}, let
\[ X_{\mathbf{n}} = \pi^{-1}(\Fl_G(\mathbf{n};n)),\]
then the restricted map $\pi:X_{\mathbf{n}}\to \Fl_G(\mathbf{n};n)$ is a morphism, denoted by $\pi_\mathbf{n}$, and $E_R(\tl R^n)=\coprod_{\mathbf{n}} X_{\mathbf{n}}$ is a locally closed decomposition.\footnote{The criterion (3c) says if a $\Gm$-variety $X$ embeds equivariantly into $\P(V)$ for a $\Gm$-vector space $V$, then each \BB stratum $X_i$ is locally closed in $X$. This condition is verified in our case, because $E_R(\tl R^n;m)$ equivariantly embeds into a Grassmannian $\Quot^R_m(\tl R^n)$, which further equivariantly embeds into some $\P(V)$ via the Pl\"ucker embedding.}
Each $E_R(\tl R^n;m)$ is a disjoint union of strata $X_\mathbf{n}$ with $\sum_{i\in G} (n-n_i)=m$.

At this moment, we have proved Theorem~\ref{thm:bb} (and the detailed version Theorem~\ref{thm:bb-detailed}) except its crux, which is the assertion that $\pi_{\mathbf{n}}$ is an affine bundle of a specific rank. This is a hard statement: without the smoothness of $E_R(\tl R^n)$ (or at least the knowledge thereof), the affine bundle assertion does not follow from the general \BB Theorem. Nevertheless, we resort to an explicit description of $X_\mathbf{n}$, to be discussed in Section~\ref{sec:toric} after a short interlude.

\subsection{Proof of some technical lemmas}
Let $c=(a-1)(b-1)$; it is known as the \defn{conductor} of $R=\k[\![T^a,T^b]\!]$, and it satisfies $T^c \tl R\subeq R$.

\begin{lemma}
    For an $\tl R$-submodule $L\in \Quot^{\tl R}_m(\tl R^n)$, we have $L\supeq T^m \tl R^n$.
\end{lemma}
\begin{proof}
    This is well-known, and we present two proofs. First, by the theory of Smith normal form, up to change of coordinate, $L=\im(\diag(T^{\lambda_1},\dots,T^{\lambda_n}))$, where $\sum_i \lambda_i=m$. In particular, $\max_i \lambda_i \leq m$, so $L\supeq T^m \tl R^n$. Second, we consider the quotient module $\tl R^n/L$. Its length is $m$, so by Nakayama's lemma, it is annihilated by $T^m$. 
\end{proof}

\begin{lemma}\label{lem:bounded-truncation}
    For an $R$-submodule $L\in \Quot^R_m(\tl R^n)$, we have $L\supeq T^{m+c}\tl R^n$. As a result, $\Quot^R_m(\tl R^n)\subeq \Quot^\k_m(\tl R^n)$.
\end{lemma}
\begin{proof}
    Since $L$ is an $R$-module and $T^c\tl R\subeq R$, we have $L=RL\supeq T^c \tl RL$. Let $m'=\dim_\k \tl R^n/\tl RL$, then we have $m'\leq m$. By the previous lemma, $\tl RL\supeq T^{m'}\tl R^n$, so $L\supeq T^{m'+c}\tl R^n\supeq T^{m+c}\tl R^n$.
\end{proof}

\begin{lemma}\label{lem:BB-is-init}
    For $L\in \Quot^\k_m(\tl R^n)$, we have
    \[ \pi(L)=\init(L).\]
    In particular, $\sum_{d\geq 0} (n-\dim \init_d(L)) = m$. 
\end{lemma}
\begin{proof}
    This is a standard fact about classical Grassmannians of the graded vector space
    \[\tl R^n/T^{m+c}\tl R^n = \bigoplus_{i=0}^{m+c-1} \k^n T^i,\]
    where $\k^n T^i$ is the vector space $\k^n$ with degree $i$ (i.e., with $\Gm$-action of weight $i$). 
\end{proof}

\begin{lemma}\label{lem:ext-fiber}
    Let $L\in \Quot^R_m(\tl R^n)$ and $\pi(L)=(V_d)_{d\geq 0}$. Then $L\in E_R(\tl R^n;m)$ if and only if $V_0=\k^n$. 
\end{lemma}
\begin{proof}
    By Nakayama's lemma, elements of $L$ generate $\tl R^n$ if and only if they generate $\tl R^n/T\tl R^n$. But the image of an element $v$ in $\tl R^n/T\tl R^n\simeq \k^n$ is just the constant term of $v$, so the set-theoretical image of $L$ in $\tl R^n/T\tl R^n$ is precisely $\init_0(L)$, a subspace. As a result, $L$ spans $\tl R^n/T\tl R^n$ if and only if $\init_0(L)=\k^n$, so we are done because $V_0=\init_0(L)$ by Lemma~\ref{lem:BB-is-init}.
\end{proof}

\subsection{Fibers of the \BB map}
\label{sec:toric}
As our first step to understand the \BB retraction $\pi_{\mathbf{n}}$, we look at its fiber through the lens of Gr\"obner theory. 

Our strategy is analogous to the methods of \cite{OblomkovRasmussenShende12,Piontkowski06}. In particular, in the case \(n=1\)
we reprove the results of the above-mentioned papers. However, even in these cases the combinatorial part of the argument differs from the
argument from the mentioned papers. The main similarity is in our use of the deformation theory, as we detail below.

Fix $V\in \Fl_G(\mathbf{n};n)$, and let $m=\sum_{i\in G} (n-n_i)$. For clarity, we denote by $\pi_\k=\init_\k$ the \BB map on $\Quot^\k(\tl R^n)$, and by $\pi_R=\init_R$ its restriction on $\Quot^R(\tl R^n)$. By Lemma~\ref{lem:ext-fiber}, for $L\in \Quot^R(\tl R^n)$, $\pi(L)=V$ automatically implies $L\in E_R(\tl R^n)$, so
\[ \pi_{\mathbf{n}}^{-1}(V)=\pi_R^{-1}(V).\]
To set up our framework, we view
\[ \pi_R^{-1}(V)=\pi_\k^{-1}(V)\cap \Quot^R(\tl R^n),\]i.e., the set of $\k$-vector spaces in $\pi_\k^{-1}(V)$ that are $R$-submodules. Recall that $\pi_\k^{-1}(V)$ consists of vector spaces $L$ between $T^{M}\tl R^n$ and $\tl R^n$ such that $\init(L)=V$. We now describe a natural \defn{Gr\"obner coordinate system} of these subspaces $L$, identifying $\pi_\k^{-1}(V)$ as an affine space. This will be reminiscent of Haiman's description of Hilbert scheme of points on the plane \cite{haiman}. After that, we express the condition of being an $R$-module into equations among the Gr\"obner coordinates.

Now let $V\in \Quot^\k_m(\tl R^n)^{\Gm}$. The forthcoming construction depends on the choice of a collection of subspaces $W=(W_d)_{d\geq 0}$ such that $W_d$ is complementary to $V_d$, or $W_d\oplus V_d=\k^n$. Note that $W_d=0$ for $d\notin G$. The decomposition induces a natural isomorphism $W_d\simeq \k^n/V_d$, so the injective $\N_{a,b}$-flag structure of $V$ endows $W$ with a surjective $\N_{a,b}$-flag structure. More concretely, for $0\leq d\leq e$ such that $e-d\in \set{a,b}$, the natural surjection $W_d\onto W_e$ is given by the composition map
\begin{equation}\label{eq:w-surj-flag}
    W(d\to e): W_d\incl \k^n\onto W_e,
\end{equation}
where the last projection is with respect to $\k^n=W_e\oplus V_e$.

We also view $W$ as a finite-dimensional homogeneous subspace of $\tl R^n$ by $W=\oplus_d W_d T^d$. Note that $\tl R^n=V\oplus W$ and $\dim_\k W=m$. Define $U^\k_W$ to be the open subset of $\Quot^\k_m(\tl R^n)$ consisting of $L$ such that $W\oplus L=\tl R^n$. From now on, we fix $m$ and work modulo $T^{M}$ throughout. Then denoting $\bbar{\tl R^n}=\tl R^n/T^{M} \tl R^n$, $U^\k_W$ is identified with the space of $\bbar{L}\subeq_\k \bbar{\tl R^n}$ such that 
$W\oplus \bbar{L} = \bbar{\tl R^n}$. This is the biggest Schubert cell of a Grassmannian, and admits the following identification with an affine space: 
\begin{equation}\label{eq:u_w}
    U^\k_W\simeq \{ h\in \Hom_\k(\bbar{\tl R^n},W): h|_W=1\}
\end{equation}
where the correspondence $L\leftrightarrow h$ is given by $L=\ker h=\im(1-h)$ and $h$ is the projection map $\bbar{\tl R^n}=W\oplus \bbar L\to W$. We call $h$ the \defn{Gr\"obner coordinate} of $\bbar L$ with respect to $W$.

Consider an affine subspace $U^\k_{W,V}$ of $U^\k_W$, defined via the above identification by a triangularity condition
\begin{equation}\label{eq:u_wv}
    U^\k_{W,V}\coloneqq \{ h\in U^\k_W: h(V_d T^d)\subeq \sum_{e=d+1}^{M-1} W_e T^e \text{ for }0\leq d<M\}.
\end{equation}
To summarize the constructions so far, we have a locally closed embedding
\[
\begin{array}{rcccl}
    U^\k_{W,V} & \overset{\text{closed}}{\incl} & U^\k_W &\overset{\text{open}}{\incl} &\Quot^\k_m(\tl R^n),\\
    h & \mapsto & h &\mapsto &\ker(h).
\end{array}
\]

\begin{lemma}
    Let $V\in \Quot^\k_m(\tl R^n)^{\Gm}$, then for any choice of $W=(W_d)_{d\geq 0}$ with $W_d\oplus V_d=\k^n$, a subspace $L\in \Quot^\k_m(\tl R^n)$ satisfies $\init(L)=V$ if and only if $L\in U^\k_{W,V}$. Hence
    \[ \init_\k^{-1}(V)=U^\k_{W,V}.\]
\end{lemma}
\begin{proof}
    Suppose $\bbar L=\ker(h)=\im(1-h)$ for $h\in U^\k_{W,V}$. Then $\im(1-h)$ consists of linear combinations of elements of the form
    \[ v_d T^d - h(v_d T^d): v_d \in V_d\setminus \{0\}, 0\leq d<M.\]
    By the triangularity condition, any such linear combination has initial term $v_d T^d$ for some $v_d\in V_d$. As a result, $\init_d(L)=V_d$, so $\init(L)=V$.

    Conversely, suppose $L\in \init_\k^{-1}(V)$. We claim that for any $f(T)\in \bbar{\tl R^n}$, there exists $w(T)\in W$ such that $f(T)\equiv w(T) \pmod{\bbar L}$, and if $f(T)\in V_dT^d$, then we can make $w(T)\in \sum_{e>d} W_e T^e$. If the claim is proved, then the composition $W\incl \bbar{\tl R^n}\onto \bbar{\tl R^n}/\bbar{L}$ is surjective, thus an isomorphism for dimension reasons. Hence $w(T)$ is uniquely determined by $f(T)$, $W\oplus \bbar{L}=\bbar{\tl R^n}$, and the corresponding map $h$ is given by $h(f(T))=w(T)$. The final assertion of the claim implies $h\in U^\k_{W,V}$.
    
    The proof of the claim is essentially the division algorithm: start with $f(T)=\sum_{d=0}^{M-1} f_d T^d$, and repeatedly do the following. If $f(T)\in W$, terminate. Otherwise, let 
    \[ e=\min \{ d<M: f_d \notin W_d\}.\]
    Decompose $f_{e}=v_{e}+w_{e}$ for $v_{e}\in V_{e}, w_{e}\in W_{e}$. Since $\init_e(L)=V_e$, there exists an element $r(T)=v_e T^e + O(T^{e+1})$ in $L$. Then replace $f(T)$ by $f(T)-r(T)$. 
    
    This process replaces $f(T)$ with an element of $f(T)+L$, kills the lowest-degree $V$-component of $f(T)$, and only introduces higher-degree terms. It is then obvious that the algorithm terminates (as $e$ strictly increases) and the resulting $f(T)$ when the algorithm terminates can be used as $w(T)$.
\end{proof}
\begin{remark}
    This lemma implies $U^\k_{W,V}$ is independent of the choice of $W$, when viewed as a subset of $\Quot^\k_m(\tl R^n)$. However, the identification with an affine space (the Gr\"obner coordinate) depends on the choice of $W$.
\end{remark}

Now assume $V\in \Quot^R_m(\tl R^n)^\Gm$. We shall write down the equations on Gr\"obner coordinates that cut out the locus $\init^{-1}_R(V) \subeq \init^{-1}_\k (V)$. This is essentially the Buchberger criterion.
\begin{lemma}
    Let $h\in U^\k_{W,V}$, viewed as an element of $\End_\k(\bbar{\tl R^n})$. Then the corresponding element $L\in \Quot^\k_m(\tl R^n)$ is an $R$-module if and only if
    \[ hT^\gamma(1-h)=0, \gamma=a,b,\]
    where $T^\gamma\in \End_\k(\bbar{\tl R^n})$ is multiplication by $T^\gamma$.
\end{lemma}
\begin{proof}
    Recall $L=\ker(h)=\im(1-h)$. Hence, $L$ is an $R$-module if and only if $T^a L\subeq L$ and $T^b L\subeq L$, namely, 
    \[ T^\gamma \im(1-h)\subeq \ker(h) \text{ for }\gamma=a,b.\]
    This directly translates to the equations claimed.
\end{proof}

Hence, $\init^{-1}_R(V)$ is isomorphic to
\begin{equation}\label{eq:groebner-equation}
    U^R_{W,V}=\{ h\in U^\k_{W,V}: hT^a(1-h)=hT^b(1-h)=0\}.
\end{equation}

\subsection{Deformation argument}
Our main geometric argument to solve the equation \eqref{eq:groebner-equation} relies on the deformation theory construction. We first make a change of coordinates. Let $1_W, 1_V\in \End_\k(\bbar{\tl R^n})$ be the projections onto $W$ and $\bbar{V}$ with respect to the decomposition $\bbar{\tl R^n}=W\oplus \bbar{V}$. For a Gr\"obner coordinate $h\in U^\k_{W,V}$, define the \defn{centralized} Gr\"obner coordinate with respect to $W,V$ by
\[ \phi=h-1_W\in \Hom_\k(\bbar{\tl R^n},W),\]
then the defining condition \eqref{eq:u_w} for $U^\k_W$, namely, $h|_W=1$, is equivalent to $\phi|_W=0$. Identify $\phi$ with an element in $\Hom_\k(\bbar{V},W)$. The triangularity condition \eqref{eq:u_wv} defining $U^\k_{W,V}$ can be neatly translated as follows. Since $\bbar{V}$ and $W$ are $\Gm$-equivariant, there is a $\Gm$-action on $\Hom_\k(\bbar{V},W)$, inducing a $\Z$-grading. Then \eqref{eq:u_wv} is equivalent to
\[ \phi\in \Hom_\k(\bbar{V},W)_{>0},\]
the positive-degree part of $\Hom_\k(\bbar{V},W)$. Hence, noting that $1-h=(1-1_W)-\phi=1_V-\phi$, we may identify
\begin{equation}\label{eq:groebner-equation-centralized}
    U^R_{W,V}=\{ \phi\in \Hom_\k(\bbar{V},W)_{>0}: (1_W+\phi)T^\gamma (1_V-\phi)=0\in \End_\k(\bbar{\tl R^n})\text{ for }\gamma=a,b\}.
\end{equation}
A centralized Gr\"obner coordinate $\phi$ determines a submodule $L\in \Quot^R_m(\tl R^n)$ (identified modulo $T^M$ with $\bbar L\subeq \bbar{\tl R^n}$) by $\bbar L=\ker(1_W+\phi)=\im(1_V-\phi)$. 

To analyze the equation~\eqref{eq:groebner-equation-centralized}, we use a Koszul-type complex for spaces of $\k$-morphisms of $R$-modules. Given $R$-modules $V,W$, we define the Koszul complex $K^\bullet(V,W)$ (with respect to the generators $T^a,T^b$) by $K^i(V,W)=\Hom_\k(V,W)^{\oplus \binom{2}{i}}$, with differentials
\begin{equation}\label{eq:Kosz-vanishing}
0\to \Hom_\k(V,W)\xrightarrow{D_1}\Hom_\k(V,W)\oplus\Hom_\k(V,W)\xrightarrow{D_2}\Hom_\k(V,W)\to 0
\end{equation}
given by the formula
\begin{equation}\label{eq:differentials-formula}
D_1=\begin{bmatrix}d_a\\ d_b\end{bmatrix}, \quad D_2=\begin{bmatrix}-d_b& d_a\end{bmatrix}, \quad d_\gamma(\phi)=\phi \, T^\gamma|_V-T^\gamma|_W \, \phi,\quad \gamma=a,b.
\end{equation}
In the above formula $T^\gamma|_V, T^\gamma|_W$ are the multiplication by $T^\gamma$ endomorphisms. We note that $K^\bullet(V,W)$ is indeed a complex because $d_a$ and $d_b$ commute:
\[ d_a d_b(\phi)=d_a(\phi \, T^b|_V-T^b|_W \, \phi)=\phi \, T^{a+b}|_V - T^a|_W \, \phi \, T^b|_V-T^b|_W \, \phi \, T^a|_V + T^{a+b}|_W \, \phi=d_b d_a (\phi),\]
in which the rewriting $T^a T^b=T^{a+b}$ implicitly captures the fact that $T^a$ and $T^b$ commute.

Let us assume that $V,W$ are $\Gm$-equivariant with $\Gm$-action compatible with the $\Gm$-action on $R$; in other words, the endomorphisms $T^\gamma|_V$ and $T^\gamma|_W$ are homogeneous of degree $\gamma$. Then note that $d_\gamma$ is homogeneous of degree $\gamma$, so the complex $K^\bullet(V,W)$ breaks down degree-by-degree into
\begin{equation}
    K^\bullet(V,W)_r: 0\to \Hom_\k(V,W)_r \to \Hom_\k(V,W)_{r+a} \oplus \Hom_\k(V,W)_{r+b} \to \Hom_\k(V,W)_{r+a+b} \to 0.
\end{equation}

\begin{lemma}\label{lem:koszul-vs-fundamental}
    Let $V\in \Quot^R_m(\tl R^n)^\Gm$. Let $M=m+(a-1)(b-1)$ and $\bbar{V}=V/T^M \tl R^n$. Let $W=\oplus_{d=0}^{M-1} W_d T^d$ with $W_d\oplus V_d=\k^n$. Then for $r\geq 0$, the complex $K^\bullet(\bbar{V},\tl R^n/V)_r$ is isomorphic to the complex $C_{\N_{a,b}}^\bullet(V[-r],W;Q_2)$ in \eqref{eq:fundamental-cochain-complex}, where
    \begin{itemize}
        \item $V$ is viewed as an injective $\N_{a,b}$-flag \eqref{eq:torus-fixed-R}.
        \item $V[-r]$ is the shift defined by $V[-r]_d\coloneqq V_{d-r}$, where $V_{d-r}\coloneqq 0$ if $d-r<0$.
        \item $W$ is viewed as a surjective $\N_{a,b}$-flag \eqref{eq:w-surj-flag}.
        \item $Q_2$ consists of ``counterclockwise''\footnote{We visualize $+a$ as the east and $+b$ as the north.} relations
        \[ [d\to d+a\to d+a+b]-[d\to d+b\to d+a+b], d\geq 0.\]
    \end{itemize}
\end{lemma}
\begin{proof}
    We equip $W$ with an $R$-module structure using the $\k$-linear isomorphism $W\simeq \tl R^n/V$. As such, $K^\bullet(\bbar{V},\tl R^n/V)_r$ is isomorphic to $K^\bullet(\bbar{V},W)_r$. 
    
    We first match $i$-th degrees of the complexes. For $i=0$, we have
    \[ \Hom_\k(\bbar{V},W)_r = \oplus_{r\leq d<M} \Hom_\k(V_{d-r},W_d) = \oplus_{d\geq 0} \Hom_\k(V_{d-r},W_d)=C^0(V[-r],W),\]
    where the second equality is because $\Hom_\k(V_{d-r},W_d)$ vanishes unless $r\leq d<M$. (Recall for $d\geq M$, $V_d=\k^n$, so $W_d=0$.) For $i=1$, we similarly have
    \[ \Hom_\k(\bbar{V},W)_{r+a}\oplus \Hom_\k(\bbar{V},W)_{r+b} = \oplus_{d\geq 0} \Hom_\k(V[-r]_d,W_{d+a}) \oplus \Hom_\k(V[-r]_d,W_{d+b}).\]
    Since the arrows of $\N_{a,b}$ consist of $d\to d+a$ and $d\to d+b$ for $d\geq 0$, this space is precisely $C^1(V[-r],W)$. For $i=2$, the identification $\Hom_\k(\bbar{V},W)_{r+a+b}=C^2(V[-r],W)$ follows analogously.

    We then match the differentials. For an element $\phi\in C^0(V[-r],W)$, we denote its components by $\phi_d\in \Hom_\k(V_{d-r},W)$. For an element $ \rho\in C^1(V[-r],W)$, we denote its components by $\rho_{d\to e} \in \Hom_\k(V_{d-r},W_e)$, where $e-d=a,b$. Isolating the parts of correct degrees, the differential $D_1$ of $K^\bullet(\bbar V,W)_r$ involves
    \[ d_\gamma: \oplus_{d\geq 0} \Hom_\k(V_{d-r},W_d) \to \oplus_{d\geq 0} \Hom_\k(V_{d-r},W_{d+\gamma}),\]
    with 
    \[(d_\gamma \phi)_{d\to d+\gamma}=\phi_{d+\gamma} \, T^\gamma|_{V[-r]_d}-T^\gamma|_{W_d} \, \phi_d\in \Hom_\k(V_{d-r},W_{d+\gamma}).\] 
    The map $T^\gamma$ on $V[-r]_d$ is simply the inclusion map $V[-r](d\to d+\gamma):V_{d-r}\incl V_{d+\gamma-r}$. Since the $R$-module structure of $W$ is defined through $\tl R^n/V$, the map $T^\gamma|_{W_d}$ is precisely the map $W(d\to d+\gamma)$ from the surjective flag structure of $W$, see \eqref{eq:w-surj-flag}. As a result, $D_1=d_a+d_b$ is the first differential of $C^\bullet(V[-r],W)$.

    Similarly, for $\rho\in C^1(V[-r],W)$, consider a component of $d_\gamma \rho \in \Hom_\k(\bbar{V},W)_{r+a+b}$, namely,
    \[ (d_\gamma \rho)_{d\to d+a+b} \in \Hom_\k(V_{d-r},W_{d+a+b}).\]
    We have
    \[ (d_\gamma \rho)_{d\to d+a+b} = \rho_{d+\gamma\to d+a+b} \, T^\gamma|_{V[-r]_d} - T^\gamma|_{W_{d+a+b-\gamma}} \, \rho_{d\to d+a+b-\gamma},\]
    so 
    \begin{align*}
        (D_2 \rho)_{d\to d+a+b} &=  (d_a \rho)_{d\to d+a+b} - (d_b \rho)_{d\to d+a+b}\\
        &=\rho_{d+a\to d+a+b} \, T^a|_{V[-r]_d} - T^a|_{W_{d+b}} \, \rho_{d\to d+b} - \rho_{d+b\to d+a+b} \, T^b|_{V[-r]_d} + T^b|_{W_{d+a}} \, \rho_{d\to d+a}.
    \end{align*}
    Again, since $T^\gamma$ on $V[-r]_d$ and $W_d$ correspond to the structure maps of the flags, this matches the second differential of $C^\bullet(V[-r],W)$ induced by the counterclockwise rule $Q_2$.
\end{proof}

\begin{theorem}\label{thm:bb-fiber}
    Let $V\in \Fl_G(\mathbf{n};n)$, identified with a point of $E_R(\tl R^n)^\Gm$ by \eqref{eq:fixed-locus}. Then $\init_R^{-1}(V)$ is an affine space of dimension $B(\mathbf{n};n)$, where
    \begin{equation}\label{eq:bb-fiber-dim}
        B(\mathbf{n};n)=\sum_{e\in G, d\in \Z} U(e-d)\, n_d (n-n_e),
    \end{equation}
    with $U(r)=1_{r>0}-1_{r>a}-1_{r>b}+1_{r>a+b}$, $n_d=0$ for $d<0$, and $n_d=n$ for $d\in \pairing{a,b}$.
\end{theorem}
\begin{proof}
    We study the system \eqref{eq:groebner-equation-centralized} degree by degree; we use the notation ${}_r(\cdots)$ to extract the degree $r$ part of an element of a graded vector space. Note that $1_W, 1_V$ are homogeneous of degree $0$ and  $T^\gamma$ is homogeneous of degree $\gamma$. For $\gamma=a,b$, the lowest-degree part of the equation $(1_W+\phi)T^\gamma (1_V-\phi)=0$, namely, the degree-$\gamma$ part
    \[ 1_W T^\gamma 1_V=0,\]
    is a trivial consequence of the fact that $V$ is an $\N_{a,b}$-flag. The first nontrivial equation is the next degree
    \[ ({}_1 \phi) T^\gamma 1_V-1_W T^\gamma ({}_1 \phi)=0,\]
    and more generally for $r>0$,
    \begin{equation}\label{eq:groebner-equation-by-degree}
        ({}_r \phi) T^\gamma 1_V - 1_W T^\gamma ({}_r \phi)={}_{r+\gamma} (\phi T^\gamma \phi).
    \end{equation}
    Note that $\phi$ has only positive degree parts, so ${}_{r+\gamma} (\phi T^\gamma \phi)$ depends only on ${}_1 \phi,\dots,{}_{r-1} \phi$, and we may solve for ${}_r \phi$ inductively. We view ${}_r\phi$ in $\Hom_\k(\bbar{V},W)_r=K^0(\bbar{V},W)_r$, where $W\simeq \tl R^n/V$ is given the $R$-module structure as in the proof of Lemma~\ref{lem:koszul-vs-fundamental}. The $R$-module structure on $W$ can be concretely given by $T^\gamma|_W=1_W T^\gamma$.

    Now for each fixed $r>0$, consider the system of equations by putting two equations of \eqref{eq:groebner-equation-by-degree} for $\gamma=a,b$ together. The left-hand side is precisely $D_1 ({}_r \phi)$. We denote the right-hand side by ${}_r \rho\in K^1(\bbar{V},W)_r$, so that the system of equations \eqref{eq:groebner-equation-by-degree} for $\gamma=a,b$ becomes
    \begin{equation}\label{eq:groebner-equation-complex}
        D_1({}_r \phi)={}_r \rho.
    \end{equation}
    At this point, we have proved Proposition~\ref{prop:bb-fiber}, using the identification $K^\bullet(\bbar{V},W)_r=C_{\N_{a,b}}^\bullet(V[-r],W)$ in Lemma~\ref{lem:koszul-vs-fundamental}. 

    Let us verify that $D_2 ({}_r \rho)=0$ assuming
    \begin{equation}
        \label{eq:induction-step}
        D_1({}_s \phi)={}_s \rho\text{ for }1\leq s<r.
    \end{equation}
    If $r=1$, ${}_1\rho=0$ and this is trivially true. More generally, for $r>0$, it is more convenient to consider all degrees together up to a certain bound. Introduce the notation $O(T^r)$: an element is in $O(T^r)$ if it only has parts of degree at least $r$. Consider
    \[ \rho=(\phi T^a \phi, \phi T^b \phi)\in \Hom_\k(\bbar{V},W)^{\oplus 2}.\]
    Then to prove $D_2({}_r \rho)=0$ assuming  \eqref{eq:induction-step}, it suffices to prove
    \[ D_2(\rho)=O(T^{r+a+b+1})\]
    assuming
    \[ D_1(\phi)=\rho+(O(T^{r+a}),O(T^{r+b})),\]
    which means
    \begin{equation}\label{eq:commute}
        \phi T^\gamma 1_V- 1_W T^\gamma \, \phi=\phi T^\gamma \phi+O(T^{r+\gamma})\text{ for }\gamma=a,b.
    \end{equation}

    Now compute
    \begin{align*}
        D_2(\rho)&=d_a(\phi T^b \phi)-d_b(\phi T^a \phi)\\
        &=\phi T^b \phi T^a 1_V - 1_W T^a \phi T^b \phi - \phi T^a \phi T^b 1_V + 1_W T^b \phi T^a \phi.
    \end{align*}
    We apply the commutation \eqref{eq:commute} to all four terms of the right-hand side, such as 
    \[\phi T^b \phi T^a 1_V = \phi T^b (1_W T^a \phi+ \phi T^a \phi+ O(T^{r+a}))=\phi T^b 1_W T^a \phi + \phi T^b \phi T^a \phi + O(T^{r+a+b+1}).\]
    This yields
    \begin{align*}
        D_2(\rho)+O(T^{r+a+b+1})&=\phi T^b 1_W T^a \phi + \phi T^b \phi T^a \phi\\
        &-\phi T^a 1_V T^b \phi + \phi T^a \phi T^b \phi \\
        &-\phi T^a 1_W T^b \phi - \phi T^a \phi T^b \phi \\
        &+\phi T^b 1_V T^a \phi - \phi T^b \phi T^a \phi \\
        &=\phi T^b(1_W+1_V) T^a \phi - \phi T^a (1_W+1_V) T^b \phi\\
        &=\phi T^b T^a \phi- \phi T^a T^b \phi = 0,
    \end{align*}
    where the last line uses $1_W+1_V=1$ and that $T^a$ and $T^b$ commute. This verifies the claim.

    Now we invoke the assumption that $V\in E_R(\tl R^n;m)^{\Gm}$, so that the poset flag $W$ is supported on the gap set $G$. We identify $K^\bullet(\bbar{V},W)_r=C^\bullet_{\N_{a,b}}(V[-r],W)$ by Lemma~\ref{lem:koszul-vs-fundamental}. Further, we notice that $C^\bullet_{\N_{a,b}}(V[-r],W)$ is identical to the restricted complex $C^\bullet_{G}(V[-r],W)$. Indeed, a typical summand of $C^\bullet_{\N_{a,b}}(V[-r],W)$ is of the form $\Hom_\k(V[-r]_d,W_e)$, where $e-d\in \{0,a,b,a+b\}\subeq \pairing{a,b}$. If it does not vanish, then $d\geq r>0$ and $e\in G$, so the fact that $e-d\in \pairing{a,b}$ implies $d\in G$. 

    We further embed $G$ into a grid poset $Q$ such that $G$ is an upper set of $Q$ (Example~\ref{eg:gap-young}). Extend the $G$-flags $V[-r]$ and $W$ to $Q$-flags by $V[-r]_x=0$ and $W_x=\k^n$ for $x\in Q\setminus G$. Because $G$ is an upper set of $Q$, $V[-r]|_Q$ is an injective flag and $W|_Q$ is a surjective flag, whose additional structural maps are either the identity map on $\k^n$ or the natural projection $\k^n\to W_d$ with respect to $\k^n=W_d\oplus V_d$. We notice once again that $C^\bullet_G(V[-r],W)$ is identical to the extended complex $C^\bullet_Q(V[-r],W)$. Indeed, a typical summand of $C^\bullet_Q(V[-r],W)$ is of the form $\Hom_\k(V[-r]_x,W_y)$, where $y-x\in \{0,1\}^2$. If it does not vanish, then $x\in G$, which forces $y\in G$ because $G$ is an upper set of $Q$.

    We have thus established $K^\bullet(\bbar{V},W)_r\simeq C^\bullet_Q(V[-r],W)$, which is exact at degrees $1$ and $2$ by Theorem~\ref{thm:acyclicity}. As a result, the solution space of \eqref{eq:groebner-equation-complex} is nonempty and is an affine space of dimension
    \[ \dim \Hom_Q(V[-r],W) = \sum_{i=0}^2 (-1)^i \dim C^i_Q(V[-r],W) = \sum_{i=0}^2 (-1)^i \dim C^i_G(V[-r],W).\]
    Modulo the geometric detail (Lemma~\ref{lem:trivial-aff-bundle}), this implies that $\init^{-1}_R(V)$ is an affine space with dimension
    \begin{align*}
    \sum_{r>0} \dim \Hom_Q(V[-r],W) &= \sum_{r>0}\sum_{i=0}^2 (-1)^i \dim C^i_G(V[-r],W)\\
    &=\sum_{r>0} \sum_{e\in G} (n_{e-r}-n_{e-a-r}-n_{e-b-r}+n_{e-a-b-r}) (n-n_e),
    \end{align*}
    where $n_d=0$ for $d<0$ and $n_d=n$ for $d\in \pairing{a,b}$.
    
    This simplifies to the required expression because for each $e$, taking the second term as an example,
    \[ \sum_{r>0} n_{e-a-r}(n-n_e) = \sum_{d\in \Z} 1_{e-d>a}\, n_d\, (n-n_e).\qedhere\]
\end{proof}

\subsection{Final geometric details}
With the notation in the proof of Theorem~\ref{thm:bb-fiber}, consider the affine variety
\[ X_r\coloneqq \set*{ ({}_s\phi)_{s=1}^r: {}_s\phi \in C^0_{\N_{a,b}}(V[-s],W), \partial ({}_s \phi)={}_s \rho}.\]
We have established that each fiber of the forgetful map $X_r \to X_{r-1}$ is an affine space of dimension $\dim \Hom_Q(V[-r],W)$. To finish the proof that $\init^{-1}_R(V)$ is an affine space, we need to show that $X_r\to X_{r-1}$ is geometrically a trivial fiber bundle. This can be resolved elementarily, without invoking a Quillen--Suslin type theorem.
\begin{lemma}
    \label{lem:trivial-aff-bundle}
    Let $V,W$ be finite-dimensional vector spaces over an algebraically closed field $\bbar\k$ and $L:V\to W$ be a linear map. Suppose there is a variety $S$ and a regular map $\rho:S\to W$ whose image is contained in $\im(L)$. Then the forgetful map
    \[ T\coloneqq \set{(v,s)\in V\times S:Lv=\rho(s)} \to S\]
    is isomorphic to the trivial fibration $\A^{\dim \ker L}\times S\to S$. Moreover, if $V,W,L,S,\rho$ descend to a subfield $\k\subeq \bbar\k$, then so does this isomorphism.
\end{lemma}
\begin{proof}
    By choosing the splittings of $L:V\onto \im(L)$ and $\im(L)\incl W$, we get vector spaces $A,B,C$ such that $V=A\oplus B$, $W=B\oplus C$, and $L:A\oplus B\to B\oplus C$ is the map sending $(v_A,v_B)$ to $v_B$. The restriction of the codomain, $\bbar\rho:S\to B$, is still a regular map since it can be produced by composing $\rho$ with the projection map $W\onto B$. Now we have
    \[ T=\{((v_A,v_B),s):v_B=\bbar\rho(s)\}=\{((v_A,\bbar\rho(s)),s):v_A\in A, s\in S\}, \]
    so the forgetful map $T\to S$ is isomorphic to the trivial fibration $A\times S\to S$. This proves all of our claims, since $\dim A=\dim V-\dim \im(L)=\dim \ker(L)$ and since this construction works over any field.
\end{proof}

At this point, we have proved that for each admissible pair $(\mathbf{n};n)\in \N^G\times \N$, each fiber of the \BB map $\pi_{\mathbf{n}}:X_\mathbf{n}\to \Fl_G(\mathbf{n};n)$ is an affine space of constant dimension $B(\mathbf{n};n)$. To show that it is a Zariski-local fiber bundle, we explicitly display a Zariski-local trivialization. Recall the definition of $U^\k_W$ from the discussion before \eqref{eq:u_w}. For $V^0\in \Fl_G(\mathbf{n};n)\subeq E_R(\tl R^n)^\Gm$ and $W=(W_d)$ such that $W_d\oplus V^0_d=\k^n$, let
\[ U^R_W=U^\k_W\cap \Fl_G(\mathbf{n};n)=\{V\in \Fl_G(\mathbf{n};n): W_d\oplus V_d=\k^n\},\]
an open neighborhood of $V^0$ in $\Fl_G(\mathbf{n};n)$.

\begin{lemma}\label{lem:bb-loc-trivial}
    Assuming the above, the \BB map $\pi_{\mathbf{n}}$ restricted to $U^R_W$ is a trivial affine bundle.
\end{lemma}
\begin{proof}
    We use the same strategy as in Theorem~\ref{thm:bb-fiber}, except that we need to take account of the fact that $V$ is varying. The key is to bring back the non-centralized Gr\"obner coordinate, $h$ in \eqref{eq:u_wv}, which recovers $L\in \pi_{\mathbf{n}}^{-1}(U^R_W)$ without reference to $V$; in fact, $L=\ker(h)=\im(1-h)$. 

    Recall $h=\phi+1_W^V\in \Hom_\k(\bbar{\tl R^n},W)$ with $\phi\in \Hom_\k(\bbar{\tl R^n},W)_{>0}$, where $1_W^V$ is the projection from $\bbar{\tl R^n}$ onto $W$ with respect to $W\oplus \bbar{V}=\bbar{\tl R^n}$. Then the degree $r$ part of $h$ satisfies ${}_0 h=1_W^V$ and ${}_r h={}_r \phi$ for $r>0$. In particular, the degree $0$ part of $h$ recovers $V$ by
    \[ \bbar{V}=\ker({}_0 h)=\im(1-{}_0 h).\]
    The only constraint on $h$ is that $h|_W=1$ and $\ker(h)$ is an $R$-module.  Thus, the defining equation for $\pi_{\mathbf{n}}^{-1}(U^R_W)$ is
    \begin{equation}
        \label{eq:groebner-equation-urw}
        \pi_{\mathbf{n}}^{-1}(U^R_W)=\set*{ h\in \Hom_\k(\bbar{\tl R^n},W): h|_W=1, hT^\gamma(1-h)=0\text{ for }\gamma=a,b}.
    \end{equation}
    Degree-wise, the condition $h|_W=1$ is just ${}_0 h|_W=1$ and ${}_r h|_W=0$ for $r>0$. The equation $h T^\gamma (1-h)=0$ reads
    \[ ({}_r h)  T^\gamma (1-{}_0 h) - ({}_0 h) T^\gamma  ({}_r h) = {}_{r+\gamma}(\phi T^\gamma \phi),\]
    for $r\geq 0$, where $\phi={}_{>0} h$ is the positive degree part of $h$. Note that it is exactly \eqref{eq:groebner-equation-by-degree}, except that the lowest degree $r=0$ is no longer trivial:
    \[ ({}_0 h)  T^\gamma (1-{}_0 h)=0,\]
    capturing the fact that $\bbar{V}=\ker({}_0 h)$ is an $R$-module.

    For $r\geq 0$, define
    \[ Y_r=\set*{({}_s h)_{s=0}^r: \begin{array}{l}
    {}_s h\in \Hom_\k(\bbar{\tl R^n},W)_s,\\
    {}_0 h|_W=1,\\
    {}_s h|_W=0\text{ for }0<s\leq r,\\
    ({}_s h)  T^\gamma (1-{}_0 h) - ({}_0 h) T^\gamma  ({}_s h) = {}_{s+\gamma}(({}_1 h+\dots+{}_{r-1} h) \cdot T^\gamma\cdot ({}_1 h+\dots+{}_{r-1} h)).
    \end{array}}
    \]
    In the formula above we assume \(\gamma=a,b\). Note that $Y_r$ stabilizes at $Y_\infty=Y_{ab-a-b}=\pi_{\mathbf{n}}^{-1}(U^R_W)$ and
    \[ Y_0=\set*{{}_0 h\in \Hom_\k(\bbar{\tl R^n},W)_0: {}_0 h|_W=1, ({}_0 h) T^\gamma (1-{}_0 h)=0}\]  
    is precisely $U^R_W$. Our goal is to show the forgetful map $Y_\infty\to Y_0$ is a trivial affine bundle. It suffices to show that for $r>0$, the forgetful map $Y_r\to Y_{r-1}$ is. 
    
    Let $A, B$ be the finite-dimensional vector spaces
    \[ A=\{ {}_r h\in \Hom_\k(\bbar{\tl R^n},W)_r: {}_r h|_W = 0\},\]
    \[ B=\Hom_\k(\bbar{\tl R^n},W)_{r+a}\oplus \Hom_\k(\bbar{\tl R^n},W)_{r+b},\]
    let $D:A\to B$ be the linear map
    \[ D({}_r h) = (({}_r h)  T^\gamma (1-{}_0 h) - ({}_0 h) T^\gamma  ({}_r h))_{\gamma=a,b},\]
    and let ${}_r \rho:Y_{r-1}\to B$ be the regular map defined by
    \[ {}_r \rho ({}_0 h, \dots, {}_{r-1}h)= ({}_{r+\gamma}(({}_1 h+\dots+{}_{r-1} h) \cdot T^\gamma\cdot ({}_1 h+\dots+{}_{r-1} h)))_{\gamma=a,b}.\]
    Then the forgetful map $Y_{r}\to Y_{r-1}$ is the same as the forgetful map
    \[ \{({}_r h, {}_{<r} h)\in A\times Y_{r-1}: D({}_r h)={}_r \rho ({}_{<r} h)\} \to Y_{r-1}.\]
    The image of ${}_r \rho$ is contained in the image of $D$ by the proof of Theorem~\ref{thm:bb-fiber} (note that this is a fixed ${}_0 h$, or fixed $V$, statement). Hence, by Lemma~\ref{lem:trivial-aff-bundle}, $Y_r\to Y_{r-1}$ is a trivial affine bundle, finishing the proof.
\end{proof}

We summarize all we know so far in the following theorem, which is the detailed statement of Theorem~\ref{thm:bb} and is the main theorem of the entire paper.

\begin{theorem}[$\subeq$ Equation \eqref{eq:fixed-locus}, Theorem~\ref{thm:poset-flag-tame}, Theorem~\ref{thm:bb-fiber}, and Lemma~\ref{lem:bb-loc-trivial}]
    Consider the $\Gm$-action on $X=E_R(\tl R^n;m)$. 
    \begin{enumerate}
        \item The connected components of the fixed point locus $X^\Gm$ are indexed by $\mathbf{n}\in \Z^G$ such that $(\mathbf{n};n)$ is an admissible pair (Proposition~\ref{prop:admissible}) and $\sum_{i\in G} (n-n_i)=m$.
        \item The component corresponding to $\mathbf{n}$, denoted by $F_\mathbf{n}$, is isomorphic to the poset flag variety $\Fl_G(\mathbf{n};n)$, which is smooth projective, is an iterated Grassmannian bundle over a point, admits an affine paving, and has motive $\qbinom{n}{\mathbf{n}}_{\L;G}$. 
        \item Let $X=\coprod_{\mathbf{n}} X_\mathbf{n}$ be the \BB decomposition, where $X_\mathbf{n}$ is the \BB stratum attracted to $F_\mathbf{n}$. Then the \BB map $X_\mathbf{n}\to F_\mathbf{n}$ is an affine bundle of rank $B(\mathbf{n};n)$ defined in \eqref{eq:bb-fiber-dim}.
    \end{enumerate}
    \label{thm:bb-detailed}
\end{theorem}

\begin{remark}
    \label{rmk:bb-upgrade}
    At this point, it is not hard to show that $X_\mathbf{n}$ has an affine paving and can be fairly explicitly constructed using widely accepted techniques in geometric representation theory \alert{[REF]}, thanks to the abundance of torus actions. We only sketch the ideas.

    There is a natural $\GL_n\times \Gm$ action on $\tl R^n=\tl R\otimes \k^n$: the $\GL_n$ acts on the $\k^n$ factor while $\Gm$ acts on $\tl R$. This induces an action of $\mathbb{T}'=\mathbb{T}_n\times \Gm$ on $X_\mathbf{n}$ and $F_\mathbf{n}$, and the \BB map $\pi_\mathbf{n}:X_\mathbf{n}\to F_\mathbf{n}$ with respect to the $\Gm$ factor is equivariant under $\mathbb{T}'$. Recall $F_\mathbf{n}$ is smooth projective and $\mathbb{T}_n$ has finitely many fixed points on $F_\mathbf{n}$ (see the proof of Theorem~\ref{thm:poset-flag-tame}). Then there is a one-parameter subgroup $\mu:\Gm\to \mathbb{T}_n$ such that $F_\mathbf{n}^{\mu(\Gm)}=F_\mathbf{n}^{\mathbb{T}_n}$. Consider a one-parameter subgroup $\mu_N:\Gm\to \mathbb{T}'$, $\mu_N(t)=(\mu(t),t^N)$ for $N\gg 0$; intuitively, it is a combination of the action $\mu$ on the base $F_\mathbf{n}$ and the action $t^N$ that rapidly retracts the fibers. By choosing $N$ sufficiently large, we ensure the fiber action strictly dominates any negative weight from the base action $\mu$, forcing the limit $\lim_{t\to 0} \mu_N(t)\cdot x$ for all $x\in X_\mathbf{n}$ to exist and land in the base fixed point locus $F_\mathbf{n}^{\mathbb{T}_n}$. As a result, the non-complete smooth variety $X_\mathbf{n}$ has a locally closed decomposition into \BB cells with respect to $\mu_N$. 
\end{remark}

\subsection{Final combinatorial details}
Theorem~\ref{thm:bb-detailed} already entails a motivic generating function:
\[ \sum_{m\geq 0} [E_R(\tl R^n;m)]\, t^m = \sum_{\mathbf{n}\in \Z^G} \qbinom{n}{\mathbf{n}}_{\L;G} \L^{B(\mathbf{n};n)} t^{\sum_{i\in G} (n-n_i)},\]
noting that the right-hand side is a finite sum because $\qbinom{n}{\mathbf{n}}_{\L;G}=0$ unless $(\mathbf{n};n)$ is admissible.

However, the current form fails to decouple $B(\mathbf{n};n)$ and masks an extra combinatorial structure that will amount to the fact that $N_{a,b;n}(q,t)$ in \eqref{eq:q,t-sum-finite} is an honest, non-Laurent polynomial in $q,t$, which ultimately stems from the positive definiteness result \cite{huang2026dinv}. Therefore, we rewrite the motivic generating function in terms of $N_{a,b;n}(q,t)$ now. This combinatorial consideration will become crucial in proving that the groupoid volume $S_1$ in Theorem~\ref{thm:groupoid-vol} is finite. 

The first step requires rewriting everything in $\L^{-1}$ variable.
\begin{lemma}
    Let $P$ be a nonempty tame poset, and let $x^-, x^+\in P\cup\{\pm \infty\}$ be defined for $x\in P$ with respect to any good traverse (Definition~\ref{def:tame}). Then
    \[ \qbinom{n}{\mathbf{n}}_{q;P}=q^{D_P(\mathbf{n};n)}\qbinom{n}{\mathbf{n}}_{q^{-1};P},\]
    where
    \begin{equation}
        D_P(\mathbf{n};n)=\sum_{x\in P} (n_{x^+}-n_x)(n_x-n_{x^-})
    \end{equation}
    with $n_\infty=n$ and $n_{-\infty}=0$.
\end{lemma}
\begin{proof}
    This follows directly from the elementary identity $\qbinom{n}{r}_q = q^{r(n-r)} \qbinom{n}{r}_{q^{-1}}$ and \eqref{eq:q-binom-tame}.
\end{proof}
\begin{remark}
    For the gap poset $G=\N\setminus \pairing{a,b}$, we may embed it into a grid poset $Q$ using Example~\ref{eg:gap-young}. The grid poset has a good traverse (Example~\ref{eg:grid-tame}) and by Lemma~\ref{lem:traverse-restrict}, it restricts to a good traverse of $G$ with
    \[ x^+ = x+b, x^-=x-a,\]
    where a number in $\pairing{a,b}$ is treated as $\infty$ and a negative number is treated as $-\infty$. It follows that
    \begin{equation}\label{eq:D_G}
        D_G(\mathbf{n};n)=\sum_{d\in G} (n_{d+b}-n_d)(n_d-n_{d-a}),
    \end{equation}
    with $n_d\coloneqq n$ if $d\in \pairing{a,b}$ and $n_d\coloneqq 0$ if $d<0$. 
\end{remark}

Let $\delta=\abs{G}=(a-1)(b-1)/2$. We have reached
\[ \sum_{m\geq 0} [E_R(\tl R^n;m)]\, t^m = t^{n\delta}\sum_{\mathbf{n}\in \Z^G} \qbinom{n}{\mathbf{n}}_{\L^{-1};G} \L^{B(\mathbf{n};n)+D_G(\mathbf{n};n)} t^{-\abs{\mathbf{n}}}, \]
where $\abs{\mathbf{n}}=\sum_{i\in G} n_i$. It turns out that the sum $B(\mathbf{n};n)+D_G(\mathbf{n};n)$ decouples nicely.

\begin{lemma}\label{lem:dinv-simplification}
    We have
    \[  B(\mathbf{n};n)+D_G(\mathbf{n};n) = n^2 \abs{G}- \dinv(\mathbf{n}).\]
\end{lemma}
\begin{proof}
    For the reader's convenience, we recall 
    $$B(\mathbf{n};n)=\sum_{e\in G, d\in \Z} U(e-d)\, \hhat n_d (n-\hhat n_e),$$
    $$D_G(\mathbf{n};n)=\sum_{d\in G} (\hhat n_{d+b}-n_d)(n_d-n_{d-a}),$$
    \[
    \dinv(\mathbf{n})=\sum_{d,e\in G}  K(e-d)\,n_dn_e,\]
    with
    \[ \hhat{n}_d=\begin{cases}
        n_d, & d\in G,\\
        n, & d\in \Gamma=\N\setminus G,\\
        0, & d<0,
    \end{cases} \quad K=1_{\geq 0}-1_{\geq a}-1_{\geq b}+1_{\geq a+b}, \quad U(r)=K(r-1).\]
    The main challenge of the proof is to treat the extended vector $\hhat{\mathbf{n}}$, and a direct combinatorial approach would involve heavy boundary analysis. Our idea is to exploit the observation
    \[ 1_\Gamma(d)-1_\Gamma(d-a)-1_\Gamma(d-b)+1_\Gamma(d-a-b) = 1_{d=0}-1_{d=ab},\]
    cf.~\eqref{eq:1_Gamma}. To systematically apply it to our setting, we work with the abelian group of integer-valued functions on $\Z$ with support bounded below. We identify it with the Laurent series ring $\Z[T^{-1}][\![T]\!]$, with $f=(f_d)_{d\in \Z}$ corresponding to $f(T)=\sum_{d\in \Z} f_d T^d$. For such functions $f,g$ such that $g$ has finite support, namely, $f\in \Z[T^{-1}][\![T]\!]$ and $g\in \Z[T^{\pm 1}]$, define the dot product
    \[ \pairing{f,g}=\sum_{d\in \Z} f_d g_d.\]
    Any $H(T)\in \Z[T^{-1}][\![T]\!]$ induces a linear operator on $\Z[T^{-1}][\![T]\!]$ by $f(T)\mapsto H(T)f(T)$. In particular, $T^r$ is the shifting operator
    \[ (T^r f)_d = f_{d-r}.\]
    If $H(T)\in \Z[T^{\pm 1}]$, then $H(T)$ has an adjoint $H(T^{-1})$:
    \[ \pairing{f,H(T)g}=\pairing{H(T^{-1})f,g}\]
    if either $f$ or $g$ has finite support. With this language, the ``convolution'' operation can be interpreted as
    \[ \sum_{d,e\in \Z} v_d H(e-d) w_e = \pairing{H(T)v,w}=\pairing{v,H(T^{-1})w}.\]

    We collect the expressions of the relevant vectors and convolution kernels:
    \[ 1_r = T^r, \quad 1_{\geq r} = \frac{T^r}{1-T},  \]
    \[ 1_G = \frac{1}{1-T}-\frac{1-T^{ab}}{(1-T^a)(1-T^b)}\text{ (see \cite{brownshiue})},\]
    \begin{equation}
        \label{eq:1_Gamma}
        1_\Gamma = 1_{\N}-1_G = \frac{1-T^{ab}}{(1-T^a)(1-T^b)},
    \end{equation}
    \[ K = 1_{\geq 0}-1_{\geq a}-1_{\geq b}+1_{\geq a+b} = \frac{(1-T^a)(1-T^b)}{1-T},\]
    a polynomial in $T$, and
    \[ U = TK = T\frac{(1-T^a)(1-T^b)}{1-T}.\]

    Now we restate the relevant functions in this language. Let $n\in \Z$ and $\mathbf{n}$ be a function supported in $G$ (so that $n_d=0$ for $d\notin G$). The extended function is $\hhat{\mathbf{n}}=\mathbf{n}+n 1_\Gamma$. Then
    \[ B(\mathbf{n};n)=\pairing{U(T) \hhat{\mathbf{n}}, n 1_G - \mathbf{n}}\]
    and
    \[ \dinv(\mathbf{n})=\pairing{K(T)\mathbf{n},\mathbf{n}}.\]
    We now rewrite $D_G$ into a more symmetric form. First, we note that
    \[ D_G(\mathbf{n};n)=\sum_{d\in \Z}(\hhat n_{d+b}-\hhat n_d)(n_d - n_{d-a}).\]
    Indeed, if $(\hhat n_{d+b}-\hhat n_d)(n_d - n_{d-a})\neq 0$, then we must have $d\notin \Gamma$ (otherwise the first factor vanishes) and $d\geq 0$ (otherwise the second factor vanishes), so $d\in G$. 
    As a result, 
    \[ D_G(\mathbf{n};n) = \pairing{(T^{-b}-1)\hhat{\mathbf{n}},(1-T^a)\mathbf{n}}=-\pairing{(1-T^{-a})(1-T^{-b})\hhat{\mathbf{n}},\mathbf{n}}.\] 
    
    The rest reduces to a routine computation. Let us compute $S(\mathbf{n};n)\coloneqq B(\mathbf{n};n)+\dinv(\mathbf{n})+D_G(\mathbf{n};n)$ in each degree of $n$ separately. Let $S(\mathbf{n};n)=\sum_{i=0}^2 S_i(\mathbf{n}) \, n^i$. 

    The constant term is contributed by
    \begin{align*}
        S_0(\mathbf{n})&=-\pairing{U(T)\mathbf{n},\mathbf{n}}  + \pairing{K(T)\mathbf{n},\mathbf{n}}-\pairing{(1-T^{-a})(1-T^{-b})\mathbf{n},\mathbf{n}} \\
        &= \pairing*{(-U(T)+K(T)-(1-T^a)(1-T^b))\,\mathbf{n},\mathbf{n}} = 0,
    \end{align*}
    where we have used the adjointness and symmetry.

    The only quadratic contribution is from $B(\mathbf{n};n)$, giving
    \begin{align*}
        S_2(\mathbf{n})&=\pairing{U(T) 1_\Gamma,1_G} =\pairing*{T\frac{(1-T^a)(1-T^b)}{1-T}\cdot \frac{1-T^{ab}}{(1-T^a)(1-T^b)}, 1_G}\\
        &=\pairing*{T\frac{1-T^{ab}}{1-T},1_G} =\pairing{T^1+T^2+\dots+T^{ab},1_G}.
    \end{align*}
    Since $G\subeq [1,ab-a-b]\subeq [1,ab]$, we have $S_2(\mathbf{n})=\abs{G}=\delta$. 

    Finally, we compute the linear term. We have
    \begin{equation*}
        S_1(\mathbf{n})=\pairing{U(T) 1_\Gamma,-\mathbf{n}}+\pairing{U(T) \mathbf{n},1_G}-\pairing{(1-T^{-a})(1-T^{-b})1_\Gamma,\mathbf{n}} =\pairing{F(T),\mathbf{n}},
    \end{equation*}
    where
    \begin{equation*}
        F(T)=-U(T)1_\Gamma + U(T^{-1})1_G-(1-T^{-a})(1-T^{-b})1_\Gamma .
    \end{equation*}
    We claim that the support of $F(T)$ is disjoint from $G$. A direct computation verifies that
    \begin{align*}
        F(T)&=T^{-(a+b)} \parens*{T(1-T^{ab})\frac{1-T^{a+b}}{1-T} - \frac{1-T^a}{1-T}\frac{1-T^b}{1-T}}\\
        &=T^{-(a+b)} \parens*{(1-T^{ab})(T+T^2+\dots+T^{a+b}) - (1+T+\dots+T^{a-1})(1+T+\dots+T^{b-1})}.
    \end{align*}
    By inspection, $F(T)$ does not contain nonzero terms with degree in $[1,ab-a-b]$, and the claim is proved.
    Since $\mathbf{n}$ is supported in $G$, we conclude that $S_1(\mathbf{n})=0$.

    Putting everything together, $B(\mathbf{n};n)+D_G(\mathbf{n};n)+\dinv(\mathbf{n})=n^2\delta$, as required.
\end{proof}

Hence, we reach the motivic generating function in our desired form:
\begin{equation}\label{eq:ext-fiber-motive}
    \sum_{m\geq 0} [E_R(\tl R^n;m)]\, t^m = \L^{n^2 \delta} t^{n\delta}\sum_{\mathbf{n}\in \Z^G} \qbinom{n}{\mathbf{n}}_{\L^{-1};G} \L^{-\dinv(\mathbf{n})} t^{-\abs{\mathbf{n}}}.
\end{equation}

\section{From extension fiber to Quot scheme and proof of Theorems~\ref{thm:paving} and \ref{thm:quot-zeta-tilde}}\label{sec:ext-to-quot}
In this section, we draw on results of \cite{huangjiang2023torsionfree} to prove results about $\Quot^R(\tl R^n)$, Theorems~\ref{thm:paving} and \ref{thm:quot-zeta-tilde}, from results about $E_R(\tl R^n)$, Theorem~\ref{thm:bb-detailed}. Intuitively, the extension fiber $E_R(\tl R^n)$ carries full information of the Quot scheme $\Quot^R(\tl R^n)$ because to each $L\in \Quot^R(\tl R^n)$ we can associate a unique $\tl L\in \Quot^{\tl R}(\tl R^n)$ by $\tl L=\tl R L$, so that $L\in E_R(\tl L)$. Because every full-rank $\tl R$-submodule of $\tl R^n$ is isomorphic to $\tl R^n$, the space of choices of $L$ once $\tl L$ is given is a copy of $E_R(\tl R^n)$. 

Geometrically, this is made precise as follows. We recall an explicit parametrization of the Schubert cell decomposition of $\Quot^{\tl R}_m(\tl R^n)$ from \cite[\S 6.1]{huangjiang2023torsionfree}. A \defn{nonnegative signature of size $m$} is a vector $\mu\in \N^n$ such that $\abs{\mu}=\sum_{i=1}^n \mu_i=m$. The \defn{Iwahori cell} of $\mu$ is an affine space consisting of matrices of certain form in $\GL_n(\tl R)$:
\begin{equation*}
    I_\mu := \set*{[\iota_{ij}]_{i,j=1}^n: \begin{matrix}\iota_{ii}=1 \\
     \iota_{ij}\in \mathrm{span}_\k \set{T^s: 0\leq s < \mu_i-\mu_j}, \text{ if $i>j$}\\
     \iota_{ij}\in \mathrm{span}_\k \set{T^s: 0< s < \mu_i-\mu_j}, \text{ if $i<j$}
    \end{matrix}
    }. 
\end{equation*}
There is a locally closed decomposition
\[ \Quot^{\tl R}_m(\tl R^n)=\bigsqcup_{\substack{\mu\in \N^n\\ \abs{\mu}=m}} X_\mu^\circ,\]
with affine parametrization
\[ I_\mu \xrightarrow{\simeq} X_\mu^\circ, \quad \iota \mapsto \iota \, \diag(T^{\mu_1},\dots,T^{\mu_n}) \tl R^n,\]
whose inverse is denoted by $\tl L\mapsto \iota_{\tl L}$.
The dimension of $I_\mu$ as well as $X_\mu^\circ$ is 
\[ \ell(\mu)=\sum_{i,j=1}^n \max\set*{0,\mu_j-\mu_i+\floor*{\frac{j-i}{n}}}.\]

Let
\[ \Quot^R_m(\tl R^n;l) = \set*{L\in \Quot^R_m(\tl R^n): \dim_\k \tl R L/L=l},\]
giving a locally closed decomposition $\Quot^R_m(\tl R^n)=\bigsqcup_{l=0}^m \Quot^R_m(\tl R^n;l)$. 
\begin{lemma}
    [{Baby version of \cite[Lemma~6.9]{huangjiang2023punctual}}]
    There is a locally closed decomposition
    \[  \Quot^R_m(\tl R^n;l) = \bigsqcup_{\abs{\mu}=m-l} \Quot^R_m(\tl R^n;l,\mu)\]
    with $\Quot^R_m(\tl R^n;l,\mu) \simeq E_R(\tl R^n;l)\times I_\mu$. \label{lem:ext-fiber-to-quot}
\end{lemma}
\begin{proof}
    For $\mu\in \N^n$ with $\abs{\mu}=m-l$, set the strata as
    \[ \Quot^R_m(\tl R^n;l,\mu) = \set*{L\in \Quot^R_m(\tl R^n): \dim_\k \tl R L/L=l, \;\tl R L \in X_\mu^\circ},\]
    then selecting $L\in \Quot^R_m(\tl R^n;l,\mu)$ amounts to choosing $\tl L\in X_\mu^\circ$ together with $L\in E_R(\tl L;l)$, but using the isomorphism $\iota_{\tl L}:\tl R^n\to \tl L$, choosing $L$ amounts to specifying $N=\iota_{\tl L}^{-1}L\in E_R(\tl R^n;l)$. We have thus a pair of morphisms that are inverses to each other:
    \begin{align*}
        E_R(\tl R^n;l)\times I_\mu &\to \Quot^R_m(\tl R^n;l,\mu),\\
        (N, \iota) &\mapsto \iota \, N,\\
        (\iota_{\tl RL}^{-1} L, \iota_{\tl RL})&\mapsfrom L,
    \end{align*}
    giving the required isomorphism.
\end{proof}
\begin{corollary}\label{cor:paving-ext-fiber-to-quot}
    If $E_R(\tl R^n;m)$ has an affine paving for all $m$, then $\Quot^R_m(\tl R^n)$ has an affine paving for all $m$. \qed
\end{corollary}

\begin{proof}
    [Proof of Theorem~\ref{thm:paving}]
    By Theorem~\ref{thm:bb-detailed} and Remark~\ref{rmk:bb-upgrade}, $E_R(\tl R^n;m)$ has an affine paving for all $m$, hence by Corollary~\ref{cor:paving-ext-fiber-to-quot}, $\Quot^R_m(\tl R^n)$ has an affine paving.
\end{proof}

Using the stratification in Lemma~\ref{lem:ext-fiber-to-quot} and a combinatorial summation formula \cite[Equation~(6.20)]{huangjiang2023torsionfree} involving $\ell(\mu)$, we get the motivic version:
\begin{lemma}
    [{cf.~\cite[Corollary~6.3]{huangjiang2023torsionfree}}]
    In $\KVar{\k}[\![t]\!]$, we have
    \[ \sum_{m\geq 0} [\Quot_m^R(\tl R^n)]\, t^m = \frac{1}{(t;\L)_n} \sum_{m\geq 0} [E_R(\tl R^n;m)]\, t^m.\]
    \label{lem:ext-fiber-to-quot-motivic}
\end{lemma}

This proves Theorem~\ref{thm:quot-zeta-tilde}.
\begin{proof}
    [Proof of Theorem~\ref{thm:quot-zeta-tilde}]
    Just combine \eqref{eq:ext-fiber-motive} and Lemma~\ref{lem:ext-fiber-to-quot-motivic}.
\end{proof}

\section{Coh zeta functions and the proof of Theorems~\ref{thm:quot-zeta-special} and \ref{thm:groupoid-vol}}
The goal of this section is to partially understand $\Quot^R_m(R^n)$ for $R=\k[\![T^a,T^b]\!]$ and prove Theorems~\ref{thm:quot-zeta-special} and \ref{thm:groupoid-vol}. To apply the knowledge of $\Quot^R_m(\tl R^n)$ to $\Quot^R_m(R^n)$, we follow a path that has been successfully carried out in \cite{huang2025inert,chernhuang2025}, which requires two more ingredients from \cite{huangjiang2023torsionfree}. 

The first is the rationality. For any finitely generated $R$-module $M$, consider the \defn{motivic Quot zeta function}
\[ \calZ^R_{M}(t)\coloneqq \sum_{m\geq 0} [\Quot_m^R(M)]\, t^m \in \KVar{\k}[\![t]\!].\]
If $M=E^n$ with $E=R$ or $\tl R$, define the \defn{normalized} motivic Quot zeta function
\[ \mathcal{N}^R_M(t)\coloneqq (t;\L)_n \calZ^R_M(t).\]
By Theorem~\ref{thm:quot-zeta-tilde},
\[ \mathcal{N}^R_{\tl R^n}(t) =  \L^{n^2\delta} \, t^{n\delta} N_{a,b;n}(\L^{-1},t^{-1}).\]
\begin{theorem}
    [{Special case of \cite[Theorem~6.2 and Corollary~6.4]{huangjiang2023torsionfree}}]
    The power series $\mathcal{N}^R_{R^n}(t)$ is a polynomial in $\KVar{\k}[t]$. Moreover, we have
    \begin{equation}
        \label{eq:match-at-center}\mathcal{N}^R_{R^n}(1)=\mathcal{N}^R_{\tl R^n}(1).
    \end{equation}
    \label{thm:rationality}
\end{theorem}
Combining the results, we get
\begin{equation}\label{eq:normalized-quot-at-1}
    \mathcal{N}^R_{\tl R^n}(1) =  \L^{n^2\delta}  N_{a,b;n}(\L^{-1},1).
\end{equation}
Note that $(1;\L)_n=0$ for $n>0$, so the evaluation at $t=1$ would not make sense on $\calZ^R_{R^n}(t)$. That is why the normalization is needed even to state \eqref{eq:match-at-center}.

The second ingredient is a point counting functional equation. Suppose $\k=\Fq$ is a finite field. Denote by $Z^R_M(t)$ and $N^R_M(t)$ the point counting specialization of $\calZ^R_M(t)$ and $\mathcal N^R_M(t)$, respectively; this means we apply the ring homomorphism
\[ \#_q: \KVar{\Fq}\to \Z, \quad [X]\mapsto \abs{X(\Fq)}\]
to each coefficient of $\calZ^R_M(t)$ and $\mathcal{N}^R_M(t)$. Note $\#_q$ maps $\L$ to $q$.
\begin{theorem}[{Special case of \cite[Theorem~1.5]{huangjiang2023torsionfree}}]
    \label{thm:pointcount-func-eq}
    Let $\gcd(a,b)=1, R=\Fq[\![T^a,T^b]\!]$, and $\delta=(a-1)(b-1)/2=\dim_{\Fq} \tl R/R$. Then the polynomial $N^R_{R^n}(t)\in \Z[t]$ satisfies the functional equation
    \begin{equation}\label{eq:func-eq}
        N^R_{R^n}(t) = q^{n^2\delta} t^{2n\delta} N^R_{R^n}(q^{-n}t^{-1}).
    \end{equation}
\end{theorem}

This gives enough to prove Theorem~\ref{thm:quot-zeta-special}.
\begin{proof}
    [Proof of Theorem~\ref{thm:quot-zeta-special}]
    Combining \eqref{eq:normalized-quot-at-1} and \eqref{eq:func-eq}, we get
    \[ N^R_{R^n}(q^{-n})=q^{n^2\delta}q^{-2n^2\delta} N^R_{R^n}(1)=N_{a,b;n}(q^{-1},1),\]
    so that \[Z^R_{R^n}(q^{-n})=(q^{-n};q)_n^{-1} N^R_{R^n}(q^{-n})=(q^{-1};q^{-1})_n^{-1} N_{a,b;n}(q^{-1},1),\]
    which is exactly the claim of Theorem~\ref{thm:quot-zeta-special}. 
\end{proof}

\subsection{Coh zeta functions} \label{sub:coh-zeta}
To prove Theorem~\ref{thm:groupoid-vol}, we recall some ingredients of \cite{huangjiang2023torsionfree,huang2024commuting} about the (arithmetic) Coh zeta functions.

We work in a general setting where $R$ is a \defn{local arithmetic order}, i.e., a local $\O$-algebra (with $\O=\Zp$ or $\Fp[\![T]\!]$) that is a free module of finite rank over $\O$. Let $\tl R$ be the \defn{normalization} of $R$, defined as the integral closure of $R$ in the total fraction ring of $R$. The normalization $\tl R$ may not be local, but is always a finite product of local arithmetic orders that are complete DVRs with finite residue fields. See also \cite[\S 2]{yun2013orbital} and \cite[\S 3]{huangjiang2023torsionfree}.

For $R$ a finite product of local arithmetic orders\footnote{We make the assumption slightly more general to include $\tl R$.}, define Dirichlet series
\[ \zeta_{R;n}(s)=\sum_{E\subeq_R R^n} \abs{R^n/E}^{-s-n}\]
and
\[ \zeta_{R;\infty}(s)=\sum_M \frac{1}{\abs{\Aut_R M}} \, \abs{M}^{-s},\]
where $E$ ranges over finite-index submodules of $R^n$ and $M$ ranges over isomorphism classes of finite-cardinality $R$-modules. Define the normalized versions by
\[ \nu_{R;n}(s)=\frac{\zeta_{R;n}(s)}{\zeta_{\tl R;n}(s)}, \quad \nu_{R;\infty}(s)=\frac{\zeta_{R;\infty}(s)}{\zeta_{\tl R;\infty}(s)}.\]
By \cite[Theorem~5.11]{huangjiang2023torsionfree}, $\nu_{R;n}(s)$ is a Dirichlet polynomial. By \cite[Proposition~3.15]{huang2024commuting}, 
\[ \lim_{n\to \infty}\zeta_{R;n}(s)=\zeta_{R;\infty}(s)\]
coefficient-wise. 

If every residue field of $R$ contains a common finite field $\Fq$, then the size of every finite $R$-module $M$ is a power of $q$, so $\zeta_{R;n}(s),\zeta_{R;\infty}(s),\nu_{R;n}(s), \nu_{R;\infty}(s)$ are power series in $t=q^{-s}$. We denote the power series by $Z_{R;n}(t),Z_{R;\infty}(t),N_{R;n}(t), N_{R;\infty}(t)$, respectively; note that they technically depend on the choice of $\Fq$, but we suppress it from the notation when it is clear from the context.

In the proof of Theorem~\ref{thm:groupoid-vol}, it is important to understand how the values of $N_{R;n}(t)$ at $t=1$ control the convergence of the coefficient-wise limit $N_{R;\infty}(t)$ at $t=1$. An additional analytic observation is needed.

First, we organize the main idea of \cite[Proposition~3.15]{huang2024commuting} and \cite[Theorem~1.1]{huangjiang2023torsionfree} into an equivalent definition of $\zeta_{R;n}(s)$.
\begin{lemma}\label{lem:span-prob}
    Let $R$ be any ring such that there are only finitely many $R$-modules up to isomorphism of a given finite cardinality. Then
    \begin{equation}\label{eq:def-quot-alternative}
    \zeta_{R;n}(s)=\sum_M \frac{\abs{M}^{-s}}{\abs{\Aut(M)}} P_n(M),
    \end{equation}
    where $M$ ranges over all isomorphism classes of finite modules over $R$ and 
    $$P_n(M)=\frac{\abs{\mathrm{Surj}_R(R^n,M)}}{\abs{M}^n}$$ is the probability that a random homomorphism from $R^n$ to $M$ is surjective. 
\end{lemma}
\begin{proof}
    Let $\Aut_R M$ act on $\mathrm{Surj}_R(R^n,M)$ by composition. Then the action is free due to the surjectivity requirement, and the orbit space is identified with the set of $E\subeq_R R^n$ such that $R^n/E\simeq M$. Hence,
    \[ \abs{\set*{E\subeq_R R^n:R^n/E\simeq M}}= \abs{\mathrm{Surj}_R(R^n,M)}/\abs{\Aut_R M}.\]
    Multiplying both sides by $\abs{M}^{-s-n}$ and taking the summation with $M$ ranging over all isomorphism classes of finite $R$-modules, the desired equality follows.
\end{proof}

The following observation is elementary.
\begin{lemma}\label{lem:monotone}
    For any finite $R$-module $M$, we have
    \begin{equation}
        P_n(M)\leq P_{n+1}(M),
    \end{equation}
    and $\lim_{n\to \infty} P_n(M)=1$.
\end{lemma}
\begin{proof}
    For $m_1,\dots,m_{n+1}\in M$, since $m_1,\dots,m_n$ generate $M$ implies that $m_1,\dots,m_{n+1}$ generate $M$, we have
    \begin{align}
        P_{n+1}(M)&=\Prob_{m_1,\dots,m_{n+1}\in M}(R\langle m_1,\dots,m_{n+1}\rangle=M)\\
        &=\E_{m_{n+1}\in M}\Prob_{m_1,\dots,m_n\in M}(R\langle m_1,\dots,m_{n+1}\rangle=M) \\
        &\geq \E_{m_{n+1}\in M}\Prob_{m_1,\dots,m_n\in M}(R\langle m_1,\dots,m_n\rangle=M) \\
        &=\E_{m_{n+1}\in M} P_n(M)=P_n(M).
    \end{align}
    Finally, $P_n(M)$ is trivially bounded below by the probability that a random set map $f:[n]\to M$ is surjective, so
    \[ P_n(M)\geq 1-\sum_{m\in M} \Prob_f(\im(f)\not\ni m) = 1-\abs{M} \parens*{\frac{\abs{M}-1}{\abs{M}}}^n,\]
    proving the claim $\lim_{n\to \infty} P_n(M)=1$.
\end{proof}

As a result, we may commute certain limits using the monotone convergence theorem.
\begin{lemma}\label{lem:zeta-at-positive}
    For any real number $\sigma$, consider $\zeta_{R;n}(\sigma)$ and $\zeta_{R;\infty}(\sigma)$, which are well-defined in $[0,\infty]$. Then we have $\lim_{n\to \infty} \zeta_{R;n}(\sigma) = \zeta_{R;\infty}(\sigma)$.
\end{lemma}
\begin{proof}
    Write
    \[ \zeta_{R;n}(s)=\sum_{d\geq 1} a_{d,n} d^{-s}, \quad \zeta_{R;\infty}(s)=\sum_{d\geq 1} a_{d} d^{-s}.\]
    By Lemmas~\ref{lem:span-prob} and \ref{lem:monotone}, we have $a_{d,n}\geq 0$, $a_{d,n}\leq a_{d,n+1}$, and $\lim_{n\to \infty} a_{d,n}=a_d$. For any substitution $s=\sigma\in \R$, $(a_{d,n} d^{-\sigma})_{d,n}$ is a sequence (indexed by $n$) of nonnegative functions on $d\in \Z_{\geq 1}$ that is pointwise increasing, $a_{d,n} d^{-\sigma}\leq a_{d,n+1} d^{-\sigma}$. By the monotone convergence theorem on the counting measure on $\Z_{\geq 1}$, we have
    \[ \lim_{n\to \infty}\sum_{d\geq 1} a_{d,n} d^{-\sigma} = \sum_{d\geq 1} \lim_{n\to \infty} a_{d,n} d^{-\sigma} = \sum_{d\geq 1} a_{d} d^{-\sigma},\]
    which is exactly our claim.
\end{proof}

To make an analogue of Lemma~\ref{lem:zeta-at-positive}, we recall Mertens' theorem on Cauchy products of sequences \cite[Theorem~3.50]{babyrudin}: if $c_n=\sum_{i=0}^n a_i b_{n-i}$, the Cauchy product of $a_n$ and $b_n$, $\sum_{n=0}^\infty a_n=A$ converges absolutely, and $\sum_{n=0}^\infty b_n=B$ converges, then $\sum_{n=0}^\infty c_n$ converges to $AB$.

\begin{lemma}\label{lem:merten}
    Suppose power series $f(t),g(t),h(t)\in \C[\![t]\!]$ satisfy 
    \begin{itemize}
        \item  $f(t)=g(t)h(t)$;
        \item $h(t)$ converges and is nonvanishing on the disc $\abs{z}<M$.
    \end{itemize}
    Then for $\abs{z}<M$, $f(z)$ converges if and only if $g(z)$ converges, and in the case of convergence $f(z)=g(z)h(z)$.
\end{lemma}
\begin{proof}
    Assume $g(z)$ converges. Since the radius of convergence of $h(t)$ is strictly larger than $\abs{z}$, $h(z)$ is absolutely convergent. Thus Mertens' theorem implies $f(z)$ converges and $f(z)=g(z)h(z)$. 

    For the converse, just switch the roles of $f(t)$ and $g(t)$ and replace $h(t)$ by $h(t)^{-1}$.
\end{proof}

\begin{lemma}\label{lem:merten-power-series}
    Let $R$ be a finite product of local arithmetic orders such that every residue field contains $\Fq$. Then for any $z$ with $\abs{z}<q$, 
    \begin{enumerate}
        \item $Z_{\tl R;\infty}(z)$ converges absolutely,
        \item $Z_{R;\infty}(z)$ converges if and only if $N_{R;\infty}(z)$ converges, and in the case of convergence
        \[ Z_{R;\infty}(z)=Z_{\tl R;\infty}(z) N_{R;\infty}(z).\]
    \end{enumerate}
\end{lemma}
\begin{proof}
    For any complete DVR $V$ with residue field $\F_{q^r}$, recall Solomon's formula \cite{solomon1977zeta}
    \[ \zeta_{V;n}(s)=1/(q^{-r}q^{-rs};q^{-r})_n,\]
    so $\zeta_{V;n}(s)=1/(q^{-r}t^r;q^{-r})_n$, where $t=q^{-s}$. It follows that $\zeta_{V;\infty}(s)=1/(q^{-r}t^r;q^{-r})_\infty$. 
    
    Since $\tl R$ is a finite product of complete DVRs with residue fields $\F_{q^r}$ for various $r$, $Z_{\tl R;\infty}(t)$ is a finite product of series of the form $1/(q^{-r}t^r;q^{-r})_\infty$. In particular, $Z_{\tl R;\infty}(z)$ converges and is nonvanishing on $\abs{z}<q$. 
    This proves (a), and part (b) follows from Lemma~\ref{lem:merten}.
\end{proof}

Combining Lemma~\ref{lem:zeta-at-positive} and Lemma~\ref{lem:merten-power-series}, we get the following general-purpose criterion to test convergence of $N_{R;\infty}(t)$ up to radius $q$.

\begin{theorem}\label{thm:convergence-pipeline}
    Let $R$ be a finite product of local arithmetic orders such that every residue field contains $\Fq$, and let $\sigma$ be any real number in $[0,q)$. Then $N_{R;\infty}(t)$ converges at $t=\sigma$ if and only if the sequence $N_{R;n}(\sigma)$ converges as $n\to \infty$. Moreover, in the case of convergence
    \[ N_{R;\infty}(\sigma)=\lim_{n\to \infty} N_{R;n}(\sigma).\]
\end{theorem}
\begin{remark}
    A particularly useful case will be $\sigma=1$. Note that $N_{R;n}(\sigma)$ is guaranteed to converge because $N_{R;n}(t)$ is a polynomial.
\end{remark}
\begin{proof}
    By Lemma~\ref{lem:merten-power-series} and Lemma~\ref{lem:zeta-at-positive}, 
    \begin{align*}
        &\text{$N_{R;\infty}(t)$ converges at $t=\sigma$}\\
        \iff &\text{$Z_{R;\infty}(t)$ converges at $t=\sigma$} \\
        \iff &\text{$Z_{R;n}(\sigma)\in [0,\infty]$ converges as $n\to\infty$}.
    \end{align*}
    Recall $Z_{R;n}(t)=Z_{\tl R;n}(t) N_{R;n}(t)$, and $N_{R;n}(t)$ is a polynomial. Thus, the radius of convergence of $Z_{R;n}(t)$ is at least the radius of convergence of $Z_{\tl R;n}(t)$, which is at least $q$. In particular, we always have $Z_{R;n}(\sigma)<\infty$ and $Z_{R;n}(\sigma)=Z_{\tl R;n}(\sigma) N_{R;n}(\sigma)$ as real numbers for all finite $n$. Because $\lim_{n\to \infty} Z_{\tl R;n}(\sigma)$ exists in $(0,\infty)$, the convergence of $\lim_{n\to \infty} Z_{R;n}(\sigma)$ and the convergence of $\lim_{n\to \infty} N_{R;n}(\sigma)$ are equivalent. This completes the proof of the desired equivalence. The required equality in the convergence case follows easily through tracing the steps of the proofs.
\end{proof}

\begin{theorem}
    \label{thm:converge-at-1}
    For $R=\Fq[\![T^a,T^b]\!]$, the infinite-rank Coh zeta functions $Z_{R;\infty}(t)$ and $N_{R;\infty}(t)$ have radius of convergence at least $1$. Moreover, they converge at $t=1$ with values
    \[ Z_{R;\infty}(1) = (q^{-1};q^{-1})_\infty^{-1} N_{a,b;\infty}(q^{-1},1), \quad N_{R;\infty}(1) = N_{a,b;\infty}(q^{-1},1).\]
\end{theorem}
\begin{proof}
    Theorem~\ref{thm:quot-zeta-special} really says
    \[ N_{R;n}(1) =  N_{a,b;n}(q^{-1},1),\]
    and by Remark~\ref{rmk:pos-def}, 
    \[ \lim_{n\to \infty} N_{R;n}(1)=N_{a,b;\infty}(q^{-1},1)\]
    converges. By Theorem~\ref{thm:convergence-pipeline}, both $Z_{R;\infty}(1)$ and $N_{R;\infty}(1)$ converge, with
    \[ N_{R;\infty}(1) = \lim_{n\to \infty} N_{R;n}(1).\]
    This completes the proof.
\end{proof}

\begin{proof}
    [Proof of Theorem~\ref{thm:groupoid-vol}]
    For $R=\Fq[\![T^a,T^b]\!]\simeq \Fq[\![X,Y]\!]/(X^a-Y^b)$, the quantities $S_1$ and $S_2$ are two equivalent interpretations of $Z_{R;\infty}(1)$: $S_1$ is just the definition, while $S_2$ is the well-known commuting-matrix interpretation (see \cite{huang2023mutually,huangjiang2023punctual,huang2024commuting}). Hence, Theorem~\ref{thm:groupoid-vol} follows from the formula of $Z_{R;\infty}(1)$ in Theorem~\ref{thm:converge-at-1}. 
\end{proof}

\section{Further discussions and future directions}\label{sec:further}

\subsection{The product side and $W$-algebra characters}\label{subsec:product}

We first make the Rogers--Ramanujan prediction in part~(5) of Conjecture~\ref{conj:three-var} explicit for torus knot singularities. Assume throughout this subsection that $a<b$ are coprime, put $d=a+b$, and let $R=\k[\![T^a,T^b]\!]$. Since
\[
  H_{R,\infty}(q,1,1)=N_{a,b;\infty}(q,1),
\]
the proposed product is
\begin{equation}\label{eq:RR-product}
  N_{a,b;\infty}(q,1)=P_{a,b}(q)\coloneqq\prod_{i\geq1}(1-q^i)^{-r_{a,b}(i)}.
\end{equation}
\begin{remark}\label{rmk:charge}
The exponent $r_{a,b}(i)$ is $d$-periodic, symmetric ($r_{a,b}(i)=r_{a,b}(d-i)$), nonnegative, and satisfies
\[
  r_{a,b}(0)=0,
  \qquad
  \sum_{i=1}^{d}r_{a,b}(i)=(a-1)(b-1).
\]
Writing $s_{a,b}(i)=r_{a,b}(i)+1$, one has
\[
  s_{a,b}(i)=\mathbf 1_{d\Z}(i)+\min\{a,b,\operatorname{dist}(ai,d\Z)\}.
\]
Geometrically, after a common rescaling, the values $s_{a,b}(1),\dots,s_{a,b}(d-1)$ are the lengths of the segments into which the grid lines of an $a\times b$ rectangle divide its diagonal; see Figure~\ref{fig:product-side-grid} for the example for $(a,b)=(4,5)$. We thank Peng Zhou for pointing this out to us.
\end{remark}

\begin{figure}[ht]
\begin{center}

\tikzset{every picture/.style={line width=0.75pt}} 

\begin{tikzpicture}[x=0.75pt,y=0.75pt,yscale=-1,xscale=1]

\draw  [draw opacity=0] (124,39) -- (275.5,39) -- (275.5,160) -- (124,160) -- cycle ; \draw   (124,39) -- (124,160)(154,39) -- (154,160)(184,39) -- (184,160)(214,39) -- (214,160)(244,39) -- (244,160)(274,39) -- (274,160) ; \draw   (124,39) -- (275.5,39)(124,69) -- (275.5,69)(124,99) -- (275.5,99)(124,129) -- (275.5,129)(124,159) -- (275.5,159) ; \draw    ;
\draw    (124,159) -- (274,39) ;

\draw (132,133) node [anchor=north west][inner sep=0.75pt]   [align=left] {4};
\draw (156,132) node [anchor=north west][inner sep=0.75pt]   [align=left] {1};
\draw (163,107) node [anchor=north west][inner sep=0.75pt]   [align=left] {3};
\draw (186,102) node [anchor=north west][inner sep=0.75pt]   [align=left] {2};
\draw (200,79) node [anchor=north west][inner sep=0.75pt]   [align=left] {2};
\draw (222,76) node [anchor=north west][inner sep=0.75pt]   [align=left] {3};
\draw (232,51) node [anchor=north west][inner sep=0.75pt]   [align=left] {1};
\draw (255,48) node [anchor=north west][inner sep=0.75pt]   [align=left] {4};

\end{tikzpicture}

\end{center}
\caption{The sequence $s_{4,5}(1),\dots,s_{4,5}(8)$ from a $4\times 5$ rectangle}
\label{fig:product-side-grid}
\end{figure}
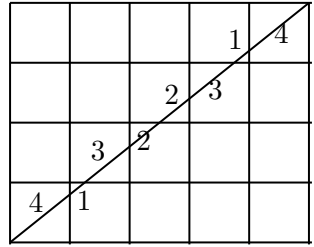

We use only the following standard family of $W$-algebra characters. The irreducible modules of the minimal model $\cW_a(a,d)$ are indexed, up to cyclic rotation, by $a$-tuples of positive integers $(j_0,\dots,j_{a-1})$ with sum $d$. The $A_{a-1}$ Macdonald formula gives the normalized character
\begin{equation}\label{eq:product_VOA}
  \bbar\chi^{a,a,d}_{(j_0,\dots,j_{a-1})}(q)
  =\frac{(q^d;q^d)_\infty^{a-1}}{(q;q)_\infty^{a-1}}
    \prod_{s=1}^{a-1}\prod_{v=0}^{a-1}
    (q^{j_v+\cdots+j_{v+s-1}};q^d)_\infty,
\end{equation}
where the subscripts of $j$ are read modulo $a$; see \cite[(5.2)]{asw1999}.

Write $d=Qa+r$ with $0<r<a$. The \defn{Euclidean rhythm} $\mathrm{ER}(a,b)$ is the cyclic $a$-tuple with $r$ entries equal to $Q+1$ and the remaining entries equal to $Q$, the larger entries being placed as evenly as possible. Equivalently, one may take
\[
  j_v=Q+1
  \quad\Longleftrightarrow\quad
  v=\left\lfloor\frac{ka}{r}\right\rfloor\pmod a
  \quad\text{for some }k\in\Z.
\]
The \defn{Dirac rhythm} $\mathrm{DR}(a,b)$ is the cyclic tuple $(b+1,1,\dots,1)$. These are respectively the maximally and minimally even distributions of $d$ among $a$ slots. For example, $\mathrm{ER}(7,9)=(3,2,2,3,2,2,2)$ and $\mathrm{DR}(7,9)=(10,1,1,1,1,1,1)$.

\begin{lemma}\label{lm:rhythm_lemma}
Let $(j_0,\dots,j_{a-1})=\mathrm{ER}(a,b)$. For $1\leq s\leq a$, the multiset
\[
  \{j_v+\cdots+j_{v+s-1}:0\leq v<a\}
\]
contains only
\[
  sQ+\left\lfloor\frac{sr}{a}\right\rfloor
  \quad\text{and}\quad
  sQ+\left\lfloor\frac{sr}{a}\right\rfloor+1;
\]
the second occurs $sr\bmod a$ times.
\end{lemma}
We omit its proof.

\begin{proposition}\label{prop:Cohchar}
For $(j_0,\dots,j_{a-1})=\mathrm{ER}(a,b)$,
\begin{equation}\label{eq:RR-character}
  P_{a,b}(q)=\bbar\chi^{a,a,d}_{\mathrm{ER}(a,b)}(q).
\end{equation}
\end{proposition}
\begin{proof}
Insert Lemma~\ref{lm:rhythm_lemma} into \eqref{eq:product_VOA}. If $n_i$ is the multiplicity of $i$ among all cyclic partial sums $j_v+\cdots+j_{v+s-1}$, it is enough to prove
\[
  a=s_{a,b}(i)+n_i,
  \qquad 1\leq i<d.
\]
Write $sr=au+v$, $0\leq v<a$. At $i=sQ+u$ the two terms on the right are $v$ and $a-v$, and at $i=sQ+u+1$ they are $a-v$ and $v$. At every other $i$, one has $n_i=0$ and $s_{a,b}(i)=a$. This proves the required equality of the exponents in each residue class modulo $d$.
\end{proof}

Consequently, Conjecture~\ref{conj:rr-type} may be written
\begin{equation}\label{eq:RR variant2}
  N_{a,b;\infty}(q,1)
  \stackrel{?}{=}\bbar\chi^{a,a,d}_{\mathrm{ER}(a,b)}(q).
\end{equation}
The specialization that lies in the opposite extreme of Conjecture~\ref{conj:RR_general} is experimentally
\begin{equation}\label{eq:RR variant3}
  N_{a,b;\infty}(q,q)
  \stackrel{?}{=}\bbar\chi^{a,a,d}_{\mathrm{DR}(a,b)}(q).
\end{equation}
\begin{example}\label{ex:twoclassicalRRs}
For $(a,b)=(2,3)$, the rhythms are $(3,2)$ and $(4,1)$, and \eqref{eq:RR variant2}--\eqref{eq:RR variant3} become the two classical Rogers--Ramanujan identities
\[
  \sum_{m\geq0}\frac{q^{m^2}}{(q;q)_m}=\frac1{(q,q^4;q^5)_\infty},
  \qquad
  \sum_{m\geq0}\frac{q^{m^2+m}}{(q;q)_m}=\frac1{(q^2,q^3;q^5)_\infty}.
\]
\end{example}

\begin{remark}[Comparison with colored Jones limits]\label{subsec:jones}
Kanade's colored Jones limits for $T(a,d)$ produce the character labelled by $\mathrm{DR}(a,b)$ rather than $\mathrm{ER}(a,b)$ \cite{Kanade_old}. Thus the Coh specialization $N_{a,b;\infty}(q,1)$ and the usual colored-Jones limit already give the two different Rogers--Ramanujan products when $(a,b)=(2,3)$. It remains natural to ask whether the $\mathrm{ER}(a,b)$ character occurs as a colored Jones limit for another knot, Lie algebra, or color.
\end{remark}

\subsection{Geometric consequences of the master polynomial}\label{subsec:master-geometry}

We now spell out the geometric statements contained in Conjecture~\ref{conj:colorHOMFLYconj}. They are recorded as consequences rather than as additional conjectures.

Let $f\in\C[\![X,Y]\!]$ be reduced, put
\[
  R=\C[\![X,Y]\!]/(f),
  \qquad
  R_f^{(n)}=\C[\![X,Y]\!]/(f^n),
\]
and write $b(R)$ for the number of branches. Thus $\tl R\simeq\prod_{j=1}^{b(R)}\C[\![T_j]\!]$. Besides the normalized Quot numerator
\[
  \mathcal N^R_{\tl R^n}(t)
  =(t;\L)_n^{b(R)}
    \sum_{m\geq0}[\Quot^R_m(\tl R^n)]t^m,
\]
consider
\begin{align}
  \mathcal N^R_{R^n}(t)
    &=(t;\L)_n^{b(R)} \calZ^{\Quot}_{f;n}(t)\coloneqq 
      (t;\L)_n^{b(R)}\sum_{m\geq0}[\Quot^R_m(R^n)]t^m,
      \label{eq:quot-numerator}\\
  \cH_{f;n}(t)
    &=(t;\L t)_n^{b(R)} \calZ^{\Hilb}_{f;n}(t)\coloneqq 
      (t;\L t)_n^{b(R)}\sum_{m\geq0}[\Hilb_m(R_f^{(n)})]t^m.
      \label{eq:H-normalized}
\end{align}
We write $Z^{\Quot}_{f;n}(q,t)$, $Z^{\Hilb}_{f;n}(q,t)$, $N_{\tl R^n}(q,t)$, $N_{R^n}(q,t)$, $\cH_{f;n}(q,t)$ for the virtual weight realizations of the corresponding motivic expressions; they are integer-coefficient polynomials or series in $q^{1/2}$ and $t$, and whenever the relevant motives are polynomials in $\L$, they are obtained from the virtual-weight expressions with $q$ specializing to $\L$. 

Setting the perverse variable equal to $1$ in Conjecture~\ref{conj:colorHOMFLYconj} first gives, without any purity assumption, the virtual-weight identity
\begin{equation}\label{eq:HQ}
  \cH_{f;n}(q,t)=N_{\tl R^n}(tq,t).
\end{equation}
Equivalently, before normalization,
\begin{equation}\label{eq:HQ-unnormalized}
  Z^{\Hilb}_{f;n}(q,t)=Z^{\Quot}_{f;n}(tq,t).
\end{equation}
Indeed, $q\mapsto tq$ carries $(t;q)_n^{b(R)}$ to $(t;tq)_n^{b(R)}$. At $n=1$, \eqref{eq:HQ-unnormalized} is the Hilb-vs-Quot conjecture of Kivinen--Trinh \cite{kivinentrinh2023}. Once both motivic numerators are known to lie in $\Z[\L,t]$, the virtual-weight identity determines the corresponding polynomial representatives and hence gives the stated motivic formula. Purity and evenness instead provide an honest Poincar\'e-theoretic interpretation, while a compatible affine paving would give a stronger geometric explanation of motivic polynomiality.

If $R$ is unibranched, let $K$ be its algebraic knot. Recall the three slices of \eqref{eq:master-knot}:
\begin{align}
  \mathcal N^R_{R^n}(t)
    &=t^{n\delta}\Bot\tcP^K_{S^n}
      \bigl(t^{1/2},1,\L^{1/2}\bigr),
      \label{eq:spec-quot-R}\\
  \mathcal N^R_{\tl R^n}(t)
    &=t^{n\delta}\Bot\tcP^K_{S^n}
      \bigl(\L^{-n/2},t^{-1/2},\L^{1/2}\bigr),
      \label{eq:spec-quot-Rt}\\
  \cH_{f;n}(\L,t)
    &=t^{n\delta}\Bot\tcP^K_{S^n}
      \bigl((t\L)^{-n/2},t^{-1/2},(t\L)^{1/2}\bigr).
      \label{eq:spec-hilb}
\end{align}
The first lies on $\tr=1$, while the second and third lie on
\begin{equation}\label{eq:subtorus}
  \qtil\tc^n=1.
\end{equation}
The substitution $\L\mapsto t\L$ carries \eqref{eq:spec-quot-Rt} to \eqref{eq:spec-hilb}, including the normalizing factor, and hence recovers \eqref{eq:HQ}.

Two further consistency checks are built into these formulas. First, GGS self-symmetry on \eqref{eq:spec-quot-R} is exactly the functional equation
\[
  \mathcal N^R_{R^n}(t)
  =\L^{n^2\delta}t^{2n\delta}
    \mathcal N^R_{R^n}(\L^{-n}t^{-1}),
\]
the conjectural motivic lift of the point-count functional equation (Theorem~\ref{thm:pointcount-func-eq}). Second, at $n=1$ one has $R_f^{(1)}=R$ and $\Hilb_m(R)=\Quot^R_m(R)$, so \eqref{eq:spec-quot-R} and \eqref{eq:spec-hilb} agree. In knot-theoretic language this is the uncolored bottom-row symmetry underlying ORS.

For the torus knot $K=T(a,b)$, Theorem~\ref{thm:quot-zeta-tilde} gives
\[
  \mathcal N^R_{\tl R^n}(t)
  =\L^{n^2\delta}t^{n\delta}
    N_{a,b;n}(\L^{-1},t^{-1}).
\]
Comparing with \eqref{eq:spec-quot-Rt}, or equivalently applying GGS self-symmetry, yields
\begin{equation}\label{eq:quad-N}
  N_{a,b;n}(q,t)
  =q^{n^2\delta}\Bot\tcP^{T(a,b)}_{S^n}
    \bigl(q^{n/2},t^{1/2},q^{-1/2}\bigr).
\end{equation}
Thus \eqref{eq:quad-N} is precisely the $H_{R,n}(q,1,t)$ slice of the master conjecture. In particular,
\begin{align}
  N_{a,b;n}(q,q)
    &=q^{n^2\delta}\Bot H^{T(a,b)}_{S^n}(q^{-1/2}),
      \label{eq:diagonal}\\
  N_{a,b;n}(1,t)
    &=N_{a,b;1}(1,t)^n,
      \label{eq:q1-growth}
\end{align}
where the second identity follows from refined exponential growth in the cases where that property holds. It is not asserted for arbitrary knots.

The Hilbert-side consequence is the explicit prediction
\begin{equation}\label{eq:hilb-N}
  Z^{\Hilb}_{a,b;n}(q,t)
  =\frac{q^{n^2\delta}t^{n(n+1)\delta}
          N_{a,b;n}((tq)^{-1},t^{-1})}
         {(t;tq)_n},
\end{equation}
or, after normalization,
\begin{equation}\label{eq:H-N}
  \cH_{a,b;n}(q,t)
  =q^{n^2\delta}t^{n(n+1)\delta}
    N_{a,b;n}((tq)^{-1},t^{-1}).
\end{equation}

\begin{remark}[Evidence for the geometric slices]\label{rmk:HQ-evidence}
The identity \eqref{eq:HQ-unnormalized} was checked by finite-field enumeration for $f=Y^a-X^b$ with $(a,b)=(2,3),(2,5),(3,4)$, for $n\leq3$ and colength at most $6$, over $\F_q$ with $q=3,5,7$. For a smooth germ $f=X$ (i.e., the unknot), it was checked for $n\leq2$ and colength at most $4$; here it predicts the Hilbert zeta function of the thickened line $\C[\![X,Y]\!]/(X^n)$ is
\[
  Z^{\Hilb}_{X;n}(q,t)=\prod_{i=1}^n(1-q^{i-1}t^i)^{-1}.
\]
At $n=1$, it was also checked through colength $4$ for the multibranch germs $XY$, $Y^2-X^4$, and $X^3-XY^2$, over the finite fields in the original computation.

For the knot-theoretic slices, the normalized $\tl R^n$ formula is supplied exactly by Theorem~\ref{thm:quot-zeta-tilde}. The independently computed $R^n$ coefficients agree with the KRSS prediction through the available colength ranges: through $t^4$ for $T(3,4)$ at $n=1,2,3$ and for $T(3,5)$ at $n=1,2$, and through $t^5$ for $T(3,7)$ at $n=1,2$. The $R^n$ side was obtained by Gr\"obner stratification, while the knot side was computed from the quivers of \cite[\S5.5]{KRSS}. Thus these are finite-order checks of \eqref{eq:spec-quot-R}--\eqref{eq:spec-hilb}, not complete computations of the $R^n$ series.
\end{remark}

\subsection{The knots--quivers model and computational evidence}\label{subsec:Nhat}

The knots--quivers correspondence is not a third definition of the master polynomial. Rather, it is a conjectural finite presentation of the colored knot invariant appearing in \eqref{eq:master-knot}. We recall only the extremal, or bottom-row, form that is used here.

Let $K$ be a knot whose bottom-row colored invariants admit a standard knots--quivers presentation. Its quiver has one vertex for each generator of the uncolored bottom row. If $(q_i,t_i)$ are the $\qtil$- and $\tr$-degrees of these generators, the KRSS formula \cite[Conjectures~4.2 and 4.4]{KRSS} asserts that there is a symmetric integral matrix $C=(C_{ij})$ with
\begin{equation}\label{eq:KRSS-diagonal}
  C_{ii}=t_i
\end{equation}
such that
\begin{equation}\label{eq:KRSS-bottom}
  \Bot\tcP^K_{S^n}(\qtil,\tr,\tc)
  =\sum_{d_1+\cdots+d_M=n}
    \qtil^{\sum_iq_id_i}\tr^{\sum_it_id_i}
    \tc^{\sum_{i,j}C_{ij}d_id_j}
    \frac{(\tc^2;\tc^2)_n}{\prod_i(\tc^2;\tc^2)_{d_i}}.
\end{equation}
Thus, conditional on KRSS, \eqref{eq:master-knot} becomes an explicit finite quiver sum for $H_{R,n}$. The original knots--quivers correspondence was proposed well beyond torus knots, but generalized forms are needed for some knots with super-exponential growth. None of this is required in the statement of Conjecture~\ref{conj:colorHOMFLYconj}; in this paper KRSS is used only for the torus-knot computations below.

We now compare the quiver parameters with the geometry of the gap poset. Let
\[
  G=\N\setminus\langle a,b\rangle,
  \qquad
  \mathcal J(G)=\{F\subseteq G:F\text{ is an order filter}\}.
\]
The indicator vectors $\mathbf1_F$, $F\in\mathcal J(G)$, are exactly the $0,1$-vectors in the cone appearing in \eqref{eq:dinv}. Put
\begin{equation}\label{eq:filter-stats}
  \ell(F)=\dinv(\mathbf1_F),
  \qquad
  s(F)=|F|+\ell(F).
\end{equation}
Then
\begin{equation}\label{eq:rank-one-filter-sum}
  N_{a,b;1}(q,t)=\sum_{F\in\mathcal J(G)}q^{\ell(F)}t^{|F|}.
\end{equation}

An \defn{atom vector of size $n$} is a tuple $\mathbf d=(d_F)_{F\in\mathcal J(G)}$ of nonnegative integers with $\sum_Fd_F=n$. It determines
\[
  \mathbf n(\mathbf d)=\sum_Fd_F\mathbf1_F,
  \qquad
  |\mathbf n(\mathbf d)|=\sum_F|F|d_F.
\]
These vectors have a direct geometric meaning. If $\mathbb T_n\subseteq\GL_n$ is the diagonal torus, then
\[
  \Fl_G(\mathbf n;n)^{\mathbb T_n}
  \longleftrightarrow
  \left\{(F_1,\dots,F_n)\in\mathcal J(G)^n:
  \sum_{i=1}^n\mathbf1_{F_i}=\mathbf n\right\}.
\]
Recording only the multiplicities of the filters gives the atom vector $\mathbf d$, and the corresponding block has $\binom n{\mathbf d}$ fixed points. Thus atom vectors on the Quot side are the same combinatorial objects as dimension vectors in \eqref{eq:KRSS-bottom}.

\begin{proposition}[Rank-one dictionary]\label{prop:rank-one-filters}
Assume the rank-one case of Conjecture~\ref{conj:colorHOMFLYconj} for $T(a,b)$. The bottom row has one generator for each $F\in\mathcal J(G)$, of degrees
\begin{equation}\label{eq:rank-one-degrees}
  \deg_\alpha=2\delta,
  \qquad
  \deg_\qtil=2(s(F)-\delta),
  \qquad
  \deg_\tr=\deg_\tc=2|F|.
\end{equation}
Consequently the extremal quiver vertices are indexed by $\mathcal J(G)$ and
\begin{equation}\label{eq:node-dictionary}
  q_F=2(s(F)-\delta),
  \qquad
  t_F=C_{FF}=2|F|.
\end{equation}
\end{proposition}
\begin{proof}
Equation \eqref{eq:rank-one-filter-sum} is immediate from the definition of $N_{a,b;1}$. In rank one the two GGS homological gradings coincide. Substitution into \eqref{eq:quad-N} then sends a generator of degrees $(A,B,B)$ to $q^{\delta+(A-B)/2}t^{B/2}$. Comparing with \eqref{eq:rank-one-filter-sum} gives \eqref{eq:rank-one-degrees}. Formula \eqref{eq:node-dictionary} now follows from \eqref{eq:KRSS-diagonal}.
\end{proof}

For $\langle3,4\rangle$, the gap poset is $G=\{1,2,5\}$, with $1,2\prec5$, and the dictionary is displayed by
\[
\begin{array}{c|ccccc}
F&\varnothing&\{5\}&\{1,5\}&\{2,5\}&G\\ \hline
|F|&0&1&2&2&3\\
\ell(F)&0&1&1&2&3\\
s(F)&0&2&3&4&6
\end{array}
\]
It gives the bottom row
\begin{equation}\label{eq:B341}
  \Bot\tcP^{T(3,4)}_{\square}
  =\qtil^{-6}+\qtil^{-2}\tr^2\tc^2+\tr^4\tc^4
   +\qtil^2\tr^4\tc^4+\qtil^6\tr^6\tc^6,
\end{equation}
in agreement with \cite[\S4.3]{GGS}. For $T(3,5)$, the two degree vectors $(q_F)_F$ and $(t_F)_F$ are
\[
  (-8,-4,-2,0,2,4,8),
  \qquad
  (0,2,4,4,6,6,8),
\]
and for $T(3,7)$ they are
\begin{equation}\label{eq:37-node-data}
  (-12,-8,-6,-4,-2,0,0,2,4,6,8,12),
  \qquad
  (0,2,4,4,6,6,8,8,8,10,10,12).
\end{equation}
These agree entry by entry with the extremal quiver data in \cite[\S5.5]{KRSS}. For $T(3,4)$, the $t$-vector printed in \cite[Eq.~(5.40)]{KRSS} is inconsistent with the diagonal of the adjacent matrix; the diagonal $(0,2,4,4,6)$ agrees with \eqref{eq:node-dictionary} and reproduces the colored polynomial.

The rank-one data determine the vertices, their two linear gradings, and the diagonal of $C$, but not its off-diagonal entries. This is substantive: for $(a,b)=(3,5)$ and $n=2$, a finite search finds quadratic forms that refine the fibers of the generalized multinomial in $N_{a,b;n}$, but none is compatible with GGS self-symmetry in all ranks tested. Hence the full matrix cannot be recovered from the two-variable polynomial $N_{a,b;n}$ alone.

\begin{remark}[The $(3,4)$ master polynomial and the KRSS matrix]\label{rmk:H-34-KRSS}
The explicit proposal for the full master polynomial of $R_{3,4}=\C[\![T^3,T^4]\!]$ contains precisely the quadratic information missing from the two-variable slice $N_{3,4;n}$. Write
\[
  \mathbf a=(a_0,a_1,a_2,a_3,a_4)
\]
for the atom multiplicities indexed, in this order, by
\[
  \varnothing,\quad \{5\},\quad \{1,5\},\quad \{2,5\},\quad G=\{1,2,5\},
\]
so that $a_0+\cdots+a_4=n$. In the variable convention of the present paper, the proposal is
\[
  H_{R_{3,4},n}(q,t,u)
  =
  \sum_{\mathbf a}
  q^{\widehat Q(\mathbf a)}
  t^{T(\mathbf a)}
  u^{U(\mathbf a)}
  \qbinom{n}{\mathbf a}_q,
\]
where
\begin{align*}
  U(\mathbf a)&=a_1+2a_2+2a_3+3a_4,\\
  T(\mathbf a)&=a_1+a_2+2a_3+3a_4,
\end{align*}
and
\begin{align*}
  \widehat Q(\mathbf a)
  ={}&
  a_1^2+a_1a_2+2a_1a_3+2a_1a_4+a_2^2
  +2a_2a_3+3a_2a_4\\
  &\qquad
  +2a_3^2+4a_3a_4+3a_4^2.
\end{align*}

To compare this formula with the KRSS presentation, order the KRSS vertices as
\[
  \varnothing,\quad \{5\},\quad \{1,5\},\quad \{2,5\},\quad G
\]
and put
\[
  \mathbf d=(a_4,a_3,a_2,a_1,a_0).
\]
The reversal is the permutation of the rank-one generators induced by GGS self-symmetry. Using
\[
  \qbinom{n}{\mathbf d}_{q^{-1}}
  =
  q^{-\sum_{i<j}d_id_j}
  \qbinom{n}{\mathbf d}_q,
\]
the quadratic exponent produced by \eqref{eq:master-knot} and \eqref{eq:KRSS-bottom} is
\[
  3n^2-\frac12\mathbf d^{\mathsf T}C\mathbf d-\sum_{i<j}d_id_j.
\]
Comparison with $\widehat Q(\mathbf a)$ therefore gives
\[
  C=
  \begin{pmatrix}
    0&1&2&3&5\\
    1&2&3&3&5\\
    2&3&4&4&5\\
    3&3&4&4&5\\
    5&5&5&5&6
  \end{pmatrix},
\]
which is the matrix $C^{T(3,4)}$ of \cite[Eq.~(5.41)]{KRSS}. The linear exponents $T(\mathbf a)$ and $U(\mathbf a)$ likewise agree with the $\tr$-degree and the $\qtil$-degree under the substitution \eqref{eq:master-knot}. Thus the explicit $(3,4)$ proposal is termwise identical to the KRSS formula for this representative of the quiver matrix.

This does not contradict the nonuniqueness of knots--quivers presentations: the master polynomial determines the displayed matrix only after the atom coordinates and their rank-one identification with the KRSS vertices have been fixed. It nevertheless supplies all off-diagonal entries that are invisible in $N_{3,4;n}$.
\end{remark}

More generally, the master framework predicts that the $R^n$ Quot slice supplies this missing quadratic data. Combining the first line of \eqref{eq:spec-quot-R} with the KRSS formula gives
\begin{equation}\label{eq:quot-quiver}
  \mathcal N^R_{R^n}(q,t)
  =\sum_{\sum_Fd_F=n}
    t^{n\delta+\frac12\sum_Fq_Fd_F}
    q^{\frac12\sum_{F,F'}C_{FF'}d_Fd_{F'}}
    \qbinom{n}{\mathbf d}_{q}.
\end{equation}
For a torus knot, \eqref{eq:node-dictionary} simplifies the exponent of $t$ to $\sum_Fs(F)d_F$. Thus the complete matrix $C$, including its off-diagonal entries, would be encoded by the motives of $\Quot^R_m(R^n)$.

\begin{remark}[Independent checks]\label{rmk:quot-quiver-evidence}
The left-hand side of \eqref{eq:quot-quiver} was computed by Gr\"obner stratification and compared with the KRSS matrices through $t^4$ for $T(3,4)$ at $n=2,3$ and $T(3,5)$ at $n=2$, and through $t^5$ for $T(3,7)$ at $n=2$. Rank three for $T(3,4)$ is the first test for which $2n\delta\neq n^2\delta$. The corresponding numbers of possible ordered atoms are $5^3=125$ and $12^2=144$ for $T(3,4)$ at $n=3$ and $T(3,7)$ at $n=2$, respectively. Every stratum encountered within these ranges was an affine space over $\Z$, so each verified coefficient agrees simultaneously as a point count and as a motivic identity. At $n=1$, the complete rank-one comparison was made for $T(2,3)$, $T(2,5)$, $T(2,7)$, $T(3,4)$, $T(3,5)$, $T(3,7)$, and $T(4,5)$.

For $T(3,4)$ at $n=2$, \eqref{eq:KRSS-bottom} gives $25$ bottom-row generators and agrees in all three gradings with the coefficient of $\alpha^{12}$ in the $121$-generator colored polynomial of \cite[\S4.3]{GGS}. This checks the KRSS presentation directly against GGS and agrees with every independently computed geometric coefficient.
\end{remark}

For the family $T(2,b)$ with $b$ odd, the gap poset is a chain and the complete extremal matrix has the particularly simple form
\begin{equation}\label{eq:C-2b}
  C^{T(2,b)}_{ij}=
  \begin{cases}
    2i,&i=j,\\
    2\max(i,j)-1,&i\neq j,
  \end{cases}
\qquad 0\leq i,j\leq\delta=\frac{b-1}{2}.
\end{equation}
It reproduces \eqref{eq:quad-N} for $b=3,5,7,9,11$ and $n\leq4$, and for $b=3,5,7$ it is the unique symmetric integral matrix in the tested range with the prescribed diagonal. Since $s(F)=2|F|$ for a chain, the perverse grading is a function of the other two in this family; this is the degenerate case in which the complete matrix is easiest to recover.

\begin{problem}\label{prob:C-geometry}
Construct a distinguished extremal matrix $C$ directly from Quot geometry. The diagonal $C_{FF}=2|F|$ and the linear gradings are already visible in the Bia\l ynicki--Birula geometry of $\Quot^R_m(\tl R^n)$. Formula \eqref{eq:quot-quiver} predicts that the off-diagonal entries are encoded by $\Quot^R_m(R^n)$, but it does not explain them geometrically. One expects them to arise from the geometry of $\Fl_G(\mathbf n;n)$ or from a filtration on $E_R(\tl R^n;m)$. Since equivalent quivers need not have the same matrix \cite{KRSS,PanfilStosicSulkowski}, such a construction should select a representative rather than prove uniqueness.
\end{problem}

\subsection{The perverse refinement}\label{sss:perv}

The three unrefined slices lie on the two subtori $\tr=1$ and $\qtil\tc^n=1$. We now describe the grading transverse to the latter.

Choose a family of curves $\mathcal C\to B$ in a smooth surface whose central fiber is the germ under study, such that the relevant relative Hilbert or Quot scheme is proper with smooth total space\footnote{See Proposition~\ref{prop:hilb-smooth}}. Let $\mathsf P_\bullet$ be the resulting perverse Leray filtration. For a mixed Hodge structure $V$, write $W(V;q)\in \Z[q^{1/2}]$ for its virtual weight polynomial, normalized by $W(\Q(-1);q)=q$. For $X=\Quot$ or $\Hilb$, let $X_m=\Quot^R_m(\tl R^n)$ and $\Hilb_m(R_f^{(n)})$, respectively, and define
\begin{equation}\label{eq:perv-series}
  \mathrm P\!\calZ_X(q,t,u)
  =\sum_{m\geq0}\sum_{i\in\Z}
    W\bigl(\gr_i^{\mathsf P}H_c^*(X_m);q\bigr)t^mu^i.
\end{equation}
This definition is virtual: it includes all cohomological degrees with their usual signs and does not assume purity or parity vanishing. Since the perverse filtration is exhaustive, setting $u=1$ gives the ordinary virtual weight series.

Normalize by
\begin{align}
  \mathcal N^{\mathrm p}_{\tl R^n}(q,t,u)
    &\coloneqq(t;q)_n^{b(R)}
      \mathrm P\!\calZ_{\Quot}(q,t,u),\label{eq:perv-normalized-quot}\\
  \cH^{\mathrm p}_{f;n}(q,t,u)
    &\coloneqq(t;qt)_n^{b(R)}
      \mathrm P\!\calZ_{\Hilb}(q,t,u).
      \label{eq:perv-normalized}
\end{align}
For $n=1$ the filtration is the one used in ORS and is canonical by the support theorem of \cite{migliorinishende2013}. For $n\geq2$, independence of the chosen family is part of the geometric prediction.

For an algebraic knot, Conjecture~\ref{conj:colorHOMFLYconj} is equivalently
\begin{align}
  \mathcal N^{\mathrm p}_{\tl R^n}(q,t,u)
    &=t^{n\delta}\Bot\tcP^K_{S^n}
      \bigl(u^{-1/2}q^{-n/2},t^{-1/2},q^{1/2}\bigr),
      \label{eq:perv-quot}\\
  \cH^{\mathrm p}_{f;n}(q,t,u)
    &=t^{n\delta}\Bot\tcP^K_{S^n}
      \bigl(u^{-1/2}(tq)^{-n/2},t^{-1/2},(tq)^{1/2}\bigr).
      \label{eq:perv-hilb}
\end{align}
The second formula is obtained from the first by $q\mapsto tq$; geometrically, the expected identity is therefore
\begin{equation}\label{eq:perv-HQ}
  \cH^{\mathrm p}_{f;n}(q,t,u)
  =\mathcal N^{\mathrm p}_{\tl R^n}(tq,t,u),
\end{equation}
which makes sense for an arbitrary reduced germ, including a multibranch germ.

Inverting the substitution in \eqref{eq:perv-hilb} gives
\begin{equation}\label{eq:perv-dictionary}
  t=\tr^{-2},
  \qquad
  q=(\tr\tc)^2,
  \qquad
  u=(\qtil\tc^n)^{-2}.
\end{equation}
Thus $u$ is precisely the coordinate transverse to \eqref{eq:subtorus}. Unlike any one of the unrefined formulas \eqref{eq:spec-quot-R}--\eqref{eq:spec-hilb}, either \eqref{eq:perv-quot} or \eqref{eq:perv-hilb} determines the full trigraded bottom row.

The following specializations are useful when testing the formula:
\begin{enumerate}
\item At $u=1$, equations \eqref{eq:perv-quot} and \eqref{eq:perv-hilb} reduce to \eqref{eq:spec-quot-Rt} and \eqref{eq:spec-hilb} at the level of virtual weight polynomials.

\item At $t=1$, the Hilbert-side formula and GGS self-symmetry give
  \begin{equation}\label{eq:perv-at-t1}
    \cH^{\mathrm p}_{f;n}(q,1,u)
    =u^{-n\delta}N_{R^n}(q,u).
  \end{equation}
Thus, up to the shift $n\delta$, the perverse grading on the thickened-Hilbert side becomes the colength grading on $\Quot^R(R^n)$.

\item At $n=1$, the generator indexed by $F\in\mathcal J(G)$ has perverse index $s(F)-\delta$ and cohomological degree $2|F|$, in agreement with the ORS description through the perverse filtration on the compactified Jacobian \cite[Proposition~4]{OblomkovRasmussenShende12}.
\end{enumerate}

\begin{remark}[Computational evidence]\label{rmk:perv-evidence}
For each available torus-knot example, the right-hand sides of \eqref{eq:perv-quot}--\eqref{eq:perv-hilb} have integral exponents, symmetric $u$-degree in $[-n\delta,n\delta]$, the correct $u=1$ specialization, the substitution symmetry $q\mapsto tq$, and the specialization \eqref{eq:perv-at-t1}. These checks were made for $T(3,4)$ at $n=1,2,3$, for $T(3,5)$ and $T(3,7)$ at $n=1,2$, for $T(2,3)$ and $T(2,5)$ at $n=2,3$, and for $T(2,7)$, $T(2,9)$, and $T(2,11)$ at $n=2$. They test compatibility with all previously computed slices, but not the existence of the predicted geometric filtration.

The exponents in \eqref{eq:perv-hilb} are strongly constrained. The rank-one ORS grading fixes the exponent of $u$ in $\qtil$ up to overall normalization; compatibility with \eqref{eq:perv-at-t1} rules out the alternative placements tested in every example with $n\geq2$. What remains undetermined by rank one alone is the scale of the perverse index, because rank one cannot distinguish $u$ from $u^n$.
\end{remark}

Because \eqref{eq:perv-series} is virtual, its specialization at $u=1$ is formally well-defined without any purity assumption. Conjecture~\ref{conj:colorHOMFLYconj} predicts more: the underlying cohomology is pure and concentrated in even degrees, and the normalized virtual series has integral weight degrees and the polynomiality required by the master formula. Purity identifies the unnormalized virtual weight series with an ordinary Poincar\'e series, but it does not force positivity after normalization. A compatible affine paving would further explain the motivic polynomiality used in \eqref{eq:spec-hilb} and \eqref{eq:HQ}.

Recall the whole construction relies on a family $\mathcal{C}\to B$ such that the relative Hilbert or Quot scheme has smooth total space. This requirement is not vacuous: on the Hilbert side, smoothness of the total space can always be achieved locally through any prescribed finite colength by adding sufficiently many deformation parameters.

\begin{proposition}\label{prop:hilb-smooth}
Let $g\in\k[\![x,y]\!]$ be nonzero and fix $M\geq1$. Put
\[
  F=g+\sum_{i+j<M}s_{ij}x^iy^j,
  \qquad
  B=\A^{\binom{M+1}{2}},
  \qquad
  \mathcal C=V(F)\subseteq\A^2\times B.
\]
Then $\Hilb^{m}(\mathcal C/B)$ is smooth of dimension $m+\dim B$ for every $m\leq M$.
\end{proposition}
\begin{proof}
By the tangent sequence of \cite[Eq.~(6)]{shende2012severi}, smoothness at $(s,I)$ follows if the map $T_sB\to\O/I$ is surjective. Every colength-$m$ ideal with $m\leq M$ contains $\mathfrak m^m\supseteq\mathfrak m^M$, while $T_sB$ is spanned by the monomials of degree $<M$. Hence $T_sB\to\O/\mathfrak m^M\to\O/I$ is surjective.
\end{proof}
Thus, for every fixed colength range, the local smoothness required for the perverse construction can be arranged. Properness requires a global compactification, and independence of the chosen family remains part of Conjecture~\ref{conj:colorHOMFLYconj}.

For a smooth germ, \eqref{eq:perv-hilb} predicts $\cH^{\mathrm p}_{X;n}=1$: the entire normalized contribution of the punctual Hilbert schemes of the thickened line (a.k.a.~ribbon) $X^n=0$ should lie in one perverse degree. The following calculation verifies the first nontrivial case.

\begin{proposition}[The ribbon at $n=2$]\label{prop:ribbon-n2}
Let $L=\mathcal O_{\mathbb P^1}(1)$, let $S=\operatorname{Tot}(L)$ with tautological
fiber coordinate $x$, and consider the rank-two spectral-curve family
\[
 \mathcal C=\{x^2+a_1x+a_2=0\}\subset S\times B,
 \qquad
 B=H^0(L)\oplus H^0(L^2)\simeq \mathbb A^5.
\]
Its fiber over the origin is the split ribbon
$R=\{x^2=0\}$ with reduced curve $R_{\mathrm{red}}=\mathbb P^1$.
For $0\le m\le 3$, the relative Hilbert scheme
\[
 X_m:=\operatorname{Hilb}^m(\mathcal C/B)
\]
is smooth of dimension $m+5$, and the structure morphism
$\pi_m:X_m\to B$ is projective.  With the perverse indexing used in
\eqref{eq:perv-series}, one has
\[
 \sum_{m\ge0}t^m\sum_{i,j}
 \dim \operatorname{gr}^P_i H^{2j}(\operatorname{Hilb}^m(R),\mathbb Q)
 q^j u^i
 =
 \frac{1}{(1-t/u)(1-tqu)}
 \frac{1}{(1-qt^2/u)(1-q^2t^2u)}+O(t^4).
\]
\end{proposition}

\begin{proof}
We separate the support-theoretic input from the low-degree stalk computation.

\smallskip
\noindent\emph{1. The relevant Hitchin support theorem.}
The base $B$ is exactly the $\mathrm{GL}_2$ Hitchin base for the
$L$-twisted Hitchin system on $\mathbb P^1$:
\[
 \mathcal A_2(L)=H^0(L)\oplus H^0(L^2).
\]
Choose Higgs-bundle degree $e=-1$ (in particular, $e$ is coprime to $2$); for this degree the structure sheaf of every spectral curve gives the canonical stable section of the Hitchin map.  Since
$\deg L=1>2g(\mathbb P^1)-2=-2$, the support theorem of
Chaudouard--Laumon applies.  It states that every simple perverse summand
of the direct image of the stable $\mathrm{GL}_2$ Hitchin map has full support
on $B$; the theorem includes the global nilpotent cone and therefore the
nonreduced spectral curve $x^2=0$ \cite[Th.~9.1]{CL16}; see also the correction
in \cite{CL17}.

In the present numerical situation the spectral curves have arithmetic genus zero:
\[
 \chi(\mathcal O_{C_b})
 =2(1-g)-\frac{2\cdot1}{2}\deg L=1.
\]
Hence the generic compactified Jacobian is a point, and the full-support theorem
reduces the primitive Hitchin contribution to the constant intersection complex
of the base.  On the reduced part of the spectral-curve family, the
Migliorini--Shende--Viviani support formula expresses the relative-Hilbert
complex in terms of the compactified-Jacobian complexes of the connected partial
normalizations.  For the trivial partition, the compactified-Jacobian complex is
precisely this Hitchin complex; Chaudouard--Laumon therefore prevents that
primitive term from acquiring a new support at the ribbon point.  This is the
support input used below.

It is important here that the twist is $L=\mathcal O(1)$.  The canonical
Hitchin system behaves differently: de Cataldo--Heinloth--Migliorini find
proper supports on the reducible spectral loci \cite{dCHM21}.

\smallskip
\noindent\emph{2. Removing the trace and versality of the ribbon deformation.}
Completing the square gives
\[
 z=x+\frac{a_1}{2},\qquad c=a_2-\frac{a_1^2}{4},
 \qquad z^2+c=0.
\]
Thus, after the algebraic change of coordinates
$(a_1,a_2)\mapsto(a_1,c)$,
\[
 B\simeq V_1\times V_2,
 \qquad V_1=H^0(L),\quad V_2=H^0(L^2)\simeq\mathbb A^3,
\]
and the family, as well as every relative Hilbert scheme, is the product of
$V_1$ with the trace-free family over $V_2$.  External product with
$\mathbb Q_{V_1}[2]$ does not change perverse degrees, so from now on we work over
\[
 B_0:=V_2=H^0(\mathcal O_{\mathbb P^1}(2)),
 \qquad \mathcal C_0=\{x^2+c=0\}.
\]

The trace-free family is the miniversal deformation of the embedded ribbon.
Indeed, for the hypersurface $R=\{x^2=0\}\subset S$,
\[
 \mathcal T^1_R
 =\operatorname{coker}\bigl(T_S|_R\xrightarrow{d(x^2)}
 \mathcal O_R(2)\bigr)
 \simeq \mathcal O_{\mathbb P^1}(2),
\]
and the Kodaira--Spencer map
\[
 T_0B_0=H^0(\mathcal O(2))\longrightarrow H^0(\mathcal T^1_R)
\]
is the identity.  Moreover $H^1(\mathcal O(2))=0$.

\smallskip
\noindent\emph{3. Smoothness through length three.}
Projectivity is immediate: $\mathcal C\to\mathbb P^1\times B$ is finite, hence
$\mathcal C\to B$ is projective, and so is its relative Hilbert scheme.

For smoothness we use the standard tangent sequence for relative Hilbert schemes
of planar curves \cite[Eq.~(6)]{shende2012severi}.  The only new point is the central
ribbon.  Let $Z\subset R$ have length $m\le3$, put
$A=H^0(\mathcal O_Z)$ and $J=xA$.  Since $x^2=0$, multiplication by $x$ induces
a surjection
\[
 A/J\twoheadrightarrow J.
\]
Writing $r=\dim(A/J)$ and $s=\dim J$, we have $r+s=m$ and $s\le r$; hence
$r\le3$ and $s\le1$.  The differential of the original five-dimensional base is
\[
 H^0(\mathcal O(1))\oplus H^0(\mathcal O(2))\longrightarrow A,
 \qquad (\dot a_1,\dot a_2)\longmapsto \dot a_1x+\dot a_2.
\]
The second summand surjects onto $A/J$, because $\mathcal O(2)$ separates every
length-$r\le3$ subscheme of $\mathbb P^1$.  The first summand surjects onto $J$,
because $s\le1$ and $\mathcal O(1)$ is globally generated.  Therefore the displayed
map is surjective.  The tangent sequence gives the smoothness of $X_m$ along the
central fiber.  Away from the origin the spectral curve is either smooth or has one
ordinary node, and the family supplies a smoothing parameter for that node; the same
tangent criterion gives smoothness there.  Thus $X_m$ is smooth for $m\le3$.

\smallskip
\noindent\emph{4. Cohomology of the central ribbon.}
Let $C=R_{\mathrm{red}}=\mathbb P^1$ and let
$\mathcal N=(x)\simeq\mathcal O_C(-1)$, so that
$\mathcal O_R=\mathcal O_C\oplus\mathcal N$ and $\mathcal N^2=0$.
For an ideal $\mathcal I\subset\mathcal O_R$, define effective divisors $D,E$ by
\[
 (\mathcal I+\mathcal N)/\mathcal N=\mathcal O_C(-D),
 \qquad
 \mathcal I\cap\mathcal N=\mathcal N(-E).
\]
The ideal condition gives $E\le D$.  If $a=\deg D$ and $b=\deg E$, then
$a+b=m$.  After writing $D=E+F$, the remaining datum is a lift
\[
 \phi\in\operatorname{Hom}
 \bigl(\mathcal O_C(-D),\mathcal N/\mathcal N(-E)\bigr),
\]
a vector space of dimension $b$.  Thus the locus with fixed $(a,b)$ is the
total space of a rank-$b$ vector bundle over
\[
 C^{(b)}\times C^{(a-b)}\simeq\mathbb P^b\times\mathbb P^{a-b}.
\]
For $m\le3$ only $b=0,1$ occur.  The $b=0$ locus is the closed subscheme
$\operatorname{Hilb}^m(C)=\mathbb P^m$, and its complement (when nonempty) is
a line bundle over $\mathbb P^1\times\mathbb P^{m-2}$.  The localization long
exact sequence therefore splits into short exact sequences in even degrees; all
terms are pure Tate.  Hence $H^*(\operatorname{Hilb}^mR)$ is pure and even for
$m\le3$, and
\[
 P_q(\operatorname{Hilb}^mR)
 =P_q(\mathbb P^m)+qP_q(\mathbb P^1)P_q(\mathbb P^{m-2})
 \quad (m=2,3).
\]
Explicitly,
\begin{align*}
 P_q(\operatorname{Hilb}^0R)&=1,\\
 P_q(\operatorname{Hilb}^1R)&=1+q,\\
 P_q(\operatorname{Hilb}^2R)&=1+2q+2q^2,\\
 P_q(\operatorname{Hilb}^3R)&=1+2q+3q^2+2q^3.
\end{align*}

\smallskip
\noindent\emph{5. The decomposition over the reduced locus.}
Put
\[
 \Delta=\{c=\ell^2:\ell\in H^0(\mathcal O(1))\}\subset B_0,
 \qquad U=B_0\setminus\Delta,
 \qquad \Delta^\circ=\Delta\setminus\{0\}.
\]
Thus $\Delta$ is the cone of rank-one binary quadrics, hence
$\Delta\simeq\mathbb A^2/\{\pm1\}$.  Over $U$ the spectral curve is a smooth
rational curve, so $\operatorname{Hilb}^m(C_c)=\mathbb P^m$.  The monodromy on
its cohomology is trivial.  Therefore the full-support contribution to the stalk at
the origin is
\[
 A_m(q,u)=\sum_{j=0}^m q^j u^{2j-m}.
\]

Over $\Delta^\circ$, writing $c=\ell^2$ gives
\[
 C_c=C_+\cup C_-,\qquad C_\pm=\{x=\pm\ell\},
\]
two copies of $\mathbb P^1$ meeting transversely in one point.  The double cover
\[
 H^0(\mathcal O(1))\setminus\{0\}\longrightarrow\Delta^\circ,
 \qquad \ell\longmapsto\ell^2,
\]
has sign local system $\varepsilon$.  The support formula of
Migliorini--Shende--Viviani for a one-node reducible curve
\cite[Th.~5.10 and Ex.~5.11]{MSV21} shows that the discriminant contribution is
the node class tensored with the cohomology of
\[
 Y_{m-1}:=\operatorname{Hilb}^{m-1}(C_+\sqcup C_-)
 =\coprod_{a+b=m-1}\mathbb P^a\times\mathbb P^b.
\]
The codimension-one term also contains the orientation line of the
smoothed node.  In the local model $uv=s$, interchange of the two branches
$u\leftrightarrow v$ reverses the vanishing-cycle orientation; consequently this
line is the sign local system $\varepsilon$.  Equivalently, its generator is the
node class $[C_+]-[C_-]$.  Thus the local system in cohomological degree
$2k+2$ is
\[
 \varepsilon\otimes H^{2k}(Y_{m-1},\mathbb Q)(-1),
\]
and it occurs in perverse degree $2k+1-m$.

At the vertex of the cone, $IC_\Delta(\mathbb Q)$ has a one-dimensional stalk,
whereas $IC_\Delta(\varepsilon)$ has zero stalk.  Indeed, the link is
$\mathbb RP^3$, and its rational cohomology with coefficients in the sign local
system vanishes.  It follows that the contribution at the ribbon is
\[
 D_m(q,u)=q\sum_{k=0}^{m-1}
 \dim H^{2k}(Y_{m-1},\mathbb Q)^{-}\,q^k u^{2k+1-m},
\]
where the superscript $-$ denotes the anti-invariant part for interchanging the two
components.

For $m\le3$ the relevant anti-invariant multiplicities are
\[
\begin{array}{c|ccc}
 r & \dim H^0(Y_r)^- & \dim H^2(Y_r)^- & \dim H^4(Y_r)^-\\ \hline
 0&0&0&0\\
 1&1&1&0\\
 2&1&2&1.
\end{array}
\]
Hence
\begin{align*}
 D_0&=D_1=0,\\
 D_2&=qu^{-1}+q^2u,\\
 D_3&=qu^{-2}+2q^2+q^3u^2.
\end{align*}

The decomposition theorem allows, a priori, additional summands supported at
$\{0\}$.  They do not occur: the stalks already supplied by the full-support and
discriminant terms exhaust the cohomology computed in Step~4.  Explicitly,
\begin{align*}
 A_1+D_1&=u^{-1}+qu,\\
 A_2+D_2&=u^{-2}+q+qu^{-1}+q^2u+q^2u^2,\\
 A_3+D_3&=u^{-3}+qu^{-1}+qu^{-2}+q^2u+2q^2+q^3u^2+q^3u^3.
\end{align*}
At $u=1$ these are respectively
$1+q$, $1+2q+2q^2$, and $1+2q+3q^2+2q^3$.
Since the decomposition is semisimple, any nonzero skyscraper summand would add a
nonzero stalk cohomology group, which is impossible.

Collecting the coefficients for $m=0,1,2,3$ gives exactly
\[
 \frac{1}{(1-t/u)(1-tqu)(1-qt^2/u)(1-q^2t^2u)}+O(t^4),
\]
as asserted.
\end{proof}

\begin{remark}\label{rem:ribbon-m4}
The restriction $m\le3$ in the smoothness statement is necessary.  At a point
$p\in\mathbb P^1$ with local coordinate $y$, consider the length-four subscheme of
the ribbon defined by $(x,y^4)$.  For every tangent vector
$\phi\in\operatorname{Hom}((x,y^4),\mathbb C[[x,y]]/(x,y^4))$ one has
$\phi(x^2)=x\phi(x)=0$, while the base variation
$\dot a_1x+\dot a_2$ has image contained in
$\langle1,y,y^2\rangle$.  Thus the differential of the tautological section has
rank at most three instead of four, so the relative Hilbert scheme is singular at
this point.  In particular, the unqualified assertion that $\operatorname{Hilb}^m
(\mathcal C/B)$ is smooth for every $m$ should not be retained.
\end{remark}

\renewcommand{\thesection}{A}
\section{Appendix: Schubert cells on singular affine Grassmannians}\label{sec:append}

\subsection{Basic properties} In this section we review the framework of affine Grassmannians for reduced curve germs laid out in \cite[Appendix A]{huangjiang2023torsionfree}. We will fix a field $\k$. 
\begin{notation}~\begin{enumerate}
    \item Let $G=\GL_n$ be defined over $\k$, with a fixed split maximal torus $\mathbb{T}$. After fixing a presentation of $G$ by invertible $n\times n$-matrices, we can take $\mathbb{T}$ to be the diagonal torus $\diag(t_1,...,t_n)$.  
    \item Let $(R,\mathfrak{m})$ be a reduced complete local ring over $\k$, with residue field $\k$ and ring of total fractions $K$, and let $\tl R\supseteq R$ be its normalization. For simplicity we will assume that $\tl R=\k[\![T]\!]$.\footnote{ There are finitely many topological generators $f_1,...,f_n\in T\k[\![T]\!]$, so that $R=\k[\![f_1,f_2,...,f_n]\!]\subseteq \k[\![T]\!]$. This follows, for example, from Cohen's structure theorem.} For example, the toric knot germs considered in this paper are of this type.    
    \item Let $\textbf{Alg}_\k$ be the category of $\k$-algebras.
    \item Let $S$ be a $\k$-algebra. We write $S \cotimes R$ \resp $S\cotimes K$ as the completion of $S\otimes R$ \resp $S\otimes K$  with respect to the $1\otimes \mathfrak{m}$-adic topology. For example, $S\cotimes \k[\![T]\!]=S[\![T]\!]$ and $S\cotimes \k(\!(T)\!)=S(\!(T)\!)$.
\end{enumerate}
\end{notation}
\begin{definition}[{\cite[Appendix A]{huangjiang2023torsionfree}}]~\begin{enumerate}
\item Let $S$ be a $\k$-algebra. An \textbf{$S$-family of rank $n$ $R$-lattices} (or informally, a \textbf{rank $n$ $R$-lattice over $S$}) is a finitely generated $S\cotimes R$-submodule $\mathcal{L}\subseteq S\cotimes K^n$, such that $\mathcal{L}\otimes (S\cotimes K)=S\cotimes K^n$ and $S\cotimes K^n/\mathcal{L}$ is  flat over $S$. When $S=\k$, we recover the notion of $R$-lattices. 
\item The \textbf{affine Grassmannian} $\Gr_{G,R}$ is the moduli space of rank $n$ $R$-lattices in $K^n$. More precisely, it is a functor over $\textbf{Alg}_\k$ sending $S$ to $S$-families of rank $n$ $R$-lattices. It is known that $\Gr_{G,R}$  is ind-projective. 

When $R=\widetilde{R}$, $\Gr_{G,\tl R}$ is the classical affine Grassmannian, and will just be denoted by $\Gr_{G}$. 
\item  Let $M\subseteq K^n$ be a rank $n$ $R$-lattice, let $\Gr_{G,R}(M)\subseteq \Gr_{G,R}$ be the closed ind-subscheme parametrizing lattices that are contained in $M$. More precisely, $\Gr_{G,R}(M)$ is the functor over $\textbf{Alg}_\k$ sending $S$ to $S$-families of rank $n$ $R$-lattices that are contained in $S\cotimes M$. The cases $M=R^n$ and $\tl R^n$ will be the ones of major interest.
\item Notation as above. There is a splitting of $\Gr_{G,R}(M)$ into connected components $\Gr_{G,R}(M)=\bigsqcup_{i\geq 0}\Gr_{G,R}^i(M)$, where each $\Gr_{G,R}^i(M)$ parametrizes sublattices of $M$ with colength $i$ in $M$.
    \end{enumerate}
\end{definition}
Let $M\subseteq K^n$ be a rank $n$ $R$-lattice. By  \cite[\S A.1.1]{huangjiang2023torsionfree}, the Quot scheme $\Quot_i(M)$ parametrizing length-$i$ quotients of $M$ is canonically identified with $\Gr_{G,R}^i(M)$ (up to reduced structure). 
\subsection{The extension map} The formulation of the  affine Grassmannians allows us to travel between different base rings $R$. More precisely, we have a natural finite map $\pi: \Spec \tl R\rightarrow \Spec R$. Following \cite[\S A.2]{huangjiang2023torsionfree}, there is a set-theoretic map \begin{equation}
\underline{\vec{\pi}^*}:  \Gr_{G,R} \rightarrow \Gr_{G},
\end{equation}
which sends a point of $\Gr_{G,R}$ (which corresponds to a lattice $L$) to its lattice extension $\tl RL$. The map $\underline{\vec{\pi}^*}$ has more structure than merely a set theoretic map. In fact, it is a \textbf{constructible morphism} in the sense of \cite[Notation A.14]{huangjiang2023torsionfree}. Roughly speaking, this means that there is a stratification $\Gr_{G,R}=\bigsqcup X_\alpha$ into locally closed subschemes, such that $\underline{\vec{\pi}^*}|_{X_\alpha}$ is the underlying set theoretic map of a genuine (ind-)scheme theoretic morphism $ X_\alpha\rightarrow \Gr_{G}$. We will call $\underline{\vec{\pi}^*}$ the \textbf{extension map}. 

By definition, the fiber of a $\k$-point $[\tl L]\in 
\Gr_{G}(\k)$ is the extension fiber $E_R(\tl L)$. The constructible structure of $\underline{\vec{\pi}^*}$ endows $E_R(\tl L)$ with the structure of a constructible set. When $\k=\overline{\k}$, we have a decomposition after passing to Iwahoric Schubert cells $\Gr_{G}=\bigsqcup_{\omega\in X_*(\mathbb{T})}X_\omega^\circ $: 
\begin{equation}\label{eq:constructibleforGrR}
\Gr_{G,R}=\bigsqcup_{\omega\in X_*(\mathbb{T})} X_\omega^\circ\times E_R(\tl R^{\oplus n} ).
\end{equation}
This enables us to think of $\Gr_{G,R}$ as a ``constructible bundle'' over $\Gr_{G}$ with fiber $E_R(\tl R^{\oplus n} )$. 
\subsection{Spherical Schubert cells} Recall that in the classical case (i.e., $R=\k[\![T]\!]$), there is a natural decomposition of 
$\Gr_{G}$ into \(G(R)\)-orbits. Let $\mathbb{T}$ be a maximal torus of $G=\GL_n$. The
orbit  $\Gr_{G,\mu}:=G(R)\mu$ can be understood as all lattices that admit a Smith normal form $\diag (t^{a_1},...,t^{a_n})$. These are called (spherical) Schubert cells; cf. \cite[\S2.1]{zhu2016}. Each Schubert cell is finite-dimensional
and \(G(R)\mu=G(R)\mu'\) if and only if
\(\mu=g(\mu')\) for some permutation \(g\in S_n\). Let $X_*(\mathbb{T})^+\subseteq X_*(\mathbb{T})$ be the subset of elements that have non-negative pairing with positive roots. Then in each orbit $S_n\cdot \mu$ there is a unique dominant weight in $X_*(\mathbb{T})^+$ and
the closure of the stratification is determined by the Bruhat order on $X_*(\mathbb{T})^+$. Let \(P_\mu\subset G\) be the parabolic subgroup associated to $\mu$. Then there is a Bia\l ynicki-Birula morphism 
$$\mathrm{b}_\mu:\Gr_{G,\mu}\rightarrow G/P_\mu$$
sending a lattice to its flag of initials, which is an isomorphism when $\mu$ is minuscule.  Recall that there is a natural $\bG_m$-action on $\k[\![T]\!]$ (which is of weight $d$ on $T^d$). This induces a $\mathbb{T}$-action on $\Gr_{G}$. The map $\mathrm{b}_\mu$ admits a canonical section realizing $G/P_{\mu}$ as the $\mathbb{T}$-contraction locus of $\Gr_{G,\mu}$. 

\subsubsection{Generalization to the  singular setting} The story can be partially generalized to $\Gr_{G,R}$. Recall our assumption that $R\subseteq \tl R=\k[\![T]\!]$, so every element of $R$ can be written as a power series in $T$ with non-negative degree. Define a semigroup $$\Gamma_R:=\{\deg f:f\in R\}\subseteq \mathbb{Z}^{\geq 0}.$$
The set $\Delta_R:=\mathbb{Z}\setminus\Gamma_R$ is called the standard set. We then associate a poset $P_R:=(\Z,\preceq_R)$ to 
$R$ with partial order  defined by 
$$ i\preceq_R j \Leftrightarrow j-i\in \Gamma_R.$$
When the context is clear, $\Gamma_R$ and $P_R$ will simply be abbreviated as $\Gamma$ and $P$. 


\begin{example}
    $\Gamma_R=\mathbb{Z}^{\geq 0}$ if and only if $R=\k[\![T]\!]$. For the toric knot germ $R=\k[\![T^a,T^b]\!]$ where $\gcd(a,b)=1$, $\Gamma_R$ is the semigroup $\{xa+yb: x,y\in \mathbb{Z}^{\geq 0}\}$. 
\end{example}

For $d\in \mathbb{Z}$ and an $R$-lattice $L\in \Gr_{G,R}(\k)$, the $d$-the initial term $in_d(L)$  is a $\k$-point of the Grassmanian $\Gr(T^d\k^n)$ classifying $\k$-subspaces of $\k^n$, defined explicitly as $$in_d(L)= \frac{(T^{d+1})\k[\![T]\!]^n+L\cap (T^d)\k[\![T]\!]^n}{(T^{d+1})\k[\![T]\!]^n} \subseteq T^d\k^n.$$ 
This coincides with the definition given in \S\ref{sub:Torus action and initial terms}. It is clear that $in_d$ extends to a constructible morphism from $\Gr_{G,R}$ to $\Gr(T^d\k^n)$ (i.e. $in_d$ is a morphism up to passing to a stratification of $\Gr_{G,R}$).

In the following we form the singular analogue of the Schubert decomposition: \begin{definition}~\label{def:singularschubert}  \begin{enumerate}
 \item \textbf{(Dimension vector)}  Fix a positive integer $n$, let \begin{equation}
    \Lambda(n)\coloneqq\set{(n_d)_{d\in \Z}:0\leq n_d\leq n, \lim_{d\to \infty}n_d=n, \lim_{d\to -\infty} n_d=0}.
  \end{equation}
 An element $\mathbf{n}\in  \Lambda(n)$ is called a dimension vector. 
    \item \textbf{(Schubert cell)} Let $\mathbf{n}\in  \Lambda(n)$ be a dimension vector. We 
define 
the locally closed subsets:
\begin{equation}
\Gr_{G,R,\mathbf{n}}\coloneqq\set{L\in \Gr_{G,R}: \dim \init_d(L)=n_d}. 
\end{equation}
 This induces a decomposition
\[\Gr_{G,R}=\bigsqcup_{\mathbf{n}\in \Lambda(n)}\Gr_{G,R,\mathbf{n}}.\]
\item \textbf{(Poset flag variety)}
To each $\mathbf{n}=(n_d)_d$ we associate a \textit{poset} flag variety:
\begin{equation}
\mathrm{Fl}_P(\mathbf{n};n)\coloneqq\set{(V_d)_{d\in \Z}:V_d\subeq \k^n, \dim V_d=n_d, V_i\subeq V_j \mbox{ if } i\preceq_R j}.
\end{equation}
\item \textbf{(Bia\l ynicki-Birula morphism)}
There is a natural map 
\begin{equation}
   \mathrm{b}_\mathbf{n}:\Gr_{G,R,\mathbf{n}}\to  \mathrm{Fl}_{P}(\mathbf{n};n).
\end{equation}
 sending $L$ to $(T^{-d}\init_d(L))_d$. 

 \item \textbf{(Standard Schubert cell)}  Let $\Lambda(n)^\circ\subseteq \Lambda(n)$ be the finite subset consisting of dimension vectors $\mathbf{n}=(n_d)_{d\in \mathbb{Z}}$ with the property that 
    \begin{enumerate}
        \item $d<0\Rightarrow n_d=0$, 
        \item  $d\in \Gamma_R \Rightarrow n_d=n$. 
    \end{enumerate}
Schubert cells $\Gr_{G,R,\mathbf{n}}$ with $\mathbf{n}\in \Lambda(n)^\circ$ are called standard Schubert cells. It is clear that we have a decomposition \begin{equation}\label{eq:extfiberinschubert}
    E_R({\tl{R}^{\oplus n}})=\bigsqcup_{\mathbf{n}\in \Lambda(n)^\circ}\Gr_{G,R,\mathbf{n}}.
\end{equation}
\end{enumerate}

\end{definition}
\begin{example}
Let's explain why Definition~\ref{def:singularschubert} serves as a generalization of the classical theory. Let $R=\tl R=\k[\![T]\!]$. Then the set $\Lambda(n)$ plays a role similar to that of $X_*(\mathbb{T})^+$. In fact, there is  an embedding $$\mathbf{V}:X_*(\mathbb{T})^+\hookrightarrow \Lambda(n),$$
sending $\mu= (t^{a_1},...,t^{a_n})\in X_*(\mathbb{T})^+$ to an element  $\mathbf{V}(\mu)\in \Lambda(n)$ where $n_d=\#\{a_i: a_i\leq d\}$. This map identifies $X_*(\mathbb{T})^+$ as the subset of $\Lambda(n)$ that consists of non-decreasing dimension vectors, i.e., those with $n_d\leq n_{d+1}$ for any $d\in \mathbb{Z}$. Then we have\begin{equation}
  \Gr_{G,R,\mathbf{n}}=\begin{cases}
  \Gr_{G,\mu},&\mathbf{n}=\mathbf{V}(\mu), \\
  \varnothing,&\mathbf{n}\notin\mathrm{im}\mathbf{V}.
  \end{cases}
\end{equation}
We recover the classical Schubert decomposition. We will leave the reader to check that the Bia\l ynicki-Birula morphism is also the classical one. 
\end{example}
\begin{question}
    Classically, the closure of a Spherical Schubert cell in $\Gr_{G}$ is determined by the Bruhat order. Does there exist a good story analogous to this for $\Gr_{G,R}$? 
\end{question}


\subsection{The case of toric knots} In this section we fix $R=\k[\![T^a,T^b]\!]$, where $\gcd(a,b)=1$. In this section we translate the results established in the main body of the paper to topological properties of $\Gr_{G,R}$. We note that in this case there is a natural $\bG_m$-action on $R$ restricted from its action on $\k[\![T]\!]$, which induces a $\mathbb{T}$-action on $\Gr_{G,R}$. Note that the $\mathbb{T}$-action maps each Schubert cell to itself.  

Note that from (\ref{eq:constructibleforGrR}), $\Gr_{G,R}$ is constructibly a product $\Gr_{G}\times E_R(\tl{R}^{\oplus n})$; and from (\ref{eq:extfiberinschubert}), $E_R(\tl{R}^{\oplus n})$ is a disjoint union of standard Schubert cells. 

The following theorem is just a restatement of Theorems~\ref{thm:bb} and \ref{thm:poset-flag} in the new language: 

\begin{theorem} Let $\mathbf{n}\in \Lambda(n)^\circ$. The following are true: 
\begin{enumerate}
    \item The Bia\l ynicki-Birula morphism $\mathrm{b}_\mathbf{n}:\Gr_{G,R,\mathbf{n}}\to  \mathrm{Fl}_{P}(\mathbf{n};n)$ is an affine bundle that admits a canonical section realizing $ \mathrm{Fl}_{P}(\mathbf{n};n)$ as the $\mathbb{T}$-contraction locus of $\Gr_{G,R,\mathbf{n}}$. 
    \item $\mathrm{Fl}_{P}(\mathbf{n};n)$ is an iterated Grassmannian bundle over a point.
\end{enumerate}
\end{theorem}

\bibliographystyle{alpha}
\bibliography{bibliography}

\end{document}